\documentclass[11pt]{article}
\usepackage{amsmath,amssymb,amsthm,eucal, slashed}
\usepackage{color}
\usepackage[all]{xy}

\title{\vspace*{-1pc} Untwisting twisted spectral triples}
\author{Magnus Goffeng*, Bram Mesland\dag, Adam Rennie\ddag
\thanks{email: 
\texttt{goffeng@chalmers.se}, \texttt{b.mesland@math.leidenuniv.nl}, \texttt{renniea@uow.edu.au}
}
\\[3pt]
${}^*$ Department of Mathematical Sciences,\\ 
Chalmers University of Technology and
University of Gothenburg,\\ 
Gothenburg, Sweden\\[3pt]
\dag Mathematical Institute, Leiden University, Leiden, Netherlands
\\[3pt]
\ddag School of Mathematics and Applied Statistics, University of Wollongong\\
Wollongong, Australia\\
}

\def\XXint#1#2#3{{\setbox0=\hbox{$#1{#2#3}{\int}$}
    \vcenter{\hbox{$#2#3$}}\kern-.5\wd0}}

\usepackage{bbm}

\advance\topmargin by -\headheight
\advance\topmargin by -\headsep
\evensidemargin=\oddsidemargin
\makeatletter
\def\section{\@startsection{section}{1}{\z@}{-3.5ex plus -1ex minus
  -.2ex}{2.3ex plus .2ex}{\large\bf}}
\def\subsection{\@startsection{subsection}{2}{\z@}{-3.25ex plus -1ex
  minus -.2ex}{1.5ex plus .2ex}{\normalsize\bf}}
\makeatother

\numberwithin{equation}{section} %% needs `amsmath' package

\theoremstyle{plain} %% needs `amsmath' package
\newtheorem{thm}{Theorem}[section]
\newtheorem{thm*}{Theorem}
\newtheorem{lemma}[thm]{Lemma}
\newtheorem{prop}[thm]{Proposition}
\newtheorem{corl}[thm]{Corollary}

\theoremstyle{definition} %% needs `amsmath' package
\newtheorem{defn}[thm]{Definition}
\theoremstyle{remark} %% needs `amsmath' package
\newtheorem{rmk}[thm]{Remark}
\newtheorem{ex}[thm]{Example}

\DeclareMathOperator{\coker}{coker} %% cokernel
\DeclareMathOperator{\Dom}{Dom}   %% domain of an operator
\DeclareMathOperator{\End}{End}   %% endomorphism algebra
\DeclareMathOperator{\Tr}{Tr}     %% operator trace
\newcommand{\Sd}{\textnormal{Sd}}

\newcommand{\rd}{\mathrm{d}}

\newcommand{\sgnlog}{\textnormal{sgnlog}}
\newcommand{\A}{\mathcal{A}}  %% an algebra
\newcommand{\B}{\mathcal{B}}  %% another algebra
\newcommand{\C}{\mathbb{C}}   %% complex numbers
\newcommand{\D}{D}  %% a selfadjoint operator
\renewcommand{\d}{\mathrm{d}} %% \d a = [\D,a]
\renewcommand{\H}{H}  %% a Hilbert space
\newcommand{\K}{\mathbb{K}}  %% compact endomorphisms
\newcommand{\Lip}{\textnormal{Lip}}
\newcommand{\N}{\mathbb{N}}   %% natural numbers
\newcommand{\sgn}{\textnormal{sgn}}
\newcommand{\R}{\mathbb{R}}   %% real numbers
\newcommand{\Z}{\mathbb{Z}}   %% integers
\newcommand{\alg}{\textnormal{alg}}

\newcommand{\stroke}{\mathbin|}     %% (for `\pair' and such)
\def\pairL_#1(#2|#3){{}_{#1}(#2\stroke#3)} %% hermitian pairing _B(s|t)
\def\pairR(#1|#2)_#3{(#1\stroke#2)_{#3}} %% hermitian pairing (s|t)_A
\def\scal<#1|#2>{\langle#1\stroke#2\rangle} %% scalar product <y|z>

\newbox\ncintdbox \newbox\ncinttbox %% noncommutative integral symbols
	\setbox0=\hbox{$-$}
	\setbox2=\hbox{$\displaystyle\int$}
	\setbox\ncintdbox=\hbox{\rlap{\hbox
		to \wd2{\hskip-.125em \box2\relax\hfil}}\box0\kern.1em}
	\setbox0=\hbox{$\vcenter{\hrule width 4pt}$}
	\setbox2=\hbox{$\textstyle\int$}
	\setbox\ncinttbox=\hbox{\rlap{\hbox
		to \wd2{\hskip-.175em \box2\relax\hfil}}\box0\kern.1em}

\begin{document}

\maketitle

\vspace{-2pc}

\begin{abstract}
We examine the 
index data associated to twisted spectral
triples and higher order spectral triples. 
In particular, we show that a Lipschitz regular twisted spectral triple 
can always be ``logarithmically dampened'' through functional calculus, to obtain an ordinary (i.e. untwisted) spectral triple. 
The same procedure turns higher order spectral triples
into spectral triples. 
We provide examples of highly regular twisted spectral triples 
with non-trivial index data for which Moscovici's ansatz for 
a twisted local index formula is identically zero.
\end{abstract}

\tableofcontents

\parskip=6pt
\parindent=0pt

\section*{Introduction}

\label{sec:intro}
The original definition of spectral triple \cite{Connes,C4,CM}
is motivated by the index theory of first order 
elliptic operators on manifolds. The general 
non-commutative and bivariant
definition of unbounded Kasparov module was formalised in \cite{BJ},
and has been shown to capture a wide variety of 
geometric-dynamic examples. Despite this success, 
there are numerous
examples which motivate more general definitions. 

In
this paper we consider index theoretic 
questions about twisted spectral triples 
and higher order spectral triples. Briefly, twisted spectral triples have 
bounded twisted commutators relative to an auxiliary algebra automorphism, 
while higher order spectral triples are analogues of elliptic operators 
of any order $>0$ on a manifold.

Developing a consistent non-commutative geometry in a way 
that is compatible with index theory turns out to be 
hard, even for the most basic dynamical 
examples (see for instance \cite{DGMW, GM, MS}).
Since index theory, and the underlying machinery of 
$KK$-theory, is our only tool for noncommutative 
algebraic topology, discussing geometry in a manner directly relatable 
to index theory is imperative. 

An appropriate analogy is that the geometric meaning 
of differential forms in calculus can be directly related 
to topological structures via
the close relationship between de Rham's differential 
topology and singular or \v{C}ech cohomology. While spectral
triples (or more generally unbounded Kasparov modules)
provide a geometric object, to obtain a connection to 
topology (via index theory) we need to know 
that the bounded transform
$$
(\A,X_B,\D)\mapsto (A,X_B,F_\D:=\D(1+\D^2)^{-1/2})
$$
yields a Kasparov module and so a (preferably non-zero) $KK$-class. Under suitable conditions the bounded transform
of
both twisted and higher order variations of the notion of 
unbounded Kasparov module
yields a Kasparov module.

Our first general result states that functional calculus using the $C^{1}$-function 
$$\sgnlog:\mathbb{R}\to\mathbb{R},\quad x\mapsto \sgn(x)\log(1+|x|),$$ 
can be used to turn both 
twisted and higher order spectral triples into ordinary spectral triples without changing the $K$-homology class (for details on the twisted case, see Theorem \ref{logdampoftwist} below). This 
logarithmic dampening has been used before in specific examples \cite{BM,DGMW,GM,GRU,MS, Pierrot}, and in Section 1 of the present paper we formalise the procedure.

Thus, if a twisted or higher order spectral triple encodes non-trivial index theoretic
data, then the same information can be recovered from 
an ordinary spectral triple, constructed through a 
well-defined procedure, albeit possibly less geometric 
than the original twisted spectral triple. 
Logarithmic dampening will for instance transform an elliptic differential operator 
to a pseudodifferential operator.
For large portions of the paper, we work in the bivariant 
context, proving our general statements for unbounded 
Kasparov modules. 

The motivation for introducing twisted spectral triples comes from situations where twisted commutators
with a natural geometric differential operator are bounded whereas ordinary commutators are not \cite{CoM}. 
This observation suggests that we seek computationally tractable representatives of $K$-homology classes using twisted commutators. 
Early successes of this philosophy concern conformal diffeomorphisms on manifolds \cite{CM, PW, PW2}, a class identified by Connes as tractable. Additionally, twists are well-known to improve 
`dimension drop' problems
in the cyclic homology of $q$-deformed 
algebras, \cite{HK1,HK2}.

Another widely held hope is that while many $C^*$-algebras (such as those that arise in examples coming from dynamical systems) do not admit \emph{finitely summable} ordinary spectral triples, by \cite{Co3}, there could be finitely summable twisted spectral triples. One reason for the interest in finite 
summability is the possibility of producing
a computable cyclic cocycle formula for the index pairing.

An
ansatz for computing such (twisted) index 
pairings was proposed
by Moscovici \cite{twistedmoscovici}, and shown to 
work for twists coming from ``scaling automorphisms''
\cite{twistedmoscovici, PW, PW2}. 
Moscovici's proposed formula was an analogue of
the known
local index formulae, adapted for regular twisted spectral triples, 
and computed in 
terms of residues of $\zeta$-functions.

Our second main result concerns the existence of twisted spectral triples that pair non-trivially with $K$-theory, but for which all twisted higher residue cochains appearing in Moscovici's ansatz vanish. 
The proof consists of examples where all $\zeta$-functions coming from operators containing a twisted commutator as a factor are entire functions. This leaves little hope for a twisted higher residue cochain to compute the index pairing in 
general. For details see Theorem \ref{counteexmosc} (on page \pageref{counteexmosc}). 

We produce
several such examples. The first example arises by introducing 
a twist on the usual spectral triple on the circle. 
Next, we extend this twisted spectral triple on the circle to the crossed product by 
a group of M\"obius transformations, thereby including examples of purely infinite $C^{*}$-algebras. The last class of
examples comes from the action of the free group on the boundary of its 
Cayley graph. These are again purely infinite (and more involved), 
but the main idea in each of the examples is the same.

Our two main results indicate that the analytic aspect of index theory for finitely 
summable twisted spectral triples is quite involved. The proofs of these results
indicate that the appropriate index theory is in 
fact closely related to the index theory for 
$\mathrm{Li}_1$-summable spectral triples 
(recently studied in \cite{GRU}). 

\subsection*{Main results}

Let us discuss the main results in a bit more detail. It has 
been known since their introduction \cite{CoM} that a 
twisted spectral triple $(\A,\H,\D,\sigma)$ defines a 
$K$-homology class only under some extra 
regularity assumption. Lipschitz regularity, namely that for 
all $a\in\A$ the twisted commutator
$$
[|D|,a]_{\sigma}:=|D|a-\sigma(a)|D|
$$ 
extends to a bounded operator, is a sufficient 
condition\footnote{Other conditions are known: in \cite{Ka}, an elaborate set of conditions (adapted to Kasparov products) guaranteeing topological content is used.}. See more below in Proposition \ref{bddtwistedtransform} on page \pageref{bddtwistedtransform}. From the perspective of the local index formula, Lipschitz regularity is a necessary requirement and we adopt it here. We will in the first half of the paper work with \emph{weakly} twisted spectral triples, meaning that the homomorphism $\sigma$ need not preserve the algebra (for more details, see Remark \ref{weaklytwisksdl} on page \pageref{weaklytwisksdl}). The first of our main results 
untwists twisted spectral triples.

\begin{thm*}
\label{logdampoftwist}
Let $(\A,\H,\D,\sigma)$ be a Lipschitz regular 
weakly twisted spectral triple such that $D$ has compact resolvent. Define the self-adjoint operator 
$$
\D_{\log}:=\sgnlog(D)=\D|\D|^{-1}\log(1+|\D|).
$$
Then $(\A,\H,\D_{\log})$ is a spectral triple defining the same class in $K$-homology as $(\A,\H,\D,\sigma)$. 
The weakly twisted spectral triple 
$(\A,\H,\D,\sigma)$ is finitely summable if and only if the spectral triple $(\A,\H,\D_{\log})$ 
is $\mathrm{Li}_1$-summable.
\end{thm*}

The reader should note that $\D_{\log}$ is well-defined also for non-invertible $D$ as it is defined from functional calculus with the $C^{1}$-function $x|x|^{-1}\log(1+|x|)$. This result appears as Theorem \ref{logdampofregtwisted} 
(see page \pageref{logdampofregtwisted}), where it is proven 
in the larger generality of compact Lipschitz regular 
weakly twisted Kasparov modules. The summability statement is found in Proposition \ref{sumoflogdamp} (see page \pageref{sumoflogdamp}).
In Theorem \ref{logtrans} (see page \pageref{logtrans}) we 
prove that the same logarithmic transform turns a higher 
order Kasparov module into an ordinary Kasparov module.
For technical reasons, we restrict to the case of compact resolvent 
when considering twisted Kasparov modules 
but our results in the higher order case are proved in general.

The index theoretic content of Theorem \ref{logdampoftwist} 
is as follows. If a $K$-homology class is represented 
by a weakly twisted spectral triple, then the $K$-homology class is also represented by an ordinary, 
i.e. untwisted, spectral triple, that can be constructed through a 
definite procedure.

We illustrate both higher order and twisted spectral triples by means of constructing various $K$-homologically non-trivial exotic spectral triples
on the crossed product algebra arising from a non-isometric diffeomorphism on the circle. As a simple application, we show that the boundary map in the Pimsner-Voiculescu sequence can be computed at an unbounded level (under mild assumptions) using a combination of logarithmic dampening and higher order spectral triples. These results can be found in Subsection \ref{subsec:bdd-trans}. This is an example of a setting in which twisted spectral triples do not provide a solution.

There is a left inverse (modulo bounded perturbations) to untwisting 
given by exponentiating an ordinary spectral triple satisfying further assumptions.
Let $(\A,\H,\D)$ be a 
spectral triple and write $F:=\D|\D|^{-1}$. 
If  for all $a\in \A$, we have
\label{expdefnein}
\[
a\Dom(\mathrm{e}^{|\D|})\subseteq 
\Dom(\mathrm{e}^{|\D|}), \,\textnormal{and}\quad 
[F,a]\H\subseteq \Dom(\mathrm{e}^{|\D|}),
\]
then $(\A,\H,F\mathrm{e}^{|D|},\sigma)$ is a 
weakly twisted spectral triple. The twist $\sigma$ used in the exponentiation procedure is defined as
$\sigma(a):=\mathrm{e}^{|\D|}a\mathrm{e}^{-|\D|}$. 
This way of twisting untwisted spectral triples preserves 
several pathological properties, and can be exploited
to construct twisted spectral triples with exotic features. 

In the second half of the paper we use 
the exponentiation process to construct 
twisted 
spectral triples that are regular, finitely summable, have 
finite discrete dimension spectrum, and pair non-trivially with $K$-theory. 
For the examples we construct, the cochain given by
Moscovici's ansatz for a local index 
cocycle \cite{twistedmoscovici} vanishes identically, and therefore fails to compute the index pairing. 
These examples thus
show that Moscovici's ansatz does not compute the index pairing in general.
The details of the construction make clear that any
index formula based on residues of traces is implausible for 
examples similar to ours.

The key to building such examples is the existence of ordinary 
$\textnormal{Li}_1$-summable spectral triples
$(\A,H,D)$ such that the commutators $[F,a]$ are finite rank or smoothing for 
all $a\in\A$. Then the conditions for exponentiation can be met,
and the exponentiated triple $(\A,H,F\mathrm{e}^{|D|},\sigma)$ is a finitely summable twisted spectral triple.

The fact that the commutators $[F,a]$ are smoothing or of finite rank  is also used to show that the
functionals appearing in Moscovici's ansatz (see Section \ref{sec:LIF} for 
detailed notation)
$$
a_0,a_1,\dots,a_m\mapsto \Tr(a_0[F\mathrm{e}^{|D|},a_1]_\sigma^{(k_1)}\cdots[F\mathrm{e}^{|D|},a_m]_\sigma^{(k_m)}
\mathrm{e}^{-(2|k|+m+2z)|D|})
$$
extend to entire functions of $z$. The cochain $(\phi_m)_{m\geq1,{\rm odd}}$ appearing in Moscovici's ansatz is a linear combination of residues of such expressions (for details, see Definition \ref{defnofloclalass} on page \pageref{defnofloclalass}).

The simplest example is given by the standard spectral triple $\left(C^{\infty}(S^1),L^2(S^1),-i\frac{\d}{\d x}\right)$, which can be turned into a twisted spectral triple 
on the circle. Here we parametrise the circle $S^1=\{z\in \C: |z|=1\}$ by $x\mapsto \mathrm{e}^{ix}$ and $-i\frac{\d}{\d x}$ is the ordinary Dirac operator on $S^1$.
An example on $S^1$ might seem too nice and forcing a twist on it as somewhat artificial. 
However, we show
that the same method extends to crossed products by groups
$\Gamma$  of conformal
diffeomorphisms
$$
\left(C^{\infty}(S^1)\rtimes^{{\rm alg}}\Gamma,L^2(S^1),-i\frac{\d}{\d x},\sigma\right)
$$
with appropriate twist $\sigma$ defined from $-i\frac{\d}{\d x}$. For Fuchsian groups of the first kind the crossed product is
purely infinite, and in this case finitely summable 
spectral triples do not exist.

Ponge and Wang 
studied the case of twisted spectral triples arising from 
groups of conformal diffeomorphisms in great detail \cite{PW,PW2}.
While they prove a local index formula, this is achieved by 
reducing to the case of trivial twist and using the local index formula
for untwisted spectral triples.
They noted in \cite[Remark 4.1]{PW} that the question of whether Moscovici's ansatz computes the index for the standard conformal twist was open. 
Our example uses a twisting automorphism 
that is different from the conformal twist, leaving the question of Ponge-Wang unanswered. What our example does show is
that Moscovici's ansatz does not compute the index pairing for all
twisted spectral triples coming from conformal diffeomorphisms. 

Our final set of examples are purely infinite Cuntz-Krieger algebras
arising from the action of the free group $\mathbb{F}_d$ on its 
boundary $\partial\mathbb{F}_d$, $d>1$. The spectral triples we start
with here were first studied in \cite{GM}, where all the necessary index pairings
and analytic behaviour were determined. It is nevertheless a lengthy 
calculation to prove the discrete dimension spectrum condition required. We postpone the more onerous details in this example to the appendix.

The above counterexamples can be summarised as our second main result.

\begin{thm*}
\label{counteexmosc}
There exist unital odd twisted spectral triples $(\A,H,Fe^{|D|},\textnormal{Ad}_{e^{|D|}})$
that are  
regular, finitely summable, have discrete dimension spectrum, and
pair non-trivially with $K_1(\A)$ such that 
the cochain $(\phi_m)_{m\geq1,{\rm odd}}$ 
provided by Moscovici's ansatz \cite{twistedmoscovici} 
for a twisted local index formula is identically zero.  
\end{thm*}

The precise statements for the three types of counterexample  
appear as Theorems \ref{thm:the-good}, 
\ref{badcounterexthm}, \ref{maintwo}.
It should be noted that all three examples considered share the property that $[F,a]$ is smoothing. 
This property is (apparently) quite rare. The circle is the only connected closed manifold admitting a metric for which the phase of the Dirac operator commutes with $C^\infty$-functions up to smoothing operators. 
Cuntz-Krieger algebras share several geometric features with the circle, and in fact, the circle and its crossed product by a Fuchsian group of the first kind are 
Cuntz-Krieger algebras (see \cite{adelaroche}).

\subsection*{Acknowledgments}

The authors would like to thank Heath Emerson and Bogdan Nica for several interesting discussions over the years that inspired the work in this paper, and Bob Yuncken 
for helpful comments.
The authors acknowledge the support of 
the Erwin Schr\"{o}dinger Institute where some of this work was conducted. 
A. R. and B. M. thank the Gothenburg Centre for Advanced Studies in Science and Technology 
for funding and the University of Gothenburg and Chalmers University of Technology 
for their hospitality in November 2017 when this project took shape.
M. G. was supported by the Swedish Research Council Grant 2018-03509. 
B. M. thanks the Hausdorff Center for Mathematics and the Max Planck Institute for Mathematics in Bonn and the University of Wollongong for their hospitality and support. Finally, the authors are grateful to the referee for a close reading which resulted in numerous improvements.

\section{Kasparov modules and index theory}
\label{sec:pre}

In this section we recall the definition of two 
non-standard notions of unbounded
Kasparov module, and hence of spectral triple. The 
first notion is that of twisted Kasparov module 
and is based on the behaviour 
of conformal diffeomorphisms on manifolds. 
The second notion is that of higher order 
Kasparov module and is inspired by 
elliptic pseudodifferential operators on manifolds of 
order $\geq 1$.  

Both variations of Kasparov modules arise when one attempts to construct 
$K$-homologically non-trivial spectral triples for crossed products by 
non-measure preserving dynamics on metric measure spaces \cite{BMR, CM,  GM, GMR, MS,PW, PW2}.

We first recall the fact that the bounded transform still produces 
bounded Kasparov modules in this generality. For twisted 
Kasparov modules we need to assume Lipschitz regularity, a mild condition
used by several authors. Subsequently we 
prove that, using the logarithm function, these exotic Kasparov 
modules can be turned into ordinary Kasparov modules via the 
functional calculus, without changing the $KK$-class.

\subsection{Twisted Kasparov modules}

Let $\A\subset A$ be a dense $*$-subalgebra of a 
unital $C^*$-algebra.
Let $\sigma:\,\A\to\A$ be an algebra homomorphism that is compatible with the involution on $A$ in the sense that $\sigma(a)^*=\sigma^{-1}(a^*)$ (called a {\em regular} automorphism).
For a $B$-Hilbert $C^*$-module $X_B$, we let $\End^*_B(X_B)$ denote the $C^*$-algebra 
of $B$-linear adjointable operators on $X_B$ and $\K_B(X_B)\subseteq \End^*_B(X_B)$ the  
$C^*$-subalgebra of $B$-compact operators.

\begin{defn}
\label{twisteddefn}
A twisted Kasparov module $(\A,X_B,D,\sigma)$ is given by a
$*$-algebra $\A$ with a regular automorphism $\sigma$, represented on a $B$-Hilbert $C^*$-module $X_B$ via the
restriction of a $*$-homomorphism
$$
\pi:A\to\End^{*}_{B}(X_B).
$$
The operator $\D$ is a densely defined regular 
self-adjoint operator
$$
D:\Dom(D)\subset X_B\to X_B
$$
such that for all $a\in\A$ the following conditions are satisfied:
\begin{enumerate}
\item  $\pi(a) \Dom(D)\subset\Dom(D)$ and
the densely defined twisted commutator 
$$[\D,\pi(a)]_\sigma:=D \pi(a)-\pi(\sigma(a))D,$$ 
is
bounded (and so extends to a bounded adjointable operator on all of $X_B$ by continuity);
\item the operator $\pi(a)(1+D^2)^{-1/2}$ is a
$B$-compact operator.
\end{enumerate}

If we in addition have prescribed an operator $\gamma\in \End^{*}_{B}(X_B)$ with
$\gamma=\gamma^*$, $\gamma^2=1$, $ D\gamma+\gamma D=0$, and for
all $a\in\A$ $\gamma\pi(a)=\pi(a)\gamma$, we call the spectral
triple {\bf even} or {\bf graded}. Otherwise we say that it is {\bf odd} or {\bf ungraded}.

If $\A$ is unital and $\pi(1)=1$ we say that $(\A,X_B,D,\sigma)$ is 
{\bf unital}. If $D$ has $B$-compact resolvent (i.e. $(D\pm i)^{-1}\in \K_B(X_B)$) 
we say that $(\A,X_B,D,\sigma)$ is {\bf compact}.
\end{defn}

\begin{rmk}
\label{droppingpi}
We will nearly always dispense with the
representation $\pi$, treating $\A$ as a subalgebra of $\End_B^*(X_B)$.
In general, $\pi$ need not be faithful but for our purposes this issue 
does not play a role and will be disregarded.
\end{rmk}

\begin{rmk}
\label{weaklytwisksdl}
It is not always the case that $\sigma$ is an automorphism of $\A$. For instance, Kaad \cite{Ka} and Matassa-Yuncken \cite{MY} consider situations where $\sigma$ fails to preserve $\A$. We consider the following situation. Let $\A\subset \End^{*}_{B}(X_B)$ be a $*$-subalgebra of bounded operators and $\sigma$ a partially defined automorphism of $\End^{*}_{B}(X_B)$ such that $\A\subseteq \Dom(\sigma)$. Here
$\sigma:\,\Dom(\sigma)\to \End^*_B(X_B)$
is multiplicative and regular, where
$\Dom(\sigma)$ is some (possibly proper) $*$-subalgebra
of $\End^*_B(X_B)$ on which  $\sigma$ is defined.
Under these assumptions, we say that $(\A,X_B,D,\sigma)$ is a {\bf weakly twisted Kasparov module} if the remaining conditions of Definition \ref{twisteddefn} are satisfied. The next result follows immediately from the 
bounded twisted commutator condition. 
\end{rmk}

\begin{prop}
\label{saturateprops}
Let $(\A,X_B,D,\sigma)$ be a weakly twisted Kasparov module with $\A$ contained
in the domain $\cap_{k\in \Z}\Dom(\sigma^k)$. We define $\A_\sigma$ as the $*$-algebra generated by $\cup_{k\in \Z}\sigma^k(\A)$ and extend $\sigma$ to an automorphism of $\A_\sigma$ by multiplicativity. Assume that for all $a\in \A$,
\begin{itemize}
\item $\sigma^k(a)\Dom(D)\subseteq \Dom(D)$ and $[D,\sigma^k(a)]_\sigma$ is bounded,
\item $\sigma^k(a)(1+D^2)^{-1/2}$ is $B$-compact for all $k\in \Z$. 
\end{itemize} 
Then $(\A_\sigma,X_B,D,\sigma)$ is a twisted Kasparov module. 
\end{prop}

We call  $(\A_\sigma,X_B,D,\sigma)$ the {\bf saturation} 
of the weakly twisted Kasparov module $(\A,X_B,\D,\sigma)$. 
Whenever referring to the saturation of a weakly twisted Kasparov 
module $(\A,X_B,\D,\sigma)$, we tacitly assume that 
$\A\subseteq \cap_{k\in \Z}\Dom(\sigma^k)$ and that all the assumptions of Proposition \ref{saturateprops} hold.

It is currently unclear if  twisted Kasparov modules carry index 
theoretic information in all generality. They do if they
satisfy the following mild condition (\cite{CM}, see Proposition 
\ref{bddtwistedtransform} below).

\begin{defn}[\cite{CoM}]
A weakly twisted Kasparov module $(\A,X_B,\D,\sigma)$
satisfies the twisted Lipschitz 
condition if 
$$
[|D|,a]_{\sigma}=|\D|a-\sigma(a)|\D|,\quad \mbox{is bounded for all }a\in\A.
$$
In short, we say that $(\A,X_B,D,\sigma)$ is {\bf  Lipschitz regular}.
\end{defn}

\begin{ex}[Non-isometric diffeomorphism on the circle]
\label{circleexdef}
We describe the simplest case of a twisted spectral triple associated to a conformal diffeomorphism on a manifold (see \cite{CM, twistedmoscovici, PW, PW2}). 
Denote by $S^{1}\subset\mathbb{C}$ the unit circle and let $\gamma:S^{1}\to S^{1}$ be a diffeomorphism, which generates an action of $\mathbb{Z}$ on $S^{1}$. A fact that we implicitly use in this example is that all diffeomorphisms of $S^1$ act conformally.
Write $|\gamma'(z)|$ for the pointwise absolute value of the derivative $\frac{\d\gamma}{\d x}(z)$. Consider the unitary
representation of $\mathbb{Z}$ generated by the operator
\begin{equation}
\label{unirep}
\pi(\gamma)\in \B(L^{2}(S^{1})),\quad \pi(\gamma)\phi(z):=|\gamma'(z)|^{-\frac{1}{2}}\phi(\gamma(z)).
\end{equation}
Define $\sigma:C^{\infty}(S^{1})\rtimes^{\alg}\mathbb{Z}\to C^{\infty}(S^{1})\rtimes^{\alg}\mathbb{Z}$ by $\sigma(f\gamma^{n})(z)=f|\gamma'|^{n}\gamma^{n}$. Then one
easily checks that the data
$(C^{\infty}(S^{1})\rtimes^{\alg}\mathbb{Z}, L^{2}(S^{1}), i\frac{\d}{\d x},\sigma)$ 
defines a Lipschitz regular twisted spectral triple. Setting $\iota:C(S^{1})\to C(S^{1})\rtimes\mathbb{Z}$ to be the inclusion, it holds that 
\[
\iota ^{*} \left[\left(C^{\infty}(S^{1})\rtimes^{\alg}\mathbb{Z}, L^{2}(S^{1}), i\frac{\d}{\d x},\sigma\right)\right]
=\left[\left(C^{\infty}(S^{1}), L^{2}(S^{1}), i\frac{\d}{\d x}\right)\right]\in K^{1}(C(S^{1})).
\]
This proves that 
\[
0\neq 
\left[\left(C^{\infty}(S^{1})\rtimes^{\alg}\mathbb{Z}, L^{2}(S^{1}), i\frac{\d}{\d x},\sigma\right)\right]
\in K^{1}(C(S^{1})\rtimes\mathbb{Z}).
\]
By adding to $i\frac{\d}{\d x}$ the projection onto the constant functions in $L^{2}(S^{1})$ we obtain an invertible self-adjoint operator $D$ in $L^{2}(S^{1})$ for which $(C^{\infty}(S^{1})\rtimes^{\alg}\mathbb{Z}, L^{2}(S^{1}),D,\sigma)$ is a twisted spectral triple representing the same $K$-homology class. Then applying \cite[Proposition 3.3]{twistedmoscovici} we find that the difference
\[
\sigma(a)-|D|a|D|^{-1}\in\mathcal{L}^{1,\infty}(L^{2}(S^{1})),
\]
and $(\sigma(a)-|D|a|D|^{-1})D$ extends to a bounded operator. Here $\mathcal{L}^{1,\infty}(L^{2}(S^{1}))$ denotes the weak trace ideal inside the algebra of bounded operators on $L^2(S^1)$ and consists of operators $T$ such that the $k$-th largest eigenvalue of $\sqrt{T^*T}$ is $O(k^{-1})$. Writing $\tilde{\sigma}(a):=|D|a|D|^{-1}$, we find 
that $(C^{\infty}(S^{1})\rtimes^{\alg}\mathbb{Z}, L^{2}(S^{1}),D,\tilde{\sigma})$ is a weakly twisted spectral triple, again representing the same $K$-homology class.
\end{ex}

\subsection{Invertible amplifications}

In order to establish the link with $KK$-theory for Lipschitz regular twisted Kasparov modules, we first need a technical construction to rid possible problems related to (non)invertibility of the operator $\D$. 
\begin{defn}
\label{defn:sign}
Let $D$ be a self-adjoint operator on a Hilbert space $H$. We define 
the operator $\sgn(D):=D|D|^{-1}$  by setting $|D|^{-1}$ to be $0$ on $\ker D$. 
\end{defn}
On Hilbert $C^{*}$-modules, $\ker D$ need not be a complemented submodule and we need to pass to an \emph{invertible amplification}, which we describe in this section. The proof of the following lemma is analogous to that of \cite[Lemma 3.6]{DGMBor}.

\begin{lemma}
\label{invamp}
Let $X_B$ be a Hilbert $C^{*}$-module over $B$ and $D:\Dom(D)\to X_B$ a self-adjoint regular operator with $(D\pm i)^{-1}\in \mathbb{K}_B(X_B)$. Set $\tilde{D}:=D\oplus (-D)$ defined on $\Dom(D)\oplus \Dom(D)\subseteq X_B\oplus X_B$. Then there exists a self-adjoint $B$-compact operator $R\in \mathbb{K}_B(X_B\oplus X_B)$ such that
\begin{enumerate}
\item $R(X_B\oplus X_B)\subseteq \Dom(\tilde{D})$ and the operators $\tilde{D}R$ and $R\tilde{D}$ extend to $B$-compact operators on $X_B\oplus X_B$;
\item the operator $D_{\rm amp}:=\tilde{D}+R$ is invertible with $D_{\rm amp}^{-1}\in \mathbb{K}_B(X_B\oplus X_B)$. 
\end{enumerate} 
If $X_B$ is $\mathbb{Z}/2$-graded by $\gamma$ and $D$ is odd, we can take $R$ odd for the grading $\gamma\oplus -\gamma$ on $X_B\oplus X_B$.  
\end{lemma}

\begin{proof}
The operator 
$$R:=\begin{pmatrix} 
0& (1+D^2)^{-1}\\ 
(1+D^2)^{-1}& 0\end{pmatrix},$$
satisfies all the necessary requirements. Indeed, $D_{\rm amp}$ is a bounded perturbation of $\tilde{D}$ and thus has $B$-compact resolvent. Now $D_{\rm amp}^2=f(D)\oplus f(D)$, where $f(x)=x^2+(1+x^2)^{-2}$, which is a strictly positive function. Since $f\geq1$ and $f(x)\to \infty$ as $|x|\to \infty$, it follows that $\frac{1}{f}$ is a $C_0$-function, so $D_{\rm amp}$ has a $B$-compact inverse.
\end{proof}

\begin{defn}
Let $(\A,X_B,D,\sigma)$ be a weakly twisted Kasparov module and suppose that $(D\pm i)^{-1}\in\mathbb{K}_{B}(X_B)$. An {\bf invertible amplification} of $(\A,X_B,D,\sigma)$ is a weakly twisted Kasparov module $(\A, X_B\oplus X_B, D_{\rm amp}, \sigma_{\rm amp})$ with $\D_{\rm amp}$ as in Lemma \ref{invamp}. Here $X_B\oplus X_B$ is equipped with the $A$-action $a(x\oplus x'):=ax\oplus 0$ and the partially defined automorphism $\sigma_{\rm amp}(T_1\oplus T_2):=\sigma(T_1)\oplus T_2$ for $T_1\in \Dom(\sigma)$ and $T_2\in \End^*_B(X_B)$. 
\end{defn}

If $(\A,X_B,D,\sigma)$ is an even cycle with grading $\gamma$, we grade the module $X_B\oplus X_B$ in the cycle $(\A, X_B\oplus X_B, D_{\rm amp}, \sigma_{\rm amp})$ by $\gamma\oplus -\gamma$, and tacitly assume that $R$ from Lemma \ref{invamp} is chosen to be an odd operator. Invertible amplifications will not play any conceptually important role in the statements this paper, but rather a technical role as it allows us to reduce proofs for general $D$ to the case that $D$ is invertible. 

One of our main tools is the following 
integral formula for fractional powers of the invertible operator $1+D^2$.

\begin{lemma}
\label{intes} 
Let $D$ be a regular self-adjoint operator on a Hilbert $C^*$-module 
$X_{B}$. For any $0<r<1$
\begin{equation}
\label{magicint} 
(1+D^{2})^{-r}=\frac{\sin(r\pi)}{\pi}\int_{0}^{\infty}
\lambda^{-r}(1+D^{2}+\lambda)^{-1}\rd\lambda,
\end{equation}
is a norm convergent integral. Moreover we have the estimates
\begin{align*} 
\|(1+D^{2}+\lambda)^{-s}\|_{\End^*_B(X_B)}&\,\leq (1+\lambda)^{-s},\\
\|D(1+D^{2}+\lambda)^{-\frac{1}{2}}\|_{\End^*_B(X_B)}\leq 1
\quad\mbox{and}&\quad \|D^{2}(1+D^{2}+\lambda)^{-1}\|_{\End^*_B(X_B)}\leq 1.
\end{align*}
\end{lemma}

The integral formula has been used in the Hilbert 
$C^*$-module context since the work of Baaj-Julg \cite{BJ}.
A detailed treatment can be found in \cite[Appendix A, Remark $3$]{CP1}. 
The estimates can be found in \cite[Appendix A, Remark $5$]{CP1}.

For a closed operator $D:\Dom(D)\to X_B$ we view $\Dom(D)$ as a Hilbert $C^*$-module over $B$ equipped with the graph inner product. The following Lemma is similar to results obtained in \cite[Appendix A]{CP1}.

\begin{lemma}
\label{perturbbddtr}
Let $X_B$ be a $B$-Hilbert $C^{*}$-module, $D:\Dom(D)\to X_B$ a self-adjoint regular operator and $R\in\mathbb{K}_{B}(X_B)$ is a self-adjoint operator such that 
\[R:X_B\to \Dom(D),\quad\textnormal{and} \quad DR\in\mathbb{K}_{B}(X_B).\]
Then the bounded adjointable operator 
$$(1+D^2)^{-1/2}-(1+(D+R)^2)^{-1/2}$$ 
on $X_B$ defines a $B$-compact adjointable operator $X_B\to \Dom(D^{2})$.  In particular $|D|-|D+R|$ has a bounded extension to $X_B$ which is $B$-compact whenever $(D\pm i)^{-1}\in\mathbb{K}_{B}(X_B)$.
\end{lemma}

\begin{proof}
As $DR$ is everywhere defined and $B$-compact, its densely defined adjoint $RD$ extends to a $B$-compact operator as well. Moreover we have
\[(D+R)^{2}=D^{2}+RD+DR+R^{2}:\Dom(D^{2})\to X_B,\]
so $(D+R)^{2}$ is a $B$-compact perturbation of $D^{2}$ and $\Dom(D^{2})=\Dom((D+R)^{2})$.
Using the integral formula \eqref{magicint}, we write 
\begin{align*}D^2((1+D^2)^{-1/2}&-(1+(D+R)^2)^{-1/2})\\
&=\frac{1}{\pi}\int_{0}^{\infty}
\lambda^{-1/2}D^2(1+D^{2}+\lambda)^{-1}(DR+RD+R^2)(1+(D+R)^2+\lambda)^{-1}\rd\lambda,
\end{align*}
an identity valid on $\Dom(D)$. Using that $DR+RD+R^2$ is $B$-compact,  the integrand is $B$-compact and by Lemma \ref{intes} it is norm bounded by $\|DR+RD+R^2\|\lambda^{-1/2}(1+\lambda)^{-1}$. Therefore the integral is norm convergent, and hence $(1+D^2)^{-1/2}-(1+(D+R)^2)^{-1/2}$ defines a $B$-compact operator $X_B\to \Dom(D^{2})$. 
Now consider $|D|-|D+R|$ and observe that the function $x\mapsto |x|-x^2(1+x^2)^{-1/2}$ belongs to $C_0(\R)$. It thus suffices to prove that $D^2(1+D^2)^{-1/2}-(D+R)^2(1+(D+R)^2)^{-1/2}$ is $B$-compact. This now follows from $B$-compactness of $DR+RD+R^{2}$, $B$-compactness of $(1+D^2)^{-1/2}-(1+(D+R)^2)^{-1/2}$ as an operator into $\Dom(D^{2})$ and boundedness of $D^{2}:\Dom(D^{2})\to X_B$.
\end{proof}

\begin{prop}
\label{inveramamamap}
Any compact weakly twisted Kasparov module $(\A,X_B,D,\sigma)$ admits an invertible amplification. Any invertible amplification $(\A,X_B\oplus X_B,D_{\rm amp},\sigma_{\rm amp})$ is a compact weakly twisted Kasparov module with $D_{\rm amp}$ invertible. Moreover $(\A,X_B,D,\sigma)$ is Lipschitz regular if and only if $(\A,X_B\oplus X_B,D_{\rm amp},\sigma_{\rm amp})$ is Lipschitz regular.
\end{prop}

\begin{proof}
While we assume that $(\A,X_B,D,\sigma)$ is compact, $D$ is Fredholm and thus $(\A,X_B,D,\sigma)$ admits an invertible amplification using Lemma \ref{invamp}. It is a short algebraic verification to check that $(\A,X_B\oplus X_B,D_{\rm amp},\sigma_{\rm amp})$ is a weakly twisted Kasparov module. Since the inverse of $D_{\rm amp}$ is $B$-compact, its resolvent is $B$-compact and therefore $(\A,X_B\oplus X_B,D_{\rm amp},\sigma_{\rm amp})$ is compact. Recall the notation $\tilde{D}:=D\oplus (-D)$ and note that $(\A,X_B,D,\sigma)$ is Lipschitz regular if and only if $(\A,X_B\oplus X_B,\tilde{D},\sigma_{\rm amp})$ is Lipschitz regular. The statement about Lipschitz regularity follows from that $|\tilde{D}|-|D_{\rm amp}|$ is $B$-compact by Lemma \ref{perturbbddtr}. 
\end{proof}

\begin{rmk}
It is clear from the construction that invertible amplifications commute with saturations in the following sense. Assume that $(\A,X_B,D,\sigma)$ is a unital weakly twisted Kasparov module satisfying all the assumptions of Proposition \ref{saturateprops} (see page \pageref{saturateprops}). Then $(\A,X_B\oplus X_B,D_{\rm amp},\sigma_{\rm amp})$ also satisifes all the assumptions of Proposition \ref{saturateprops} and clearly 
$$(\A_\sigma,X_B\oplus X_B,D_{\rm amp},\sigma_{\rm amp})=(\A_{\sigma_{\rm amp}},X_B\oplus X_B,D_{\rm amp},\sigma_{\rm amp}).$$
In this equation, the left hand side is the invertible amplification of the saturation, and the right hand side is the saturation of the invertible amplification.
\end{rmk}

\subsection{The bounded transform for Lipschitz regular twisted Kasparov modules}

We present a proof of the fact that the
bounded transform of a Lipschitz regular weakly twisted 
unbounded Kasparov module is a well-defined bounded Kasparov module. Our proof closely follows the proof in the special case of twisted spectral triples due to Connes-Moscovici \cite{CoM}. At the expense of a range of other assumptions, Kaad in \cite{Ka} proves a similar statement in the absence of the Lipschitz condition.

Recall that a \emph{normalising function} $\chi$ is a function $\chi:\mathbb{R}\to\mathbb{R}$, continuous on $\textnormal{Spec}(D)$, with the properties that  $\chi(-x)=-\chi(x)$ and $\lim_{x\to\pm \infty}\chi(x)=\pm 1$.

\begin{prop}[\cite{CM}]
\label{bddtwistedtransform}
Let $(\A,X_B,\D,\sigma)$ be a compact Lipschitz regular weakly twisted Kasparov module. Then for any choice of normalising function $\chi$, the triple $(\A,X_B,\chi(\D))$ is a bounded Kasparov module. 
\end{prop}

\begin{proof}
Let $(\A,X_B\oplus X_B,D_{\rm amp},\sigma_{\rm amp})$ be an invertible amplification of $(\A,X_B,\D,\sigma)$ as in Proposition \ref{inveramamamap}. In the special case $\chi(x)=x(1+x^2)^{-1/2}$ we have that $\chi(D)\oplus (-\chi(D))-\chi(D_{\rm amp})\in \K_B(X_B\oplus X_B)$ by Lemma \ref{perturbbddtr}. Up to $B$-compact perturbations, $\chi(D)$ and $\chi(D_{\rm amp})$ are independent of the normalising function $\chi$, so $\chi(D)\oplus (-\chi(D))-\chi(D_{\rm amp})\in \K_B(X_B\oplus X_B)$ for any normalising function. It is clear that $(\A,X_B,\chi(\D))$ is a bounded Kasparov module if and only if $(\A,X_B\oplus X_B,\chi(\D)\oplus (-\chi(D)))$ is a bounded Kasparov module, which in turn is equivalent to $(\A,X_B\oplus X_B,\chi(\D_{\rm amp}))$ being a bounded Kasparov module by the preceeding argument. Moreover, $(\A,X_B\oplus X_B,D_{\rm amp},\sigma_{\rm amp})$ is Lipschitz regular since $(\A,X_B,\D,\sigma)$ is Lipschitz regular. We can in particular assume that $(\A,X_B,\D,\sigma)$ satisfies that $D$ is invertible. 

Since $(\A,X_B,\D,\sigma)$ is compact, the operator $\D$ has $B$-compact resolvent by definition, and it suffices to prove the proposition for the specific choice of function $\chi(x):=\mathrm{sgn}(x)$. We set $F:=\D|\D|^{-1}$ and compute for $a\in \A$ that 
\begin{align}
\nonumber
[F,a]&=D|D|^{-1}a-aD|D|^{-1}=|D|^{-1}(D a-|D|aF) =|D|^{-1}([D,a]_\sigma+\sigma(a)D-|D|aF)\\
\label{somecomlpos}
&=|D|^{-1}([D,a]_\sigma+(\sigma(a)|D|-|D|a)F)=|D|^{-1}([D,a]_\sigma-[|D|,a]_{\sigma}F).
\end{align}
The operator $[D,a]_\sigma-[|D|,a]_{\sigma}F$ has a bounded extension by the bounded twisted commutator condition and the twisted Lipschitz condition. By $B$-compactness of $|D|^{-1}$, we conclude that $[F,a]$ is $B$-compact for all $a\in \A$.
\end{proof}

We note that the operator $\chi(D)$ in Proposition \ref{bddtwistedtransform} is independent of the twist $\sigma$. That is, if there exists an automorphism $\tilde{\sigma}$ of $\A$ such that $D$ and $|D|$ have bounded twisted commutators with $\A$, this suffices to prove that the bounded transform is a Fredholm module. This observation is of crucial importance in the sequel.

\subsection{Logarithmic dampening of Lipschitz regular twisted Kasparov modules}
In this subsection we associate to a Lipschitz regular weakly twisted unbounded Kasparov module an ordinary unbounded Kasparov module in the same $KK$-class. To this end we use the functional calculus and the logarithm function. In \cite[Theorem 8]{Pierrot} Pierrot observed that this is possible for crossed products by a conformal diffeomorphism. In \cite[Theorem 9.14]{DGMW}, \cite[Section 10]{GU} and \cite[Section 2.2.1 and 5.1]{GRU} a related, but slightly broader set of examples is discussed. First, we require a series of lemmas.

\begin{lemma}
\label{logbdd}
Let $\Delta$ be a positive self-adjoint regular operator on a Hilbert $C^{*}$-module $X_B$ with $\Delta^{-1}\in\End^{*}_{B}(X_B)$. Let $a\in\End^{*}_{B}(X_B)$ be a self-adjoint bounded operator on $X_B$ such that $a\Dom(\Delta)\subseteq \Dom(\Delta)$. Then $a\Dom(\log(\Delta))\subseteq \Dom(\log(\Delta))$  and the operators $\Delta a\Delta^{-1}$, $\Delta^{-1}a\Delta$ and $[\log(\Delta),a]$ extend to bounded adjointable operators on $X_B$.
\end{lemma}

\begin{proof}
Since $\Dom(\Delta)=\textnormal{Ran }\Delta^{-1}$, 
the operator $\Delta a\Delta^{-1}$ is closed and defined on all of $X_B$, so by the closed graph theorem it is bounded. It has a densely defined adjoint $\Delta^{-1} a\Delta$, so $\Delta a\Delta^{-1}$ is bounded and adjointable. The adjoint of $\Delta a\Delta^{-1}$ is the closure of $\Delta^{-1} a\Delta$ which therefore extends to a bounded adjointable operator. 
Invertibility of $\Delta$ ensures that the operator $\log(\Delta)$ is defined through the functional calculus for self-adjoint regular operators \cite[Theorem 10.9]{Lancebook} and by definition the submodule
\[\left\{f(\Delta)x: x\in X_B, f\in C_{c}(\mathbb{R})\right\}\subseteq\Dom(\Delta),\]
is a core for $\log(\Delta)$. Since $\log(\Delta)$ is self-adjoint and regular and $\Delta^{-1}\log\Delta$ extends to a bounded operator that equals $(\log\Delta)\Delta^{-1}$, we have $\Dom(\Delta)\subseteq \Dom(\log(\Delta))$, so $\Dom(\Delta)$ is a
core for $\log(\Delta)$.  Therefore $a$ preserves
a core for $\log(\Delta)$. Provided that
$[\log(\Delta),a]$ is bounded on this core, we will see that in fact $a$ preserves
$\Dom(\log(\Delta))$.

For each state $\varphi:B\to\mathbb{C}$, consider the Hilbert space $X_{\varphi}:=X_B\otimes_{\varphi} H_{\varphi}$, with $H_{\varphi}$ the associated GNS representation of $B$. Applying the Hadamard three lines theorem, for any $z\in \C$ with $\mathrm{Re}(z)\in [-1,1]$ the operator $\Delta^{-z}a\Delta^z$ has a bounded extension to $X_{\varphi}$ and 
$$\| \Delta^{-z}a\Delta^z\|\leq \|\Delta^{-1} a\Delta\|^{(1+\mathrm{Re}(z))/2}\|\Delta a\Delta^{-1} \|^{(1-\mathrm{Re}(z))/2}.$$
Moreover, the function $z\mapsto \Delta^{-z}a\Delta^z\in \B(X_\varphi)$ is holomorphic. We compute its derivative at $0$ to be 
$$\frac{\mathrm{d}}{\mathrm{d}z}\bigg|_{z=0}\Delta^{-z}a\Delta^z=[\log(\Delta),a].$$
Holomorphicity of $z\mapsto \Delta^{-z}a\Delta^z\in \B(X_\varphi)$ allows us to deduce that $[\log(\Delta),a]$ extends to a bounded operator on $X_{\varphi}$. Since the state $\varphi$ is arbitrary, all statements are valid in the $C^{*}$-module $X_B$ by the
local-global principle, \cite{KaLe2}.
\end{proof}

\begin{lemma}
\label{perturblog}
Let $X_B$ be a $B$-Hilbert $C^{*}$-module, $D:\Dom(D)\to X_B$ a self-adjoint regular operator and $R\in\mathbb{K}_{B}(X_B)$  a self-adjoint operator such that 
\[R:X_B\to \Dom(D),\quad\textnormal{and} \quad DR\in\mathbb{K}_{B}(X_B).\]
Then $\log(1+|D|)-\log(1+|D+R|)$ has a bounded extension to $X_B$ which is $B$-compact whenever $(D\pm i)^{-1}\in\mathbb{K}_{B}(X_B)$.
\end{lemma}

\begin{proof}
Since $x\mapsto 2\log(1+|x|)-\log(1+x^2)$ is in $C_{0}(\mathbb{R})$, it suffices to prove that $\log(1+D^2)-\log(1+(D+R)^2)$ is $B$-compact. On the core $\Dom(D^2)$ we use the strongly convergent integral expression 
\[\log(1+D^2)=\int_0^1 D^2(1+tD^2)^{-1}\mathrm{d}t.\]
Since $\Dom(D^2) =\Dom((D+R)^2)$, we can write 
\begin{align*}
\log(1+D^2)-&\log(1+(D+R)^2)=\int_0^1 D^2(1+tD^2)^{-1}\mathrm{d}t-\int_0^1 (D+R)^2(1+t(D+R)^2)^{-1}\mathrm{d}t.
\end{align*}
As $DR+RD+R^2\in\mathbb{K}_{B}(X_B)$, $(i\pm D)^{-1}\in \mathbb{K}_B(X_B)$ and $\|(1+t(D+R)^2)^{-1}\|\leq 1$ , it suffices to prove $B$-compactness of 
\begin{align*}
\int_0^1 D^2\big((1+t(D+R)^2)^{-1}-&(1+tD^2)^{-1}\big)\,\mathrm{d}t\\
&=\int_0^1 tD^2(1+tD^2)^{-1}(DR+RD+R^2)(1+t(D+R)^2)^{-1})\,\mathrm{d}t.
\end{align*}
We have that $\|tD^2(1+tD^2)^{-1}\|\leq 1$  for all $t\in [0,1]$, so the integrand is norm-continuous and $B$-compact on $[0,1]$. We conclude that $\log(1+D^2)-\log(1+(D+R)^2)$ is $B$-compact.
\end{proof}

Consider the continuous function $\textnormal{sgnlog}(x):=x|x|^{-1}\log(1+|x|)\in C(\mathbb{R}),$ whose derivative 
 is the $C_0$-function $x\mapsto (1+|x|)^{-1}$. For a self-adjoint regular operator $D$ on a Hilbert $C^{*}$-module $X_B$, the self-adjoint regular operator $\sgnlog(D)$ satisfies $\Dom(D)\subset \Dom(\sgnlog(D))$ and $\Dom(D)$ is a core for $\sgnlog(D)$.

\begin{thm}
\label{logdampofregtwisted}
Assume that $(\A,X_B,D,\sigma)$ is a compact Lipschitz regular weakly twisted unbounded Kasparov module. Then the logarithmic transform
\[
D_{\log}:=\textnormal{sgnlog}(D),
\]
makes $(\A,X_B,D_{\log})$ into a compact unbounded Kasparov module which represents the same $KK$-class as $(\A,X_B,D,\sigma)$.
\end{thm}

\begin{proof}
It is clear that if the operation $(\A,X_B,D,\sigma)\mapsto (\A,X_B,D_{\log})$ is well defined, it preserves $KK$-classes (since $x(1+x^2)^{-1/2}-\textnormal{sgnlog}(x)(1+\log(x)^2)^{-1/2}\in C_0(\R)$). Consider an invertible amplification $(\A, X_B\oplus X_B, D_{\rm amp}, \sigma_{\rm amp})$ of $(\A, X_B, D, \sigma)$ as in Proposition \ref{inveramamamap}. We note that $x\mapsto \textnormal{sgnlog}(x)-x(1+x^2)^{-1/2}\log(1+|x|)$ is a $C_0$-function, so $\sgnlog(D)-D(1+D^2)^{-1/2}\log(1+|D|)$ and $\sgnlog(D_{\rm amp})-D_{\rm amp}(1+D_{\rm amp}^2)^{-1/2}\log(1+|D_{\rm amp}|)$ are $B$-compact. By Lemma \ref{perturbbddtr} and Lemma \ref{perturblog} the operator
$$(\tilde{D}(1+\tilde{D}^2)^{-1/2}\log(1+|\tilde{D}|))-D_{\rm amp}(1+D_{\rm amp}^2)^{-1/2}\log(1+|D_{\rm amp}|)$$ 
is $B$-compact. We conclude that $\sgnlog(\tilde{D})=\textnormal{sgnlog}(D)\oplus (-\textnormal{sgnlog}(D))$ is a $B$-compact perturbation of $\sgnlog(D_{\rm amp})$ on $X_B\oplus X_B$. Since $a\in\A$ preserves the domain $D$ it preserves the domain of $\log(1+|D|)$ (and so of $\textnormal{sgnlog}(D)$), and $[\textnormal{sgnlog}(\tilde{D}),a]=[\textnormal{sgnlog}(D),a]\oplus 0_X$, we can without loss of generality assume that $D$ is invertible.

The only thing to prove is that $D_{\log}$ has bounded commutators with $\A$. Consider
\[D_{\log,0}:=|D|^{-1}\log(|D|)D=TD\quad \textnormal{with}\quad T=|D|^{-1}\log(|D|).\]
Since $D_{\log}-D_{\log,0}$ is bounded, it suffices to prove that $D_{\log,0}=TD$ has bounded commutators with $\A$. 

The operator $T[D,a]_{\sigma}$ is bounded and
\begin{align*}
[TD,a]=T[D,a]_{\sigma}+(T\sigma(a)-aT)D.
\end{align*}
It therefore suffices to show that $(T\sigma(a)-aT)D$ is bounded. We look at
\begin{align}
\label{1stterm}
(T\sigma(a)-aT)D&=((|D|^{-1}\log|D|)\sigma(a)-a|D|^{-1}\log|D|)D.
\end{align}
Since
\[|D|^{-1}\sigma(a)-a|D|^{-1}=-|D|^{-1}[|D|,a]_{\sigma}|D|^{-1},\] 
it holds that 
\[(|D|^{-1}\log|D|)\sigma(a)=(\log|D|)a|D|^{-1}-(\log|D|) |D|^{-1}[|D|,a]_{\sigma}|D|^{-1}.\]
Thus up to a bounded operator \eqref{1stterm} equals
\[[\log|D|,a]|D|^{-1}D=[\log|D|,a]F.\]
Now $a\Dom(|D|)\subset\Dom(|D|)$ and Lemma \ref{logbdd} applied to $\Delta=|D|$ gives that $[\log |D|,a]$ is bounded, so we are done.
\end{proof}

\begin{prop}
\label{sumoflogdamp}
Let $(\A,\H,\D,\sigma)$ be a compact Lipschitz regular weakly twisted spectral triple. For $s>0$ we have $(1+D^2)^{-s/2}\in \mathcal{L}^1$ if and only if
$\mathrm{e}^{-s|D_{\log}|}\in \mathcal{L}^1$. Hence the weakly twisted spectral triple
$(\A,\H,\D,\sigma)$ is
finitely summable if and only if the 
spectral triple $(\A,\H,\D_{\log})$ 
is $\mathrm{Li}_1$-summable.
\end{prop}

The proposition follows from the definitions, 
see \cite{GRU}, because
\[
\mathrm{e}^{-s|D_{\log}|}=\mathrm{e}^{-s\log(1+|D|)}=(1+|D|)^{-s}.\qedhere
\]

\begin{ex}
\label{loggedcirclediff}
Let us revisit the twisted spectral triple $(C^{\infty}(S^{1})\rtimes^{\alg}\mathbb{Z}, L^{2}(S^{1}),D,\sigma)$ of Example \ref{circleexdef}.
The logarithmic dampening $(C^{\infty}(S^{1})\rtimes^{\alg}\mathbb{Z}, L^{2}(S^{1}),D_{\log})$ is an ordinary spectral triple by Theorem \ref{logdampofregtwisted}. We can replace the operator $D_{\log}:=D|D|^{-1}\log (1+|D|)$ by the operator $\partial_{\log}$, given by
\[\partial_{\log}(\mathrm{e}^{2\pi int}):=(\sgn(n)\log |n| )\cdot \mathrm{e}^{2\pi int}, \quad \mbox{for}\ n\neq 0\quad\mbox{and}\quad \partial_{\log}(1)=0
\]
as the operator $\partial_{\log}-D_{\log}$ is bounded.
\end{ex}

\begin{corl}
\label{easycoroflog}
Let $(\A,X_B,\D)$ be a triple containing the following information:
\begin{itemize}
\item $X_B$ is a countably generated Hilbert $C^*$-module over $B$;
\item $D$ is a regular self-adjoint operator on $X_B$ with $B$-compact inverse;
\item $\mathcal{A}$ is a $*$-algebra represented on $X_B$ such that $\A\Dom(D)\subseteq \Dom(D)$.
\end{itemize}
Let $F:=\D|\D|^{-1}$ as in Definition \ref{defn:sign} and assume that $[F,a]:X_B\to \Dom(\D)$ for all $a\in \A$. Then with 
$$D_{\log}:=\sgnlog(D),$$
the collection $(\A,X_B,D_{\log})$ is an unbounded Kasparov module.
\end{corl}

\begin{proof}
By Theorem \ref{logdampofregtwisted}, it suffices to prove that $(\A,X_B,\D,\sigma)$ is a Lipschitz regular weakly twisted Kasparov module for $\sigma(a):=|\D|a|\D|^{-1}$. We have that $\A\subseteq \Dom(\sigma)$ by Lemma \ref{logbdd} (see page \pageref{logbdd}) because $\A$ preserves $\Dom(\D)=\Dom(|\D|)$. The closed graph theorem and boundedness of $[F,a]$ on $X_B$ imply that if $[F,a]X_B\subseteq \Dom(\D)$ then $[F,a]:X_B\to \Dom(\D)$ is continuous. In particular, $|\D|[F,a]$ is a bounded adjointable operator. Moreover, for any $a\in \A$, $[\D,a]_\sigma$ is bounded because
$$[\D,a]_\sigma=\D a-\sigma(a)\D=F|\D|a-|\D|aF=|\D|[F,a].$$
Since $D$ has $B$-compact inverse, the preceeding argument shows that $(\A,X_B,\D,\sigma)$ is a weakly twisted spectral triple. Clearly, $[|D|,a]_\sigma=|D|a-|D|a=0$ is bounded so $(\A,X_B,\D,\sigma)$ is a Lipschitz regular weakly twisted unbounded Kasparov module.
\end{proof}

\subsection{Higher order Kasparov modules}

In this section we describe a weakening of the 
definition of Kasparov module, which as far as we 
know first appeared in the work of Wahl, \cite{Charlotte}. 
Key observations can be found in \cite[Lemma 51]{grensing} and the notion reappeared
in work by the first two listed authors on Cuntz-Krieger algebras \cite{GM}.
Related notions are anticipated in the literature, e.g. \cite{CP1}.
It allows for both higher order elliptic 
operators in classical settings, and provides a method 
for handling some of the difficulties that arise in dynamical examples. 
The main idea here is to relax the requirement that the 
commutators $[D,a]$ be bounded, by only asking for 
a weaker bound relative to $D$.

To introduce the concept we need to ensure that domain issues are
appropriately addressed, and so need some preliminary definitions.

\begin{defn}
Let $B$ be a $C^*$-algebra.
Let $X_B$ be a countably generated right $B$-Hilbert $C^*$-module, 
$T$ a densely defined operator on $X_B$, $D$ a  densely 
defined self-adjoint regular operator on $X_B$ and $\varepsilon>0$. 
We say that $T$ is $\varepsilon$-bounded with respect to 
$D$ if the operators $T(1+D^{2})^{-\frac{1-\varepsilon}{2}}$ 
and $(1+D^{2})^{-\frac{1-\varepsilon}{2}} T$ are densely defined and norm-bounded. 
\end{defn}

\begin{defn} 
Let $B$ be  a $C^*$-algebra. 
Let $X_B$ be a countably generated right 
$B$-Hilbert $C^*$-module. An operator $a\in \End_B^*(X_B)$ 
has $\varepsilon$-\emph{bounded commutators} with 
the self-adjoint regular operator $D:\,\Dom(D)\subset X_B\to X_B$ if 
\begin{enumerate} 
\item $a\Dom(D)\subset \Dom(D) $;
\item $[D,a]$ is $\varepsilon$-bounded with respect to $D$.
\end{enumerate}
In short we say that $[D,a]$ is $\varepsilon$-bounded. 
\end{defn}

\begin{ex}
Let us give a geometric example of $\varepsilon$-bounded commutators to explain the appearance of the parameter $\varepsilon>0$, and the name `higher order spectral triple'. Let $D$ be a self-adjoint elliptic pseudodifferential operator of order $m >0$ acting on a vector bundle $E\to M$ on a closed manifold $M$. 
The Hilbert space is $\H=L^2(M,E)$. The domain of $D$ is the Sobolev space $W^{m,2}(M,E)$. If $a\in C^\infty(M)$, then it is well-known that
$[D,a]$ is a pseudodifferential operator of order $m-1$. Hence $(1+D^2)^{\frac{1-m}{2m}}[D,a]$ and $[D,a](1+D^2)^{\frac{1-m}{2m}}$ are pseudodifferential operators of order $0$, thus bounded on $L^2(M,E)$. 

We conclude that any $a\in C^\infty(M)$ has $1/m$-bounded commutators with $D$. As such, one can consider the reciprocal $\varepsilon^{-1}$ as an ``order" of the operator $D$ appearing in an $\varepsilon$-bounded commutator.
\end{ex}

\begin{rmk}
Somewhat undermining the notion of order,  it is for most purposes not actually necessary for the value of $\varepsilon$ to be the same for all $a\in\A$. We will not carry this level of generality with us.
\end{rmk}

Wahl called these Kasparov modules 
`truly unbounded'. Goffeng-Mesland \cite{GM} 
dubbed them $\varepsilon$-unbounded Kasparov modules, with $\varepsilon$ the analogue of the reciprocal of the order. 
Due to 
examples arising from higher order differential operators, 
we feel that the adjective `higher order'
is most appropriate to describe the notion.

\begin{defn} 
\label{defn:higher-order}
Let $\A$ be a $*$-algebra, $m>0$ and set $\varepsilon=\frac{1}{m}$. An odd 
order $m$ Kasparov module is  a triple $(\A,X_B,D)$ where $X_B$ is a 
countably generated $B$-Hilbert $C^*$-module with a 
$*$-representation of $\A$ and $D$ is a self-adjoint regular operator such that
\begin{enumerate}
\item $a(1+D^{2})^{-\frac{1}{2}}\in \mathbb{K}_B(X_B)$ for all $a\in \A$ and
\item the image $\A$ in $\End^{*}_{B}(X_B)$ is contained in the space
\[\Lip^{\varepsilon}(X_B,D)
:=\{a\in \End^{*}_{B}(X_B):\  a\cdot\Dom(\D)\subset\Dom(\D),\ \ [D,a] \textnormal{ is $\varepsilon$-bounded}\}.\]
\end{enumerate}
If $X_B$ is a $\Z/2\Z$-graded $B$-Hilbert $C^*$-module,
and $(\A,X_B,D)$ is as above with $D$ odd in the grading 
on $X_B$ and $A$ acting as even operators, 
we say that $(\A,X_B,D)$ is an even order $m$ 
spectral triple. Otherwise, we say that $(\A,X_B,D)$ is odd. 
\end{defn}

As above (see Remark \ref{droppingpi}), we treat $\A$ as a subalgebra of $\End_B^*(X_B)$ despite it acting via a possibly non-faithful representation.

\begin{rmk}  
\label{arbiorderrem}
A higher order $m$ Kasparov module for  $m>0$ 
is a higher order $m'$ Kasparov module for $m'>0$ 
whenever $m'\geq m$. If we can take $m=1$ then we talk about an ordinary
unbounded Kasparov module, and if $B=\C$ so that $X_B$ is a Hilbert space
we speak about higher order spectral triples.
\end{rmk}

Apart from higher order (elliptic) differential operators,
higher order spectral triples arise in the construction of the unbounded Kasparov product. The context is typically that of a dynamical system on a metric measure space, for which the dynamics does not preserve the metric nor the measure. The phenomenon has been examined in detail for Cuntz-Krieger algebras \cite[Section 5 and 6]{GM}, Cuntz-Pimsner algebras of vector bundles \cite[Section 4]{GMR}, group $C^{*}$-algebras and boundary crossed products of groups of M\"{o}bius transformations \cite[Section 4]{MS} and Delone sets with finite local complexity \cite[Section 5]{BM}. For later reference in this context, we include the following general construction.

\begin{ex}
\label{raisebyone}
Let $(\A, X_B,D)$ be an ordinary unbounded Kasparov module and $0<s<1$. If the bounded operators $[D,a]$ satisfy the mapping property
\begin{equation}
\label{domaindoncncd}
[D,a]:\Dom(D)\to \Dom(|D|^{s}) \quad\mbox{ for all $a\in \A$},
\end{equation}
then the triple $(\A,X_B,D|D|^{s})$ is a higher order Kasparov module of order $m=\frac{1+s}{s}$. This can be seen by writing $\Dom(D)=(1+|D|^{2+2s})^{-\frac{1-\varepsilon}{2}}X_B$ with $\varepsilon:=\frac{s}{1+s}$. The operator
\begin{equation}
[D|D|^{s},a](1+|D|^{2+2s})^{-\frac{1-\varepsilon}{2}},
\label{eq:formerly-above}
\end{equation}
is bounded because on $\Dom(D)$ it holds that 
$$[D|D|^{s},a]=[|D|^{s}D,a]=|D|^{s}[D,a]+[|D|^{s},a]D,$$ 
and $[|D|^{s},a]$ is bounded for $0<s<1$ whereas $[D,a]:\Dom(D)\to \Dom(|D|^{s})$ by assumption. 

The notion of order carries some degree of arbitrariness due to Remark \ref{arbiorderrem}. In fact, if $D$ was a self-adjoint first order elliptic operator on a closed manifold we would expect the order of $D|D|^s$ to be $1+s$ rather than $\frac{1+s}{s}$. Under additional smoothness assumptions the order of $(\A,X_B,D|D|^{s})$ improves. For simplicity, we assume that $(\A, X_B,D)$ is an ordinary unbounded Kasparov module with $D$ invertible. Let $s\in (-1,1]$ and assume that 
\begin{equation}
\label{domaindoncncdcomm}
[|D|^s,a]|D|^{1-s} \quad\textnormal{and}\quad |D|^s[D,a]|D|^{-s}\quad \mbox{are bounded for all $a\in \A$}.
\end{equation}

These two conditions are indeed satisfied when $D$ is a self-adjoint first order elliptic operator on a closed manifold and are often easy to check in practice. Set $\varepsilon=\frac{1}{1+s}$ so $\frac{1-\varepsilon}{2}=\frac{s}{2+2s}$. Under the conditions \eqref{domaindoncncdcomm}, we can continue the computation 
from Equation \eqref{eq:formerly-above}, obtaining
\begin{align*}
[D|D|^{s},a]&(1+|D|^{2+2s})^{-\frac{1-\varepsilon}{2}}=\left(|D|^{s}[D,a]+[|D|^{s},a]D\right)(1+|D|^{2+2s})^{-\frac{s}{2+2s}}=\\
=&\left(|D|^{s}[D,a]|D|^{-s}+[|D|^{s},a]|D|^{1-s}F \right)|D|^s(1+|D|^{2+2s})^{-\frac{s}{2+2s}}
\end{align*} 
Since $|D|^s(1+|D|^{2+2s})^{-\frac{s}{2+2s}}$ is bounded, the conditions \eqref{domaindoncncdcomm} ensure that $[D,a]$ is $\varepsilon$-bounded for $\varepsilon=\frac{1}{1+s}$. In particular, $(\A,X_B,D|D|^{s})$ is of order $1+s$.
\end{ex}

\subsection{The bounded transform for higher order Kasparov modules}
\label{subsec:bdd-trans}
We now come to the main result about higher 
order Kasparov cycles, concerning the bounded 
transform and the relation to $KK$-theory.
The bounded transform,
\begin{equation}
\label{eq:bdd-trans}
F_D:=D(1+D^2)^{-1/2},
\end{equation}
of a higher order Kasparov module provides a Fredholm module.

\begin{thm}[cf. \cite{BJ}]
\label{epsilonFred} 
The bounded transform $(A,X_B, D(1+D^{2})^{-\frac{1}{2}})$ 
of a higher order Kasparov module $(\A,X_B,D)$ is an $(A,B)$ 
Kasparov module and hence defines a class in $KK_{*}(A,B)$. 
Moreover, for any choice of normalising function $\chi$ it holds that 
$a(D(1+D^{2})^{-\frac{1}{2}}-\chi(D))\in \K_B(X_B)$ 
is $B$-compact for all $a\in A$.
\end{thm}

For $0<\varepsilon<1$, this theorem was first proved in \cite[Discussion after Definition 2.4]{Charlotte} and independently in \cite[Lemma $51$]{grensing}. Several years later the result resurfaced in \cite[Appendix]{GM}, where several statements concerning the Kasparov product of higher order modules are discussed. The proof is based on the integral formula \eqref{magicint}, and is essentially identical to the proof for ordinary unbounded Kasparov modules \cite{BJ}.
\begin{ex} 
\label{PVextmod}
 Let $\gamma:S^{1}\to S^{1}$ be a (non-isometric) diffeomorphism, generating an action of $\mathbb{Z}$ on $C(S^{1})$. Consider the $C(S^1)$-Hilbert $C^{*}$-module $X_{C(S^1)}:=\ell^{2}(\mathbb{Z})\otimes C(S^{1})$, and write $e_{n}$ for the standard basis of $\ell^{2}(\mathbb{Z})$. Represent $C(S^{1})$ and $\gamma$ on $X_{C(S^1)}$ via
\[\pi(f)(e_n\otimes \psi):= e_{n}\otimes (f\circ \gamma^{-n})\cdot \psi, \quad \alpha(\gamma^{k})(e_n\otimes \psi):=e_{n+k}\otimes \psi.\]
Then we can view $\alpha:\mathbb{Z}\to \End^{*}_{C(S^{1})}(X_{C(S^1)})$ as a unitary representation and it holds that
 $$\alpha(\gamma^{k})\pi(f)\alpha(\gamma^{-k})=\pi(f\circ\gamma^{k}),\quad k\in\mathbb{Z}$$

and so $(\pi,\alpha)$ form a covariant representation as $C(S^1)$-linear adjointable endomorphisms of $(C(S^{1}),\gamma)$ on $X_{C(S^1)}$. We denote the corresponding representation by $\pi_X:C(S^{1})\rtimes\mathbb{Z}\to \End_{C(S^1)}^*(X_{C(S^1)})$. Define a self-adjoint regular operator $N$ on $X_{C(S^1)}$ by $N(e_{n}\otimes\psi):=ne_{n}\otimes \psi$. Then
\[(1+N^{2})^{-1}\in\mathbb{K}_{C(S^1)}(X_{C(S^1)}),\quad [N,\alpha(\gamma)]=\alpha(\gamma),\quad [N,\pi(f)]=0,\]
and $(C^{\infty}(S^{1})\rtimes^{\alg}_{\gamma}\mathbb{Z}, X_{C(S^1)}, N)$ is an unbounded Kasparov module. Its associated $KK$-class in $KK_1(C(S^{1})\rtimes_{\gamma}\mathbb{Z},C(S^{1}))$ coincides with that of the Pimnser-Voiculescu extension as its bounded transform defines the Toeplitz extension appearing in Cuntz' proof of the Pimsner-Voiculescu sequence \cite{cuntzandktheo}. 

Pick an $s\in (-1,1]$. We use the convention that $0^s=0$. It is easily verified that
$$
[|N|^s,\alpha(\gamma)]=(|N|^s-|N-1|^s)\alpha(\gamma).
$$
It follows that 
$$
[|N|^s,\alpha(\gamma)]|N|^{1-s}
=(|N|^s-|N-1|^s)|N-1|^{1-s}\alpha(\gamma),
$$
which is bounded. Moreover, 
$$
|N|^s[N,\alpha(\gamma)]|N|^{-s}=|N|^s\alpha(\gamma)|N|^{-s}
=|N|^s|N-1|^{-s}\alpha(\gamma),
$$
is also bounded. Therefore, $(C^{\infty}(S^{1})\rtimes^{\alg}_{\gamma}\mathbb{Z}, X_{C(S^1)}, N)$ satisfies \eqref{domaindoncncdcomm}. The argument in Example \ref{raisebyone} shows that
\[(C^{\infty}(S^{1})\rtimes^{\alg}_{\gamma}\mathbb{Z}, X_{C(S^1)}, N_{s}),\quad N_{s}:=N|N|^{s},\quad s\in (-1,1],\]
is an order $(1+s)$ Kasparov module representing the same class as $(C^{\infty}(S^{1})\rtimes^{\alg}_{\gamma}\mathbb{Z}, X_{C(S^1)}, N)$.
\end{ex}

The modification $N_{s}$ in Example \ref{PVextmod} may seem artificial. The remainder of this section is devoted to the discussion of an example that shows that higher order Kasparov modules arise naturally in a dynamical context, in which the operator $N_s$ plays an essential role. Theorem \ref{epsilonFred} above 
allows one to keep track of the index theoretic information in such cases. In particular, in Proposition \ref{PVprod} below, higher order spectral triples provide a solution to a problem where twisted spectral triples do not.

Consider the spectral triple $(C^{\infty}(S^{1}), L^{2}(S^{1}), \partial_{\log})$, with $\partial_{\log}$ defined as in Example \ref{loggedcirclediff}. Then the class 
\[[(C^{\infty}(S^{1}), L^{2}(S^{1}), \partial_{\log})]\in K^{1}(C(S^{1})),\]
generates $K^{1}(C(S^{1}))$. Let $H:=X_{C(S^1)}\otimes_{C(S^{1})}L^{2}(S^{1})\simeq \ell^{2}(\mathbb{Z})\otimes L^{2}(S^{1})$ and  write $\pi:=\pi_{X}\otimes 1$ for the representation of $C(S^{1})\rtimes_{\gamma}\mathbb{Z}$ on $H$ induced by the representation $\pi_{X}$ on $X_{C(S^1)}$ from Example \ref{PVextmod}.

Define a self-adjoint operator $\slashed{D}_{\log}$ on $\H$ by $\slashed{D}_{\log}(e_{n}\otimes\psi):=e_{n}\otimes \partial_{\log}\psi$, for $\psi\in\Dom(\partial_{\log})$ and write $N_{s}$ for the self-adjoint operator $N_{s}\otimes 1$. 

\begin{prop}
\label{PVprod}
Let $\gamma:S^{1}\to S^{1}$ be a (non-isometric) diffeomorphism. For any $0<s\leq 1$ 
\begin{equation*}
\left(C^{\infty}(S^{1})\rtimes^{\alg}_\gamma\mathbb{Z}, H\oplus H,  D:=\begin{pmatrix}0 & N_{s}+i\slashed{D}_{\log} \\ N_{s}-i\slashed{D}_{\log}&0\end{pmatrix}\right),
\end{equation*}
is a higher order spectral triple of order $\frac{1+s}{s}$. It represents the class 
$$
\partial \left[\left(C(S^{1}), L^{2}(S^{1}), i\frac{\d}{\d x}\right)\right]\in K^{0}(C(S^{1})\rtimes_{\gamma}\mathbb{Z}),
$$
where $\partial: K^{1}(C(S^{1}))\to K^{0}(C(S^{1})\rtimes_{\gamma}\mathbb{Z})$ is the boundary map in the Pimsner-Voiculescu exact sequence.
\end{prop}

\begin{proof}
 It is straightforward to check that $D$ is essentially self-adjoint with compact resolvent, as $D^{2}=\textnormal{diag} ( |N|^{2+2s}+\slashed{D}_{\log}^{2})$ on $C_{c}(\mathbb{Z})\otimes^{\alg}H^{1}(S^{1})\otimes\mathbb{C}^{2}$. Moreover, the hypotheses of \cite[Theorem A.7]{GM} are all trivially satisfied, so that if $D$ defines a higher order spectral triple, it represents the product of the classes $[N_{s}]$ and $[\partial_{\log}]$.

It thus remains to verify $\varepsilon$-bounded commutators. For now we assume that $s\in (-1,1]$ and consider an $\varepsilon\geq 0$. It holds that
\begin{equation}
\label{theeasycomm}
[\slashed{D}_{\log},\alpha(\gamma)]=[N_{s},\pi(f)]=0,
\end{equation}
and these are bounded operators. Since $\Dom(D)=\Dom(N_{s})\cap\Dom(\slashed{D}_{\log})$ the operator 
$$(1+N_{s})^{\frac{1-\varepsilon}{2}}(1+N_{s}^{2}+\slashed{D}_{\log}^{2})^{-\frac{1-\varepsilon}{2}}, $$ 
is bounded for $\varepsilon\in [0,1]$. Thus for any $T$ and $\varepsilon\in [0,1]$, boundedness of $T(1+N_{s}^{2})^{-\frac{1-\varepsilon}{2}}$ implies boundedness of $T(1+N_{s}^{2}+\slashed{D}_{\log}^{2})^{-\frac{1-\varepsilon}{2}}$. 

From Example \ref{PVextmod}, we know that
$$[N_s,\alpha(\gamma)](1+N_s^{2})^{-\frac{1-\varepsilon}{2}}, $$
is a bounded operator for $0\leq \varepsilon\leq \frac{1}{1+s}$. From this fact and Equation \eqref{theeasycomm} we deduce that $[D,\alpha(\gamma)]$ is $\varepsilon$-bounded with respect to $D$ for $0<\varepsilon\leq \min\left(1,\frac{1}{1+s}\right)$.

Lastly we address $[\slashed{D}_{\log},\pi(f)](1+ N^{2}_s)^{-\frac{1-\varepsilon}{2}}$.
Recall the notation $\pi(\gamma^{n})$ from Example \ref{circleexdef}. It holds that 
\begin{equation}
\label{controlthis}
\left\|[\slashed{D}_{\log},\pi(f)](1+ |N|^{2+2s})^{-\frac{1-\varepsilon}{2}}\right\|=\sup_{n\in\mathbb{Z}}\left\| [\partial_{\log},\pi(\gamma^{n})\pi(f)\pi(\gamma^{-n})](1+|n|^{2+2s})^{-\frac{1-\varepsilon}{2}}\right \|.
\end{equation}
From
\begin{align*}
[\partial_{\log},\pi(\gamma^{n})\pi(f)\pi(\gamma^{-n})]&=[\partial_{\log}, \pi(\gamma^{n})]\pi(f)\pi(\gamma^{-n})+\pi(\gamma^{n})[\partial_{\log},\pi(f)]\pi(\gamma^{-n})\\
&\quad \quad\quad +\pi(\gamma^{n})\pi(f)[\partial_{\log},\pi(\gamma^{-n})].
\end{align*}
Since $\|[\partial_{\log},\pi(\gamma^{-n})]\|\leq C|n|$ for $C:=\|[\partial_{\log},\pi(\gamma)]\|$, it follows that 
$$\|[\partial_{\log},\pi(\gamma^{n})\pi(f)\pi(\gamma^{-n})]\|\leq 2C|n|\|\pi(f)\|+\|[\partial_{\log},\pi(f)]\|.$$
The expression 
$$|n|(1+|n|^{2+2s})^{-\frac{1-\varepsilon}{2}},$$
is uniformly bounded in $n$ if and only if $\varepsilon \leq \frac{s}{1+s}$. From this we deduce that the supremum in \eqref{controlthis} is finite if and only if $\varepsilon \leq \frac{s}{1+s}$. By extension, Equation \eqref{theeasycomm} allows us to deduce that $[D,\pi(f)]$ is $\varepsilon$-bounded with respect to $D$ for $0<\varepsilon\leq \frac{s}{1+s}$.

Summarizing, we need $\varepsilon$ to satisfy $\varepsilon \in (0,1]$, $\varepsilon\leq \frac{1}{1+s}$ and $\varepsilon \leq \frac{s}{1+s}$. It is therefore required that $0<s\leq 1$. For $s\in [0,1]$, $\frac{s}{1+s}\leq \frac{1}{1+s}$. These conditions can be condensed to 
$$0<s\leq 1\quad\mbox{and}\quad 0<\varepsilon\leq \frac{s}{1+s}.$$ 
We conclude that \eqref{PVextmod} is a higher order Kasparov module of order $\frac{1+s}{s}$ whenever $s\in (0,1]$. 
\end{proof}
\begin{rmk}It should be noted that these estimates are not necessarily sharp. It can be verified that $[D,\alpha(\gamma)]$ is $\varepsilon$-bounded with respect to $D$ if and only if $0<\varepsilon\leq \min\left(1,\frac{1}{1+s}\right)$. It is however possible that a finer analysis of 
$$
\left\|\big[\partial_{\log},\pi(\gamma^{n})\pi(f)\pi(\gamma^{-n})\big]\left(1+|n|^{2+2s}+\partial_{\rm log}^2\right)^{-\frac{1-\varepsilon}{2}}\right\|
$$ 
would reveal that $[D,\pi(f)]$ is $\varepsilon$-bounded with respect to $D$ for some range of $\varepsilon>\frac{s}{1+s}$.
\end{rmk}
For an abstract $C^{*}$-algebra $A$ with a $*$-automorphism $\gamma$, the boundary map in the Pimsner-Voiculescu exact sequence 
is represented by the unbounded Kasparov module $(A\rtimes_{\gamma}^{\alg}\mathbb{Z}, X_{A}, N)$ (see \cite[Theorem 3.1]{RRS}), where $X_{A}:=\ell^{2}(\mathbb{Z})\otimes A$ and $N$ is the self-adjoint regular operator on $X_{A}$ defined by $N(e_{n}\otimes a):=ne_{n}\otimes a$. As above, $A$ and $\gamma$ are represented on $X_{A}$ via
\[
\pi(a)(e_n\otimes b):= e_{n}\otimes (\gamma^{n}(a))\cdot b, 
\quad \alpha(\gamma^{k})(e_n\otimes b):=e_{n+k}\otimes b.
\]
The operator $N_{s}$ can now be constructed as above, and we obtain the following general result.

\begin{thm}
\label{PVprodgen}
Let $(\A,H_0,D_0)$ be an odd spectral triple and set $F:=\textnormal{sgn}(D_0)=D_0|D_0|^{-1}$ as in Definition \ref{defn:sign}. Assume that $\gamma$ is a $*$-automorphism of $\A$ that is implemented by a unitary $U$ on $H_0$ such that 
$$U:\Dom D_0\to \Dom D_0,\quad [F,U]:H_0\to \Dom(D_0).$$
Let $(A\rtimes_{\gamma}^{\alg}\mathbb{Z}, X_{A}, N)$  be the Kasparov module 
defined by $\gamma$. Set $H:=X_{A}\otimes_{A}H_0=\ell^2(\Z)\otimes H_0$ and  $\slashed{D}_{\log}:=1_{\ell^2(\Z)}\otimes \mathrm{sgnlog}(D_0)$. For any $0<s\leq 1$ 
the data
\begin{equation*}
\left(\A\rtimes^{\alg}_\gamma\mathbb{Z},H\oplus H,  D:=\begin{pmatrix}0 & N_{s}+i\slashed{D}_{\log} \\ N_{s}-i\slashed{D}_{\log}&0\end{pmatrix}\right),
\end{equation*}
defines a higher order spectral triple of order $\frac{1+s}{s}$. It represents the class 
$$
\partial \left[\left(\A,H_0,D_0\right)\right]\in K^{0}(A\rtimes_{\gamma}\mathbb{Z}),
$$
where $\partial: K^{1}(A)\to K^{0}(A\rtimes_{\gamma}\mathbb{Z})$ is the boundary map in the Pimsner-Voiculescu exact sequence.
\end{thm}

\begin{proof}
We first replace $D_0$ by $D_0+p_0$ where $p_0$ is the projection onto $\ker D_0$. Since 
\[(D_0+p_0)|D_0+p_0|^{-1}=D_0|D_0|^{-1}+p_0,\, p_0:H_0\to\ker D_0\subset\Dom D_0, \, U:\Dom D_0\to\Dom D_0,\]
it follows that $[p_0,U]:H_0\to \Dom D_0$. Thus $D_0+p_0$ is invertible and satisfies the same assumptions as $D_0$. Furthermore
\begin{align*}
\sgnlog(D_0+p_0)&=(D_0|D_0|^{-1}+p_0)\log(1+|D_0|+p_0)\\&=D_0|D_0|^{-1}\log(1+|D_0|+p_0)+p_0\log(1+|D_0|+p_0),
\end{align*}
and $p_0\log(1+|D_0|+p_0)$ is bounded. Since $p_0$ commutes with $D_0$ also
\begin{align*}
(\log(1+|D_0|+p_0)-\log(1+|D_0|))&=\log\left((1+|D_0|+p_0)(1+|D_0|)^{-1}\right)\\
&=\log(1+p_0(1+|D_0|)^{-1}),\end{align*}
is a bounded operator. Thus $\sgnlog(D_0+p_0)-\sgnlog(D_0)$ is bounded and it suffices to prove the theorem for invertible $D_0$.
The assumptions on $U$ and $[F,U]$ and Corollary \ref{easycoroflog} (see page \pageref{easycoroflog}) guarantee that we obtain a spectral triple $(\A\rtimes^{\alg}_\gamma\mathbb{Z},H_0,\sgnlog(D_0))$ when letting $\A\rtimes^{\alg}_\gamma\mathbb{Z}$ act on $H_0$ via $U$. The theorem is then proved ad verbatim to Proposition \ref{PVprod}.
\end{proof}

We note that we have used both the logarithmic dampening of $\partial$ (or $D_0$) and the higher order picture in order to achieve the last two results. 
If $[D_0,U]$ is bounded (i.e.$\gamma$ is isometric in the case of Proposition \ref{PVprod}), one can replace $N|N|^{s}$ by $N$ and $\slashed{D}_{\log}$ by $\slashed{D}=1\otimes D_0$.
This results in an ordinary spectral triple, as was studied in detail in \cite{BMR}. In general the operator $\slashed{D}$ gives rise to bounded twisted commutators and the operator $N$ to bounded commutators,
and this prevents their sum giving either a spectral triple or 
a twisted spectral triple. 
\begin{rmk}
A version of Theorem \ref{PVprodgen} for more general Pimsner exact sequences using the unbounded representative of the Cuntz-Pimsner extension constructed in \cite{GMR}
would be of great interest. Some further technical problems need to be addressed, centering around the construction of connections, which in the above cases could be chosen trivial and have thus been omitted from the discussion. These issues are present already for generalised crossed products and justify a separate study.
\end{rmk}
\subsection{Logarithmic dampening of higher order Kasparov modules}

From the perspective of index theory, the homotopy class of the bounded transform $F_{D}$ contains the relevant information. 
From a purely topological point of view, the unbounded representative of 
the $KK$-class of $F_D$ is largely irrelevant and somewhat arbitrary.
In this section we show that the logarithm function can be used to turn any higher order Kasparov module into an order $1$ (i.e. ordinary) Kasparov module, \emph{with $B$-compact commutators}. This logarithmic dampening supplies us with a class of unbounded Kasparov modules that are analytically very close to bounded Fredholm modules.

We first present a lemma concerning the integral representation of the logarithm function. Recall that given a $C^*$-subalgebra $A\subset \End^{*}_{B}(E)$ an operator $K\in\End^{*}_{B}(E)$ is $A$-{\bf locally $B$-compact} if $aK,Ka\in\mathbb{K}_B(E)$.

\begin{lemma} 
\label{triv}
Let $K$ be positive and $A$-locally $B$-compact. Then $\log(1+K)$ is $A$-locally $B$-compact.
\end{lemma}

\begin{proof} 
The operator norm convergent integral expression
\begin{align*}
\log(1+K)=K\int_{0}^{1}(1+tK)^{-1}dt,
\end{align*}
immediately gives the statement.
\end{proof}

\begin{lemma}
\label{weaklip}
Let $0\leq \alpha<\frac{1}{2}$, $a\Dom(D)\subset\Dom(\D)$ and $[D,a](1+D^{2})^{-\alpha}$ be bounded. Then for all $\beta$ with $\alpha<\beta<1$, $D[(1+D^{2})^{-\beta},a]$ is bounded.
\end{lemma}

\begin{proof}
We use the integral formula \eqref{magicint}
\begin{align*}
\int_{0}^{\infty} D\lambda^{-\beta}(1+\lambda+D^{2})^{-1}(D[D,a]+[D,a]D)(1+\lambda +D^{2})^{-1}\rd \lambda
\end{align*}
and we see that the two terms behave like $\lambda^{-\beta}(1+\lambda)^{-1+\alpha}$ and are therefore integrable at $0$ and $\infty$. 
\end{proof}

\begin{lemma}
\label{absred}
Let $0\leq \alpha<\frac{1}{2}$, $a\Dom(D)\subset\Dom(\D)$ and $[D,a](1+D^{2})^{-\alpha}$ be bounded. Then for $\beta>\alpha$ the operator $[(1+D^{2})^{\frac{1}{2}-\beta},a]$ is bounded.
\end{lemma}

\begin{proof}
We prove that $(1+D^{2})^{1/2-\beta}[(1+D^{2})^{-\frac{1}{2}+\beta},a](1+D^{2})^{\frac{1}{2}-\beta}$ is bounded using the integral formula. Expansion gives us
\begin{align*}
\int_{0}^{\infty} \lambda^{-1/2+\beta}(1+D^{2})^{\frac{1}{2}-\beta}(1+\lambda+D^{2})^{-1}(D[D,a]+[D,a]D)(1+\lambda +D^{2})^{-1}(1+D^{2})^{\frac{1}{2}-\beta}\rd \lambda,
\end{align*}
and both  terms behave like $\lambda^{-\frac{1}{2}+\beta}(1+\lambda)^{-\beta}(1+\lambda)^{-\frac{1}{2}-\beta+\alpha}$, so are integrable at $0$ and $\infty$.
\end{proof}

\begin{lemma}
\label{relcpt}
Let $0\leq \alpha<\frac{1}{2}$, $a\Dom(D)\subset\Dom(\D)$ and $[D,a](1+D^{2})^{-\alpha}$ be bounded. Then for $\alpha<\beta<\frac{1}{2}$ and $a,b\in\mathcal{A}$ the operator $D(1+D^{2})^{-\beta}[(1+D^{2})^{-\frac{1}{2}+\beta}\log(1+(1+D^{2})^{\frac{1}{2}-\beta}),a]b$ is $B$-compact.
\end{lemma}

\begin{proof} 

Since $D(1+D^{2})^{-\beta}$ behaves like $(1+D^{2})^{\frac{1}{2}-\beta}$ we consider
\[
(1+D^{2})^{-\frac{1}{2}+\beta}  \log(1+(1+D^{2})^{\frac{1}{2}-\beta})=\int_{0}^{1}(1+t(1+D^{2})^{\frac{1}{2}-\beta})\,\rd t,
\]
and expand the expression
\begin{align}
\label{logexpansion}
(1+&D^{2})^{\frac{1}{2}-\beta}[(1+D^{2})^{-\frac{1}{2}+\beta}\log(1+(1+D^{2})^{\frac{1}{2}-\beta}),a]b\\
&=\int_{0}^{1} (1+D^{2})^{\frac{1}{2}-\beta}(1+t(1+D^{2}))^{-\frac{1}{2}+\beta}t[(1+D^{2})^{\frac{1}{2}-\beta},a](1+t(1+D^{2}))^{-\frac{1}{2}+\beta}b\,\rd t.
\end{align}

By Lemma \ref{absred} $[(1+D^{2})^{\frac{1}{2}-\beta},a]$ 
is a bounded operator and the operators
\[
t(1+D^{2})^{\frac{1}{2}-\beta}(1+t(1+D^{2}))^{-\frac{1}{2}+\beta}, \quad (1+t(1+D^{2}))^{-\frac{1}{2}+\beta}b,
\]
are uniformly bounded in $t$, and the latter operator is $B$-compact for $t\in (0,1]$. Thus the left hand side of \eqref{logexpansion} is the integral of a uniformly bounded function with values in the $B$-compact operators, hence it is $B$-compact.
\end{proof}

\begin{defn} 
For a higher order Kasparov module $(\mathcal{A}, X_B,D)$ and $0<\beta<\frac{1}{2}$ we define its $\beta$-\emph{logarithmic transform} $(\mathcal{A},X_B, D_{\beta,\log})$  by the self-adjoint regular operator
\[D_{\beta, \log}:=D(1+D^{2})^{-\frac{1}{2}}\log(1+(1+D^{2})^{\frac{1}{2}-\beta}).\]
\end{defn}

Recall our notation $D_{\log}:=\sgnlog(D)=D|D|^{-1}\log(1+|D|)$. We now arrive at the main result of this section.

\begin{thm} 
\label{logtrans}
Let $(\mathcal{A},X_B,D)$ be a higher order Kasparov module. Then the logarithmic transform $(\mathcal{A},X_B,D_{\log})$ is an ordinary unbounded Kasparov module representing the same $KK$-class. Moreover the commutators $[D_{\log},ab]$ are $B$-compact for $a,b\in\mathcal{A}$.
\end{thm}

\begin{proof} 
As the function
$$x\mapsto x(1+x^{2})^{-\frac{1}{2}}\log(1+(1+x^{2})^{\frac{1}{2}-\beta})-(1-2\beta)\sgnlog(x),$$
belongs to $C_0(\R)$, it follows that $D_{\beta, \log}-(1-2\beta)D_{\log}$ is an $A$-locally $B$-compact operator. It thus suffices to show that for some $\beta\in (0,1/2)$, the $\beta$-logarithmic transform $(\mathcal{A},X_B,D_{\beta,\log})$ is an ordinary unbounded Kasparov module with the commutators $[D_{\beta,\log},ab]$ being $B$-compact for $a,b\in\mathcal{A}$. Since $(\mathcal{A},X_B,D)$ is a higher order Kasparov module, there is an $0\leq \alpha<\frac{1}{2}$ such that for all $a\in \mathcal{A}$, $[D,a](1+D^{2})^{-\alpha}$ and $(1+D^{2})^{-\alpha}[D,a]$ are bounded. Take $\beta$ with $\alpha<\beta<\frac{1}{2}$.

The operator $D_{\beta,\log}$ has $A$-locally $B$-compact resolvent by construction.  As in the proof of Lemma \ref{logbdd}, since $(1+D^{2})^{-1/2}\log(1+(1+D^{2})^{\frac{1}{2}-\beta})$ is a bounded operator, $\Dom(D)$ is a core for $D_{\beta,\log}$ and each $a\in\A$ preserves this core. We thus need only prove boundedness and $B$-compactness of the relevant commutators. To this end we expand
\begin{align}
\label{moeilijk}
[D_{\beta,\log},a]&=D(1+D^{2})^{-\beta}[(1+D^{2})^{-\frac{1}{2}+\beta}\log(1+(1+D^{2})^{\frac{1}{2}-\beta},a]\\
&\label{mindermoeilijk1}
\quad\quad\quad +[D,a](1+D^{2})^{-\beta}(1+D^{2})^{-\frac{1}{2}+\beta}\log(1+(1+D^{2})^{\frac{1}{2}-\beta})\\
&\label{mindermoeilijk2}
\quad\quad\quad\quad\quad\quad +D[(1+D^{2})^{-\beta},a](1+D^{2})^{-\frac{1}{2}+\beta}\log(1+(1+D^{2})^{\frac{1}{2}-\beta}).
\end{align}
By Lemma \ref{relcpt} the first term is bounded and $B$-compact after multiplication from the left by $b$. The second term is bounded since $[D,a](1+D^{2})^{-\alpha}$ is bounded and Lemma \ref{weaklip} gives boundedness of the third term. 
For the terms \eqref{mindermoeilijk1} and \eqref{mindermoeilijk2}, multiplication from the right by $b$ makes them $B$-compact as
\[x\mapsto (1+x^{2})^{-\frac{1}{2}+\beta}\log(1+(1+x^{2})^{\frac{1}{2}-\beta}),\]
is a $C_{0}$-function. Since $\mathcal{A}$ is a $*$-algebra and $D$ is self-adjoint, we have that $[D,b^{*}]a^{*}$ is $B$-compact as well, and thus so are $a[D,b]$ and $[D,ab]$.
\end{proof}

\begin{corl} 
\label{cptcomm}
Every class in $KK_*(A,B)$ can be represented 
by an unbounded Kasparov module $(\mathcal{A},X_B,D)$ 
such that $\overline{AX_B}=X_B$ and for all 
$a\in \mathcal{A}$ the commutators 
$[D,a]$ are $B$-compact.
\end{corl}

\begin{proof}
It follows from \cite[Proposition 18.3.6]{Bl} 
and \cite[Lemma 1.4]{Kuc2}) that any class 
in $KK_*(A,B)$ can be represented by an 
unbounded Kasparov module $(\mathcal{A}_0, X_B, D_0)$ 
with $\overline{AX_B}=X_B$. Then 
\[
\mathcal{A}_{0}^{2}
:=\textnormal{span}_{\mathbb{C}}
\left\{ab: a,b\in\mathcal{A}\right\},
\] 
is a dense $*$-subalgebra of $A$. 
By Theorem \ref{logtrans} the triple
\[
(\mathcal{A}, X_B, D):=(\mathcal{A}_{0}^{2},X_B, (D_0)_{\log}),
\]
is an unbounded Kasparov module for which all 
elements $a\in\mathcal{A}$ have $B$-compact 
commutators with $D$.
\end{proof}
If one drops the requirement that $\overline{AX_B}=X_B$, 
then Corollary \ref{cptcomm} is in fact implicitly proven 
in \cite{BJ} (as noted by Kaad \cite{Karece}), but that proof does not cover the stronger 
result presented here.
\begin{corl}
Let $\gamma:S^{1}\to S^{1}$ be a 
(non-isometric) diffeomorphism. The logarithmic 
dampening of the higher order spectral triple 
\eqref{PVextmod} in Proposition \ref{PVprod} 
gives an ordinary spectral triple representing the 
class $\partial [(C(S^{1}), L^{2}(S^{1}), i\frac{\d}{\d x})]\in K^{0}(C(S^{1})\rtimes_{\gamma}\mathbb{Z})$. 
\end{corl}
\section{Review of local index formulae}
\label{sec:LIF}
The first and most important application of spectral triples
was to provide computationally tractable expressions for
the index pairing. This programme was initiated by Connes and Moscovici 
to enable them to study the transverse fundamental class of foliations, crossed products 
and more generally triangular structures on manifolds, \cite{CM}. 
The outcome was the first expression and proofs of the local index formula.

Since then, refinements and extensions have been developed 
by Higson \cite{hig} and in \cite{CGRS2,CNNR,CPRS2,CPRS3,CPRS4,RS}.
All the various statements of the local index formula rely on
two basic assumptions: smoothness and finite summability. 
As summability requires reference to the 
Schatten ideals in $\B(H)$, 
it has so far not been developed for unbounded Kasparov modules. 
Moreover, summability in the non-unital case is 
more technical (see \cite{CGRS1,CGRS2}). 
In this section we therefore restrict our attention 
to unital (twisted) spectral triples.
We remark that the discussion also extends 
to the semifinite setting (see more 
in \cite{CC,BeF,CPRS2,CPRS3,CPRS4,CGRS2}), 
but we refrain from this level of generality.
\subsection{Smoothness and summability}
Here we discuss the standard definitions of smoothness and summability.
\label{subb:smo-sum}
\begin{defn}
\label{qck} 
A unital spectral triple $(\A,\H,\D)$ is regular if for all $k\geq 1$
and all $a\in\A$ the operators $a$ and
$[\D,a]$ are in the domain of $\delta^k$, where
$\delta(T)=[|\D|,T]$ is the partial derivation on $\B(\H)$ defined
by $|\D|$. 
\end{defn}

\begin{defn}
\label{summability}
A unital spectral triple $(\A,\H,\D)$ is called finitely summable if
there is a $s>0$ such that $(i+D)^{-1}\in \mathcal{L}^s(\H)$.
If this is the case, we say that $(\A,\H,\D)$ is $s$-summable. 
If $(\A,\H,\D)$ is a finitely summable spectral triple, we call 
\begin{equation*}
p:=\inf\{s\in{\R}:{\rm Tr}((1+\D^2)^{-s/2})<\infty\}
\end{equation*}
the {\bf spectral
dimension} of $(\A,\H,\D)$.
\end{defn}

There are $C^{*}$-algebras that do not admit finitely summable spectral triples, even
when they do admit finitely summable Fredholm modules (more on this
later). We quote the following obstruction result. 

\begin{thm}[Connes, \cite{Co3}] 
\label{Connestrace}
Let $A$ be a unital $C^*$-algebra and
$(\A,\H,\D)$ a unital finitely summable spectral triple, with
$\A\subset A$ dense. Then there exists a 
tracial state $\tau:A\to\mathbb{C}$ on
$A$.
\end{thm}

This theorem was stated for ordinary spectral 
triples in \cite{Co3}, but the proof extends 
mutatis mutandis to the general compact 
higher order (and semi-finite) setting. 
More generally, if 
$(\A,\H,\D)$ is a (semifinite higher order) spectral triple with $\Tr(\mathrm{e}^{-t|D|})<\Tr(\chi_{[0,\infty)}(D))=\infty$ for all $t>0$ 
then there is a tracial state on $A$ (see \cite[Theorem 3.22]{GMR}).
One can conclude that algebras with no non-zero trace, 
such as the Cuntz algebras, do
not carry unital higher order (semi-finite) 
finitely summable spectral triples. 
In particular, the obstructions to finite summability 
of spectral triples remain when 
generalizing to higher order as well as to semi-finite 
spectral triples.
The lack of finitely summable spectral triples does not 
preclude the existence of finitely summable Fredholm 
modules (see \cite{EN, GM}), and we exploit this later.

\subsection{The index cocycle}
For any finitely summable bounded Fredholm module one can construct an associated index cocycle called its Connes-Chern character \cite[Chapter IV]{Connes}. When the Fredholm module comes from a finitely summable regular spectral triple 
$(\A,\H,\D)$ satisfying an additional meromorphicity assumption there is a
different representative of the Connes-Chern character in the finitely supported $(b,B)$-bicomplex called the residue cocycle.
Versions of this cocycle appear in the various
local index formulae in non-commutative geometry, 
\cite{CM0,CM,hig,CPRS2,CPRS3}.

Connes and Moscovici imposed the discrete dimension spectrum
assumption to prove their original version of the local index
formula, and we state this below.
The proof of the local index formula in
\cite{CPRS2,CPRS3} requires less restrictive 
hypotheses on the zeta
functions, but we will focus here on discrete dimension spectrum which implies the conditions of \cite{CPRS2,CPRS3}.
We will introduce some notation and definitions and then state the
odd local index formula using \cite{CPRS2}.

\begin{defn} 
Let $(\A,\H,\D)$ be a regular spectral triple. The
algebra $\B(\A)\subseteq \mathcal \B(\H)$ is the algebra of polynomials
generated by
 $\delta^n(a)$ and $\delta^n([\D,a])$ for $a\in\A$ and $n\geq 0.$
 A regular spectral triple $(\A,\H,\D)$ has {\bf discrete
dimension spectrum}
\index{Discrete dimension spectrum}
$\Sd\subseteq {\C}$ if $\Sd$ is a discrete set
and for all $b\in\B(\A)$ the function $\zeta_b(z):= {\rm Tr}(b(1+\D^2)^{-z})$ is
defined and holomorphic for ${\rm Re}(z)$ large, and analytically
continues to ${\C}\setminus \Sd$. We say the dimension spectrum is
{\bf simple} if this zeta function has poles of order at most one
for all $b\in\B(\A)$, {\bf finite}
if there is a $k\in{\N}$ such
that the function has poles of order at most $k$ for all
$b\in\B(\A)$ and {\bf infinite},
if it is not finite.
\end{defn}

Introduce multi-indices $\mathbbm{k}=(k_1,\dots,k_m)$, 
$k_i=0,1,2,\dots$, whose
length $m$ will always be clear from the context and let
$|\mathbbm{k}|=k_1+\cdots+k_m$. Define 
\begin{equation*} 
\alpha(\mathbbm{k})=\frac{1}{k_1!k_2!\cdots k_m!\,(k_1+1)(k_1+k_2+2)\cdots(|k|+m)}, 
\end{equation*}
and the numbers
$\tilde\sigma_{n,j}$ 
are defined by the equalities 
\begin{equation*}
\frac{\Gamma(n+z+1/2)}{\Gamma(z+1/2)}=\prod_{j=0}^{n-1}(z+j+1/2)=\sum_{j=0}^{n}z^j\tilde\sigma_{n,j}.
\end{equation*}

If $(\A,\H,\D)$ is a regular spectral triple and $T\in\B(\H)$
then $T^{(n)}$ is the $n^{th}$ iterated commutator with $\D^2$ (whenever defined),
that is, $T^{(n)}:=[\D^2, [\D^2,[\cdots,[\D^2,T]\cdots]]]$.

Now let $(\A,\H,\D)$ be a spectral triple with finite
dimension spectrum. For operators $b\in\B(\H)$ of the form
$$
b=a_0[\D,a_1]^{(k_1)}\cdots[\D,a_m]^{(k_m)}(1+\D^2)^{-m/2-|k|}
$$
we can define, for $j\in\mathbb{N}$, the functionals 
\begin{equation*} 
\tau_j(b):={\rm res}_{z=0}z^j\zeta_b(z).
\end{equation*}
The
 hypothesis of finite dimension spectrum
is clearly sufficient to define the residues. 
We adapt part of the statement of the odd local index formula
from \cite{CPRS2} to our situation.

\begin{thm}[Odd local index formula]\label{SFLIT}
Let $(\A,\H,\D)$ be an odd finitely summable regular spectral
triple with spectral dimension $p\geq 1$ and discrete, finite dimension spectrum. 
Let $P$
be the spectral projection of $\D$ corresponding to the interval
$[0,\infty)$.
Let $N=[p/2]+1$ where
$[\cdot]$ denotes the floor function (i.e.integer part), and let $u\in\A$ be unitary.
Then the index pairing can be computed by means of the formula
\begin{equation*}
\langle [(\A,\H,\D)],[u]\rangle={\rm Index}(PuP:P\H\to P\H)
=\sum_m \phi_m({\rm Ch}_m(u))
\end{equation*} 
where ${\rm Ch}_m(u)$ are the components of the Chern character of $u$ (see \cite{CPRS2}), and
for $a_0,\dots,a_m\in\A$ 
\begin{align*}
&\phi_m(a_0,\dots,a_m)=\sum_{|\mathbbm{k}|=0}^{2N-1-m}(-1)^{|\mathbbm{k}|} \alpha(k)\times\\
&\times\sum_{j=0}^{|\mathbbm{k}|+(m-1)/2}\tilde\sigma_{(|\mathbbm{k}|+(m-1)/2),j}\tau_j
\left(a_0[\D,a_1]^{(k_
1)}\cdots[\D,a_m]^{(k_m)}(1+\D^2)^{-|\mathbbm{k}|-m/2}\right).
\end{align*} 
The collection of functionals
$(\phi_m)_{m=1,{\rm odd}}^{2N-1}$ is a $(b,B)$-cocycle for $\A$. 
\end{thm}
\subsection{(Local) index theory for twisted spectral triples}
\label{subsec:ansatz}

The obstruction to finite summability expressed in Theorem \ref{Connestrace} and the dependency on finite summability in the local index formula of Theorem \ref{SFLIT} call for a 
different approach in  purely infinite situations. Theorem \ref{Connestrace} does not rule out the existence of finitely summable twisted spectral triples in the absence of a trace. 

Motivated by this issue, Moscovici gave an ansatz for a local index formula for twisted spectral triples in \cite{twistedmoscovici}. Moscovici's ansatz is a cyclic cochain in the $(b,B)$-bicomplex associated with a twisted spectral triple. Moscovici proved that his  ansatz  computes the index pairing with $K$-theory in a special case in \cite{twistedmoscovici}. In order to formulate Moscovici's ansatz, we need to adapt the notions of regularity, finite summability and dimension spectrum to the twisted setting. The notion of regularity we make use of is due to Matassa and Yuncken \cite{MY}.

\begin{defn}
Let $(\A,\H,\D,\sigma)$ be a weakly twisted spectral triple. 
\begin{itemize}
\item $(\A,\H,\D,\sigma)$ is said to be finitely summable if there is a $p>0$ such that $(i+\D)^{-1}\in \mathcal{L}^p(\H)$. In this case, we say that $(\A,\H,\D,\sigma)$ is $p$-summable.
\item $(\A,\H,\D,\sigma)$ is said to be regular if there is a $*$-algebra $\mathcal{B}\subseteq \mathbb{B}(\H)$ of bounded operators containing $\A$ and $[\D,\A]_\sigma$ to which $\sigma$ extends to a linear isomorphism $\theta:\mathcal{B}\to \mathcal{B}$ such that $\mathcal{B}$ is invariant under the derivation $\delta_\theta(b):=[|\D|,b]_\theta=|\D|b-\theta(b)|\D|$.
\end{itemize}
\end{defn}

\begin{rmk}
If $(\A,\H,\D,\sigma)$ is regular, $|\D|a-\sigma(a)|\D|=\delta_\theta(a)\in \mathcal{B}$ for $a\in \A$. In particular, regular twisted spectral triples satisfy the twisted Lipschitz condition.
\end{rmk}

In fact, Matassa-Yuncken \cite{MY} showed that a regular twisted spectral triple admits a twisted pseudo-differential calculus (see \cite[Definition 3.6]{MY}). We denote this twisted pseudo-differential calculus by $\Psi^*_{\sigma,\D}(\A)$. For full details on the twisted pseudo-differential calculus see \cite{MY}, but let us point out that $\Psi^*_{\sigma,\D}(\A)$ contains $\A$, $[\D,\A]_\sigma$, all powers of $|\D|$ and is closed under $\delta_\theta$ for a suitable linear isomorphism $\theta$ of $\Psi^*_{\sigma,\D}(\A)$ extending $\sigma$. If we can take $\theta$ to be an automorphism we say that $(\A,\H,\D,\sigma)$ is \emph{strongly regular}. 

Before going into the index theory of twisted spectral triples, 
we record some  basic results concerning regularity.

\begin{prop}
\label{strongreforsimsi}
Let $(\A,\H,\D,\sigma)$ be a twisted spectral triple with $\D$ invertible and $\sigma(a)=|\D|a|\D|^{-1}$. Then  $(\A,\H,\D,\sigma)$ is strongly regular and we can take $\Psi^*_{\sigma,\D}(\A)$ to be the filtered algebra generated by $F:=D|D|^{-1}$, $\A$, $[\D,\A]_\sigma$ and all complex powers of $|\D|$ with $\theta=\sigma$, the filtering coming from the order of the power of $|\D|$ and the one-parameter family of algebra automorphisms $(\Theta^z)_{z\in \C}$ is defined from conjugation by $|\D|^{z}$. 
\end{prop}

\begin{proof}
The definitions given in the statement imply that
$$
\delta_\sigma(a):=[|\D|,a]_\sigma=|\D|a-|\D|a|\D|^{-1}|\D|=|\D|a-|\D|a=0.
$$
In particular, the algebra generated by $F$, $\A$, $[\D,\A]_\sigma$ and all powers of $|\D|$ is invariant under $\delta_\sigma$. It follows that the filtered algebra $\Psi^*_{\sigma,\D}(\A)$ satisfies the assumptions of \cite[Definition 3.6]{MY} and forms a twisted pseudo-differential calculus.
\end{proof}

\begin{prop}
\label{satofeasytwist}
Let $(\A,\H,\D,\sigma)$ be a regular weakly twisted spectral triple with $\D$ invertible, and set $\tilde{\sigma}(a):= |D|a|\D|^{-1}$. Then $\sigma(a)-\tilde{\sigma}(a)\in \Psi^{-1}_{\sigma,\D}(\A)$. Moreover, $(\A,\H,\D,\tilde{\sigma})$ is a weakly twisted spectral triple and its saturation  $(\A_{\tilde{\sigma}},\H,\D,\tilde{\sigma})$ is a strongly regular twisted spectral triple.
\end{prop}

\begin{proof}
The first part of the result is a consequence of \cite[Lemma 3.7]{MY}. The second part of the proposition is proven as follows. The property that $\sigma(a)-\tilde{\sigma}(a)\in \Psi^{-1}_{\sigma,\D}(\A)$ ensures that $(\A,\H,\D,\tilde{\sigma})$ is a weakly twisted spectral triple satisfying the conditions of Proposition \ref{saturateprops} (see page \pageref{saturateprops}) so its saturation $(\A_{\tilde{\sigma}},\H,\D,\tilde{\sigma})$ is well defined. Strong regularity of its saturation follows from Proposition \ref{strongreforsimsi}. 
\end{proof}

\begin{rmk}
\label{satofeasytwistrmk}
In the terminology of \cite[Section 3]{MY}, $\sigma$ and $\tilde{\sigma}$ are equivalent twists. In particular, they can be equipped with related twisted pseudodifferential calculi. 
\end{rmk}

Continuing towards Moscovici's ansatz, we define dimension spectrum for regular twisted spectral triples. 

\begin{defn}
Let $(\A,\H,\D,\sigma)$ be a finitely summable regular weakly twisted spectral triple with twisted pseudo-differential calculus $\Psi^*_{\sigma,\D}(\A)$. We say that $(\A,\H,\D,\sigma)$ has discrete dimension spectrum if there is a discrete set $\Sd\subseteq \C$ such that the $\zeta$-functions 
$$
\zeta_T(z)
:=\mathrm{Tr}(T|\D|^{-2z}), \quad T\in  \Psi^*_{\sigma,\D}(\A),
$$
defined for $\mathrm{Re}(z)>>0$, have meromorphic extensions to $\C$ holomorphic outside $\Sd$. If there is an $N\in \N$ such that all poles of $\{\zeta_T: T\in \Psi^*_{\sigma,\D}(\A)\}$ are of order at most $N$ in $\Sd$, we say that $(\A,\H,\D,\sigma)$ has finite discrete dimension spectrum. If we can take $N=1$, we say that $(\A,\H,\D,\sigma)$ has simple discrete dimension spectrum.
For $T\in\Psi^*_{\sigma,\D}(\A)$ we write 
$T^{(1)}:=[\D^2,T]_{\sigma^2}$ and $T^{(k)}:=[\D^2,T^{(k-1)}]_{\sigma^2}$, and let
$\tau_j(T):=\mathrm{Res}_{z=0}z^j\Tr(T|D|^{-2z})$
\end{defn}

Moscovici introduced his ansatz for a twisted local index formula 
only for the case of simple dimension spectrum, but this is purely a matter of technical convenience. 
Adjusting Moscovici's constants as in the 
discussion of renormalisation 
in \cite[Section II.3]{CM} brings them in line with 
the renormalised formula presented in Theorem \ref{SFLIT}.

We stress that {\bf these constants do not affect 
our later arguments at all}. The crux of the issue for us is that,
for certain examples, all multilinear functionals defined using residues of zeta functions are identically zero.

\begin{defn}
\label{defnofloclalass}
Let $(\A,\H,\D,\sigma)$ be a finitely summable regular twisted spectral triple with discrete dimension spectrum. For $m,j\in \N$ and $\mathbbm{k}=(k_1,\ldots, k_m)\in \N^m$, we define the $m$-cochain $\phi_{m,j,\mathbbm{k}}$ on $\A$ by setting 
\begin{align*}
\phi_{m,j,\mathbbm{k}}(&a_0,a_1,\ldots,a_m)\\
&:=\tau_j\left(\gamma a_0[\D,\sigma^{-2k_1-1}(a_1)]_\sigma^{(k_1)}\cdots [\D,\sigma^{-2(k_1+k_2+\cdots +k_m)-1}(a_m)]_\sigma^{(k_m)}|\D|^{-2|\mathbbm{k}|-m}\right),
\end{align*}
for $a_0,a_1,\ldots, a_m\in \A$. Here $\gamma$ denotes the grading operator in the even case and $\gamma=F$ in the odd case. Adjusting the coefficients $c_{m,\mathbbm{k}}$ from \cite{twistedmoscovici} to take care of the renormalisation procedure of \cite{CM}, 
we define the $m$-cochain $\phi_m$ as 
$$
\phi_m(a_0,a_1,\ldots,a_m)
:=\sum_{\mathbbm{k}\in \N^m} (-1)^{|\mathbbm{k}|}\alpha(\mathbbm{k})
\sum_{j=0}^{|\mathbbm{k}|+(m-1)/2}\tilde{\sigma}_{(|\mathbbm{k}|+(m-1)/2),j}
\phi_{m,j,\mathbbm{k}}(a_0,a_1,\ldots,a_m).
$$
\end{defn}

The reader should note that due to the finite summability assumption, there are only finitely many non-zero terms in the sum defining $\phi_m$ and $\phi_m=0$ for $m>>0,$ sufficiently large. In \cite{twistedmoscovici}, the cochain $\phi_m$ is defined by means of a slightly different expression, which makes use of the fact that Moscovici assumes simple dimension spectrum. Our 
definition extends Moscovici's in a way that is consistent with the derivation from the twisted JLO cocycle, taking into account the
renormalisation procedure of \cite{CM}.

For $j=0,1$ the parity of the regular and finitely summable
 twisted spectral triple $(\A,\H,\D,\sigma)$ with discrete dimension spectrum,
we shall call the cochain $(\phi_{2m+j})_{m\in \N}$ 
in the $(b,B)$-bicomplex the \emph{Moscovici ansatz} 
for an index cocycle. We note that the 
cochain $(\phi_{2m+j})_{m\in \N}$ is only an ansatz, 
and it is unclear how generally this ansatz actually 
provides a cocycle, let alone an index cocycle. 
In the next section (Section \ref{thelackofloc}), 
we shall provide examples of regular finitely summable
 twisted spectral triples with finite discrete dimension spectrum 
 for which Moscovici's ansatz can not 
 represent an index cocycle.

\begin{rmk}[Computing index pairings using twisted spectral triples]

The only to date known local index formula for twisted spectral triples 
is due to Moscovici who considered twisted spectral triples arising from scaling automorphisms. Given an ordinary 
spectral triple $(\A,\H,\D)$,
a scaling automorphism is an automorphism of $\A$ implemented by a unitary $U\in\B(\H)$ such that there is a positive real number $\mu(U)$ for which $U\D U^*=\mu(U)\D$. If $\Gamma$ is a group of scaling automorphisms, $(\A\rtimes \Gamma,\H,\D,\sigma)$, with $\sigma(aU):=\mu(U)aU$ is a twisted spectral triple. 

In \cite{twistedmoscovici}, it was proven that for twisted spectral triples associated to scaling automorphisms  that 
\begin{equation}
\label{dodycpair}
\langle [(\phi_{2m+j})_{m\in \N}],{\rm Ch}(x)\rangle
=\langle [(\A,\H,\D,\sigma)],x\rangle, \quad \forall x\in K_*(A).
\end{equation}
Here the left hand side denotes the pairing of cyclic cohomology with the cyclic homology class given by the Chern character of the $K$-theory element $x$. The right hand side denotes the pairing of $K$-homology with $K$-theory. The counterexample alluded to in the preceding paragraph is found by providing a counterexample to the equality \eqref{dodycpair}.

The local index theory was worked out in detail by Ponge and Wang \cite{PW,PW2} for groups of conformal diffeomorphisms. To do this they reduced to the case of
trivial twist by choosing an invariant metric, and employed the local index formula for ordinary spectral triples.
This left open the question of whether Moscovici's ansatz computed the index.
As noted in \cite{PW,PW2}, 
there are very few examples.
\end{rmk}

\section{Vanishing of the twisted local index formula}
\label{thelackofloc}

One of the motivations to introduce twisted spectral triples is that twisting allows for the existence of Dirac operators with better spectral properties, as
discussed in the Introduction and \cite{CoM}. One sought after spectral property is finite summability, and one might hope for a local index formula for twisted spectral triples in the style of Connes-Moscovici's local index formula \cite{CM}. In \cite{twistedmoscovici}, Moscovici provided an ansatz for a cyclic cocycle that generalized Connes-Moscovici's local index cocycle. Moscovici's ansatz reproduces the index character for twisted spectral triples associated to so-called scaling automorphisms. 

In this subsection we discuss various examples showing that Moscovici's ansatz can not be extended to the case of a general finitely summable regular twisted spectral triple with finite discrete dimension spectrum. Our examples are highly regular odd twisted spectral triples pairing non-trivially with $K$-theory yet having all twisted higher residue cochains (appearing in Moscovici's ansatz) vanishing. Here ``high regularity'' means that all twisted commutators of the algebra elements with the Dirac operator in the twisted spectral triple are not just bounded but smoothing or even of finite rank -- leaving little hope for a general index formula in terms of residues of $\zeta$-functions.

\subsection{Set up and statement}
\label{subsub:statement}
The construction of our counterexamples,
and the proof of all their regularity properties follow a general pattern. The main idea is 
condensed in the following lemma. We will say that a 
Borel function $f:\R\to \R$ is polynomially bounded if there 
are constants $C$ and $N$ such that $|f(x)|\leq C(1+|x|)^N$ for all $x\in \R$.

\begin{lemma}
\label{asstoprovtauzero}
Let $(\A,\H,\D,\sigma)$ be a finitely summable regular twisted spectral triple with $\D$ invertible and $\sigma(a)=|\D|a|\D|^{-1}$. Assume that 
\begin{enumerate}
\item[a)] For any two polynomially bounded
Borel functions $f_1,f_2:\R\to \R$ and $a\in \A$ the twisted commutator $[\D,a]_{\sigma}$ preserves $\Dom(f_1(\D))$ and the operator $f_1(\D)[\D,a]_\sigma f_2(\D)$ extends to a trace class operator on $\H$.
\item[b)] There is a discrete subset $\mathcal{P}\subseteq \C$ such that for any $a\in \A$, the $\zeta$-functions $\zeta_a(z):=\mathrm{Tr}(a|\D|^{-2z})$ and $\zeta_{Fa}(z):=\mathrm{Tr}(Fa|\D|^{-2z})$ extend to meromorphic functions in $\C$ holomorphic outside $\mathcal{P}$ with poles in $\mathcal{P}$ of uniformly bounded order.
\end{enumerate}
Then $(\A,\H,\D,\sigma)$ is a strongly regular finitely summable twisted spectral triple with finite discrete dimension spectrum. Moreover, for $m>0$, $a_0,a_1,\ldots,a_m\in\A$ 
and $\mathbbm{k}\in \N^m$, the function 
\begin{equation}
\label{agnadaddaka}
z\mapsto \Tr\left(\gamma a_0[\D,\sigma^{-2k_1-1}(a_1)]_\sigma^{(k_1)}\cdots [\D,\sigma^{-2(k_1+k_2+\cdots +k_m)-1}(a_m)]_\sigma^{(k_m)}|\D|^{-2|\mathbbm{k}|-m-2z}\right)
\end{equation}
extends {\bf holomorphically} to 
all of $\C$. In particular, $\phi_m=0$ for $m>0$.
\end{lemma}

\begin{proof}
Note that if $\sigma(a)=|\D|a|\D|^{-1}$, then 
\begin{equation}
\label{twistide}
[\D,a]_\sigma=\D a-|\D|a|\D|^{-1} \D=|\D|(Fa-aF)=|\D|[F,a].
\end{equation}

It follows from Proposition \ref{strongreforsimsi} 
that $(\A,\H,\D,\sigma)$ is strongly regular and 
that $\Psi^*_{\sigma,\D}(\A)$ is generated by 
$F$, $\A$, $[\D,\A]_\sigma$ and all powers of $|\D|$. 
We first show that $(\A,\H,\D,\sigma)$ has discrete 
dimension spectrum. It suffices to show that 
$\zeta_B(z)$ extends  to a meromorphic function in 
$\C$ holomorphic outside $\mathcal{P}$ with possible 
poles in $\mathcal{P}$ when $B$ is a product of elements 
from $\{F\}\cup\A\cup [\D,\A]_\sigma\cup\{|\D|^k: \; k\in \Z\}$. 
It follows from Assumption a) that if $B$ contains a 
factor from $[\D,\A]_\sigma$ then $\zeta_B(z)$ extends holomorphically to $\C$.

Modulo terms with 
factors in $[\D,\A]_\sigma$, we can write $B=F^j|\D|^za|\D|^w$ 
for some $j\in \Z$, $z,w\in \C$ since 
$\sigma(a)=|\D|a|\D|^{-1}$ and using \eqref{twistide}. 
It follows from Assumption b) that $\zeta_B(z)$ extends  
to a meromorphic function in $\C$ holomorphic 
outside $\mathcal{P}$ with possible poles in $\mathcal{P}$. 

It remains to show that for $m>0$, 
$a_0,a_1,\ldots,a_m\in\A$ and $\mathbbm{k}\in \N^m$, 
the function in Equation \eqref{agnadaddaka} is holomorphic for all $z\in \C$. Indeed, since $m>0$ we can for 
any $s\in \R$ write the expression 
$$
\gamma a_0[\D,\sigma^{-2k_1-1}(a_1)]_\sigma^{(k_1)}\cdots [\D,\sigma^{-2(k_1+k_2+\cdots +k_m)-1}(a_m)]_\sigma^{(k_m)}|\D|^{-2|\mathbbm{k}|-m-2z},
$$ 
as a sum of elements of the form 
$$
\gamma a_0f_0(\D)[\D,b_1]_\sigma f_1(\D)\cdots f_{m-1}(\D)[\D,b_m]_\sigma f_m(\D)|\D|^{-s-2z},
$$
for some elements $b_1,\ldots, b_m\in \A$ 
and  polynomially bounded Borel functions 
$f_0,f_1,\ldots, f_m:\R\to \R$. By Assumption b), 
there is a constant $s_0\in \R$ independent of $s$ 
such that the function
$$
z\mapsto \mathrm{Tr}\left(\gamma a_0f_0(\D)[\D,b_1]_\sigma f_1(\D)\cdots f_{m-1}(\D)[\D,b_m]_\sigma f_m(\D)|\D|^{-s-2z}\right)
$$ 
is holomorphic for $\mathrm{Re}(z)>s_0-s$. Therefore, the function in Equation \eqref{agnadaddaka} is holomorphic for $\mathrm{Re}(z)>s_0-s$, and since $s$ is arbitrary the function in Equation \eqref{agnadaddaka} extends holomorphically to all of $\C$.
We deduce that $\phi_{m,\mathbbm{k}}(a_0,a_1,\ldots,a_m)=0$ for all  $m>0$, 
$a_0,a_1,\ldots,a_m\in\A$ and $\mathbbm{k}\in \N^m$
and so $\phi_m$
is zero.
\end{proof}

The assumptions in Lemma \ref{asstoprovtauzero} 
are very strong and may seem unrealistic. Nevertheless,
we will construct examples of such twisted spectral triples below. The main ingredient comes from spectral triples with positive spectral dimension whose associated Fredholm modules are $0$-dimensional. The equality \eqref{dodycpair} is disproved by the following theorem.

\begin{thm}
\label{counterexthm}
Let $(\A,\H,\D,\sigma)$ be an odd twisted spectral triple satisfying all the assumptions of Lemma \ref{asstoprovtauzero} and 
such that for some $x\in K_1(A)$,
$$
\langle [(\A,\H,\D,\sigma)],x\rangle\neq 0.
$$
Then $(\A,\H,\D,\sigma)$ is a strongly 
regular finitely summable twisted spectral triple 
with finite discrete dimension spectrum for which the equality \eqref{dodycpair} fails.
\end{thm}

\begin{proof}
In the odd case we always have $m>0$, 
and so the cochain provided by Moscovici's ansatz is zero
by Lemma \ref{asstoprovtauzero}.
Hence it  can not compute the non-zero pairing.
\end{proof}

So our task is to find a twisted spectral triple pairing non-trivially with $K$-theory and satisfying the assumptions
of Lemma \ref{asstoprovtauzero}. We shall provide three examples where this phenomena occurs.
The examples that matter most in this context are the ones
that sidestep Connes' obstruction for finite summability. The construction relies on manipulating spectral triples  $(\A,\H,D)$  for which commutators
$[D|D|^{-1},a]$ with the phase are smoothing or even finite rank.

\subsection{The good}
\label{thegoodsubsection}
In this subsection and the subsequent two we construct counterexamples to the equality \eqref{dodycpair} consisting of finitely summable twisted spectral triples for purely infinite algebras.
 In the current section we present a simple, commutative counterexample, in which essentially all phenomena can be observed.

Consider the spectral triple for the circle
$$
\left(C^{\infty}(S^{1}), L^{2}(S^{1}), D=-i\frac{\d}{\d x}+P_0\right).
$$
Here $P_0$ denotes the orthogonal projection onto the space of constant functions. 
Let $F$ denote the phase of $-i\frac{\d}{\d x}+P_0$ and $z\in C(S^1)$ the coordinate function $z:S^1\to \C$. 
One readily verifies that $[F,z]=2P_0z$ is a smoothing operator of rank $1$ and norm $2$. 
It follows that commutators with the phase $F$ of 
$D=-i\frac{\d}{\d x}+P_0$ are finite rank for polynomials in $(z,\bar{z})$ 
and smoothing for
$C^\infty$-functions. 
Consider the twist given by
$$
\sigma(a)=|D|a|D|^{-1}.
$$
In the first instance we only obtain a weakly twisted spectral triple
but we can get a twisted spectral triple by taking the saturation
$\A$ as the $*$-algebra generated by $\cup_k\sigma^k(C^\infty(S^1))$. 
We know that (the restriction to $C^\infty(S^1)$ of) this twisted spectral triple represents the class $[-i\d/\d x]$ in $K$-homology, and so pairs non-trivially with $K_1(C(S^1))$. The twisted spectral triple $(\A,L^2(S^1),D,\sigma)$ is readily seen to be finitely summable. The fact that $\mathcal{A}$ is contained in the classical order zero pseudodifferential operators on $S^1$ guarantees that it is regular and has simple discrete dimension spectrum. 

Then $[D,a]_\sigma=|D|[F,a]$ is smoothing for $a\in C^\infty(S^1)$, and likewise if
we consider $\sigma^k(a)$. A brief calculation shows that for $a_{0},a_{1}\in\A$ we have
\begin{align*}
\Tr(a_0[D,\sigma^{-1}(a_1)]_\sigma |D|^{-(1+2z)})
&=\Tr(a_0[F,a_1]|D|^{-2z}),
\end{align*}
and the right hand side is holomorphic for all $z\in\C$ when 
$a_1\in C^\infty(S^1)$. Lemma \ref{asstoprovtauzero} shows that similar
comments apply for the other terms 
$$
\Tr(a_0[D,\sigma^{-1}(a_1)]_\sigma^{(k_1)} 
|D|^{-(2k_1+1+2z)}),
$$
appearing in the twisted local index formula. Hence the twisted 
local index formula can not compute the index pairing
as all the residue functionals will vanish identically.

\begin{thm}
\label{thm:the-good}
Let $(\A,L^2(S^1),D,\sigma)$ denote the regular finitely summable twisted
spectral triple with simple discrete dimension spectrum obtained from saturating the weakly twisted spectral triple $(C^\infty(S^1),L^2(S^1), D, \sigma)$ with $D:=-i\frac{\d}{\d x}+P_0$ and $\sigma(a):= |D|a|D|^{-1}$. The twisted spectral triple $(\A,L^2(S^1),D,\sigma)$ pairs non-trivially with $K_1(\A)$ but
the cochain $(\phi_m)$ provided by Moscovici's ansatz is the zero cochain.
Hence the equality \eqref{dodycpair} does not hold.
\end{thm}

One could argue that this counterexample is artificial, and introduces a twist where none is needed. After all, $C^{\infty}(S^{1})$ admits well-behaved, 
finitely summable spectral triples for which the Connes-Moscovici local index formula holds. In the next two sections we consider algebras for which no finitely summable spectral triples exist, and we construct finitely summable twisted spectral triples for which the twisted local index formala fails. While the details become more complicated, the essential story remains the same.

\subsection{The bad}

In this subsection we will 
construct finitely summable spectral triples on a class of $C^*$-algebras 
that in general do not admit finitely summable spectral triples. 

Consider a discrete subgroup $\Gamma\subseteq SU(1,1)$. The Lie group 
$$
SU(1,1):=\left\{\gamma=
\begin{pmatrix} a& b\\\bar{b}&\bar{a}\end{pmatrix}\in M_2(\C): 
|a|^2-|b|^2=1\right\},
$$
acts on the circle $S^1$ by M\"obius transformations  
$$
\gamma(z):=\frac{az+b}{\bar{b}z+\bar{a}}.
$$
We consider the algebra 
$\mathcal{A}_0:=C^\infty(S^1)\rtimes^{\rm alg} \Gamma$. As in example \ref{circleexdef}, 
we can realize $\mathcal{A}_0$ as a $*$-algebra of operators on 
$L^2(S^1)$ via the covariant representation 
\begin{equation}
\label{unifuchsrep}
\pi(f)\phi(z):=f(z)\phi(z), \, f\in C^\infty(S^1),\quad \pi(\gamma)\phi(z):=|\gamma'(z)|^{-\frac{1}{2}}\phi (\gamma z), \, \gamma\in\Gamma,
\end{equation}
of the $C^{*}$-dynamical system $(C(S^{1}),\Gamma)$, as in Equation \eqref{unirep}. 
In the case that $\Gamma$ is a nonelementary Fuchsian group of the first kind,
the crossed product $C(S^1)\rtimes_r \Gamma$ is purely infinite (see \cite[Proposition 3.1]{delaroche} and \cite[Lemma 3.8]{Lott}) 
and does not admit any finitely summable spectral triples. 

We consider the self-adjoint elliptic first order pseudodifferential operator 
$D:=-i\frac{\mathrm{d}}{\mathrm{d}x}+P_0$ on $L^2(S^1)$ as in Subsection \ref{thegoodsubsection}. 
The classical order zero pseudodifferential operators on 
$S^1$ will be denoted by $\Psi^0_{\rm cl}(S^1)$ 
and its algebraic crossed product with $\Gamma$ by 
$\Psi^0_{\rm cl}(S^1)\rtimes^{\rm alg} \Gamma$. 
We define the regular automorphism $\sigma$ of 
$\Psi^0_{\rm cl}(S^1)\rtimes^{\rm alg}\Gamma$ as 
$\sigma(a):=|D|a|D|^{-1}$. Since 
$\mathcal{A}_0\subseteq \Psi^0_{\rm cl}(S^1)\rtimes^{\rm alg} \Gamma$, 
we have that $\mathcal{A}_0\subseteq \cap_{k\in \Z} \Dom(\sigma^k)$. 
In particular, we can define the saturation
$$
\A:=(\A_0)_\sigma \subseteq \Psi^0_{\rm cl}(S^1)\rtimes^{\rm alg} \Gamma.
$$
We let $\Psi^{-\infty}(S^1)$ denote the $*$-algebra of 
smoothing operators on $L^2(S^1)$, that is, operators 
with Schwartz kernel in $C^\infty(S^1\times S^1)$.

\begin{prop}
\label{firstproperties}
The collection $(\A,L^2(S^1),D,\sigma)$ is a finitely summable regular twisted spectral triple such that for any $a\in \A$, the twisted commutator $[D,a]_\sigma\in \Psi^{-\infty}(S^1)$. Moreover, $\Psi^*_{\sigma,D}(\A)$ can be taken to be generated by $\A$, $\Psi^{-\infty}(S^1)$ and all complex powers of $|D|$. 
\end{prop}

\begin{proof}
We start by showing that $[D,a]_\sigma\in \Psi^{-\infty}(S^1)$ for all $a\in \A$. Note that for $a\in \mathcal{A}$, 
$$[D,a]_\sigma=|D|[F,a].$$
It therefore suffices to show that $[F,a]\in \Psi^{-\infty}(S^1)$ for $a\in \A$. Since $[F,\sigma^k(a)]=\sigma^k([F,a])$, the Leibniz rule implies that it suffices to show that $[F,a]\in \Psi^{-\infty}(S^1)$ for $a\in \A_0$. We note that the (non-unitarised)  action $\gamma\phi(z):=\phi(\gamma z)$ of $SU(1,1)$ on $L^{2}(S^{1})$ preserves 
$$\mathrm{im}(F+1)=H^2(S^1)=\{f\in L^2(S^1): f\mbox{  extends to a holomorphic function on the disc}\}.$$
Thus, since $z\mapsto |\gamma'(z)|^{-\frac{1}{2}}\in C^{\infty}(S^{1})$, Equation \eqref{unifuchsrep} and the Leibniz rule imply that it suffices to show that $[F,a]\in \Psi^{-\infty}(S^1)$ for $a\in C^\infty(S^1)$, which we have already noted to be true.

Since $D$ is an elliptic first order pseudodifferential operator, $(i\pm D)^{-1}\in \mathcal{L}^p(L^2(S^1))$ for all $p>1$. We can conclude that $(\A,L^2(S^1),D,\sigma)$ is a finitely summable twisted spectral triple. It is regular and admits the prescribed twisted pseudodifferential calculus by Proposition \ref{strongreforsimsi} (see page \pageref{strongreforsimsi}).
\end{proof}

\begin{prop}
The twisted spectral triple $(\A,L^2(S^1),D,\sigma)$ satisfies all the assumptions of Lemma \ref{asstoprovtauzero} (see page \pageref{asstoprovtauzero}). Moreover, $(\A,L^2(S^1),D,\sigma)$ has simple discrete dimension spectrum $\mathrm{Sd}:=1-\frac{1}{4}\N$.
\end{prop}

\begin{proof}
We begin by showing that for any two polynomially bounded Borel functions $f_1,f_2:\R\to \R$ we have that $[D,a]_\sigma$ preserves $\Dom(f_1(D))$ and $f_1(D)[D,a]_\sigma f_2(D)$ extends to a trace class operator on $L^2(S^1)$. First, we note that since $f_1$ and $f_2$ are polynomially bounded, there is an $N\geq 0$ such that $|f_1(x)|\leq C_1(1+|x|)^N$ and $|f_2(x)|\leq C_2(1+|x|)^N$ for some $C_1,C_2\geq 0$. In particular, 
\[\Dom(|D|^{N})\subset \Dom(f_{i}(D)), \quad i=1,2.\]
Since $[D,a]_{\sigma}\in \Psi^{-\infty}(S^{1})$ it holds that $[D,a]_{\sigma}:L^{2}(S^{1})\to \bigcap_{s\geq 0} \Dom(|D|^{s}),$ and it follows that
\[(1+|D|)^{s}f_{1}(D)[D,a]_{\sigma}f_{2}(D),\]
extends to a bounded operator for all $s\geq 0$. Since $(1+|D|)^{-s}$ is a trace class operator for $s>1$, it follows that
\[f_{1}(D)[D,a]_{\sigma}f_{2}(D)=(1+|D|)^{-s}(1+|D|)^{s}f_{1}(D)[D,a]_{\sigma}f_{2}(D),\]
is of trace class.

Next we show that for any $a\in \A$, the $\zeta$-functions $\zeta_a$ and $\zeta_{Fa}$ extend meromorphically to functions on $\C$ that are holomorphic outside $1-\frac{1}{4}\N$ with at most simple poles in $1-\frac{1}{4}\N$. It was shown in \cite[Subsection 3.2]{twistedmoscovici} that for $A\in \Psi^0_{\rm cl}(S^1)\rtimes^{\rm alg} \Gamma$ the function $\zeta_A(z):=\mathrm{Tr}(A|D|^{-2z})$ extends meromorphically to a function on $\C$ that is holomorphic outside $1-\frac{1}{4}\N$ with at most simple poles in $1-\frac{1}{4}\N$. We can now conclude the desired result from noting that $\A+F\A \subseteq \Psi^0_{\rm cl}(S^1)\rtimes^{\rm alg} \Gamma$. 
\end{proof}

Consider the class $x\in K_1(\A)$ of the unitary $u\in C^\infty(S^1)\subseteq \A$ defined by $u(z):=z$. By the index theorem for Toeplitz operators, $$
\langle [(\A,\H,\D,\sigma)],x\rangle=-1,
$$
Using Theorem \ref{counterexthm} we can now conclude the following theorem.

\begin{thm}
\label{badcounterexthm}
Let $(\A,L^2(S^1),D,\sigma)$ denote the regular finitely summable twisted
spectral triple with simple discrete dimension spectrum obtained from saturating the weakly twisted spectral triple $(C^\infty(S^1)\rtimes^{\rm alg}\Gamma,L^2(S^1),D,\sigma)$ with $D:=-i\frac{\d}{\d x}+P_0$ and $\sigma(a):= |D|a|D|^{-1}$. The twisted spectral triple $(\A,L^2(S^1),D,\sigma)$ pairs non-trivially with $K_1(\A)$ but
the cochain $(\phi_m)$ provided by Moscovici's ansatz is the zero cochain.
Hence the equality \eqref{dodycpair} does not hold.
\end{thm}

We note that the weakly twisted spectral triple $(C^\infty(S^1)\rtimes^{\rm alg}\Gamma, L^2(S^1),D,\sigma)$ whose saturation appears in Theorem \ref{badcounterexthm} is in fact the exponentiation (see page \pageref{expdefnein} of the Introduction) of the logarithmic dampening of a twisted spectral triple on $C^\infty(S^1)\rtimes^{\rm alg}\Gamma$ constructed as in Example \ref{circleexdef} (see page \pageref{circleexdef}). This coincidence indicates that the study of index cocycles associated to twisted spectral triples is highly sensitive to the choice of twist.

\subsection{The ugly}
\label{subsec:counterexample}

The algebras underlying our 
final family of counterexamples are saturations of crossed products
arising from the  free groups $\mathbb{F}_d$ acting
on the boundaries $\partial\mathbb{F}_{d}$ of their Cayley
graphs. Here $d>1$. The action of $\mathbb{F}$ on $\partial \mathbb{F}_d$ is amenable, so there is no distinction between the
full and reduced crossed products.

In order to obtain a twisted spectral triple for the purely infinite 
crossed product $C(\partial\mathbb{F}_{d})\rtimes\mathbb{F}_{d}$, 
we utilise the
fact that it is isomorphic to an explicit Cuntz-Krieger algebra.
We then exploit recent advances in the construction of spectral triples
and $K$-homology classes for these algebras  \cite{GM,GM2,GMR}. While these spectral
triples are not
finitely summable, commutators with the phase are again finite rank.
As in the previous two examples, through exponentiation, we construct finitely summable regular weakly twisted
spectral triples for $C(\partial\mathbb{F}_{d})\rtimes\mathbb{F}_{d}$ whose saturation satisfies the hypotheses of 
Lemma \ref{asstoprovtauzero}.

\subsubsection{The generators and relations picture}
We begin our exposition with background on Cuntz-Krieger algebras and their spectral triples.
Let $N\in \N_{>0}$ and $\pmb{A}=(A_{ij})_{i,j=1}^N$ denote an $N\times N$-matrix of $0$'s and $1$'s. For simplicity we assume that no row or column is $0$. The associated Cuntz-Krieger algebra $\mathcal{O}_{\pmb{A}}$ is defined as the universal unital $C^*$-algebra generated by $N$ partial isometries $S_1,S_2,\ldots, S_N$ satisfying the relations
$$S_i^*S_j=\delta_{ij}\sum_{k=1}^N A_{jk}S_kS_k^*\quad\mbox{and}\quad \sum_{k=1}^N S_kS_k^*=1.$$
Cuntz-Krieger algebras are nuclear. If $\pmb{A}$ is a primitive matrix, the $C^*$-algebra $\mathcal{O}_{\pmb{A}}$ is simple and purely infinite. In particular, if $\pmb{A}$ is primitive, there are no traces on $\mathcal{O}_{\pmb{A}}$ and a spectral triple on $\mathcal{O}_{\pmb{A}}$ is never finitely summable (by Theorem \ref{Connestrace} on page \pageref{Connestrace}, see also \cite{Co3}). The $K$-theory and $K$-homology of Cuntz-Krieger algebras have been computed (see for instance \cite{kaminkerputnam,Rae})
$$
K_*(\mathcal{O}_{\pmb{A}})\cong \begin{cases} 
\coker(1-\pmb{A}^T), \;&*=0,\\
\ker(1-\pmb{A}^T), \; &*=1,\end{cases},\quad K^{*+1}(\mathcal{O}_{\pmb{A}})\cong 
\begin{cases} 
\coker(1-\pmb{A}), \;&*=0,\\
\ker(1-\pmb{A}), \; &*=1.\end{cases}
$$
Here we consider $\pmb{A}$ to be a matrix acting 
on $\Z^N$ and $\coker(1-\pmb{A})=\Z^N/(1-\pmb{A})\Z^N$. 
By \cite{kaminkerputnam} that the index pairing 
$K_*(\mathcal{O}_{\pmb{A}})\times K^*(\mathcal{O}_{\pmb{A}})\to \Z$ 
(under the isomorphisms above) coincides with the pairing $\coker(1-\pmb{A}^T)\times \ker(1-\pmb{A})\to \Z$ 
induced from the inner product on $\Z^N$. 

\subsubsection{The shift space and groupoid picture}

If $\mu\in \{1,\ldots,N\}^k$ we call $\mu$ a word of length $|\mu|:=k$ in the alphabet $\{1,\ldots,N\}$. If the word $\mu=\mu_1\cdots \mu_k$ satisfies that $A_{\mu_j,\mu_{j+1}}=1$ for $j=1,\ldots, k-1$ we say that $\mu$ is {\bf admissible}. We define the empty word $\emptyset$ to be a word of length $0$ and we define it to be admissible. We introduce the notation $\mathcal{V}_{\pmb{A},k}$ for the set of all admissible words of length $k$ and $\mathcal{V}_{\pmb{A}}:=\cup_{k\in \N}\mathcal{V}_{\pmb{A},k}$. Similarly, we can consider infinite words $x=x_1x_2\cdots \in \{1,\ldots,N\}^{\N_{>0}}$. The set of infinite admissible words is denoted by $\Omega_{\pmb{A}}$ and is topologized as a compact Hausdorff space by its subspace topology $\Omega_{\pmb{A}}\subseteq \{1,\ldots,N\}^{\N_{>0}}$. The space $\Omega_{\pmb{A}}$ is totally disconnected. If $\pmb{A}$ is a primitive matrix, $\Omega_{\pmb{A}}$ contains no isolated point and $\Omega_{\pmb{A}}$ is a Cantor space.

For $\mu=\mu_1\cdots \mu_k\in \mathcal{V}_{\pmb{A}}$ we define $S_\mu:=S_{\mu_1}\cdots S_{\mu_k}$. The linear span $\mathcal{A}_0$ of the subset $\{S_\mu S_\nu^*: \mu,\nu\in \mathcal{V}_{\pmb{A}}\}$ is a dense $*$-subalgebra of $\mathcal{O}_{\pmb{A}}$. The linear span of the subset $\{S_\mu S_\mu^*: \mu\in \mathcal{V}_{\pmb{A}}\}$ is an abelian $*$-algebra whose closure is a maximally abelian subalgebra of $\mathcal{O}_{\pmb{A}}$ isomorphic to $C(\Omega_{\pmb{A}})$ when identifying $S_\mu S_\mu^*$ with the characteristic function $\chi_{C_\mu}\in C(\Omega_{\pmb{A}})$ of the cylinder set 
$$
C_\mu:=\{x=x_1x_2\cdots \in \Omega_{\pmb{A}}: x_1x_2\cdots x_k=\mu\}.$$
A more geometric approach to Cuntz-Krieger algebras stems from a description of $\mathcal{O}_{\pmb{A}}$ as a groupoid $C^*$-algebra over $\Omega_{\pmb{A}}$. The groupoid picture is useful for geometric constructions complementing the computational virtues of the description in terms of the generators $S_1,\ldots, S_N$. The space $\Omega_{\pmb{A}}$ carries a local homeomorphism 
$$\sigma_{\pmb{A}}:\Omega_{\pmb{A}}\to \Omega_{\pmb{A}}, \quad \sigma_{\pmb{A}}(x_1x_2x_3\cdots):=x_2x_3\cdots.$$
We define the groupoid $\mathcal{G}_{\pmb{A}}$ over $\Omega_{\pmb{A}}$ by
$$
\mathcal{G}_{\pmb{A}}:=\{(x,n,y)\in \Omega_{\pmb{A}}\times \Z\times \Omega_{\pmb{A}}:\; \exists k\ \mbox{such that}\  \sigma_{\pmb{A}}^{n+k}(x)=\sigma_{\pmb{A}}(y)\},
$$
with domain map $d(x,n,y):=y$, range mapping $r(x,n,y):=x$, unit $e(x)=(x,0,x)$ and composition $(x,n,y)(y,m,z):=(x,n+m,z)$. In the definition, it is implicit that $k, n+k\geq 0$. We also define the following integer valued functions on $\mathcal{G}_{\pmb{A}}$:
\begin{equation}
c_{\pmb{A}}(x,n,y):=n\quad\mbox{and}\quad 
\kappa_{\pmb{A}}(x,n,y):=\min\big\{k\geq \max(0,-n):\;\sigma_{\pmb{A}}^{n+k}(x)=\sigma_{\pmb{A}}^k(y)\big\}.
\end{equation}
We topologize $\mathcal{G}_{\pmb{A}}$ by declaring $d$ and $r$ to be local homeomorphisms and $c_{\pmb{A}}$ and $\kappa_{\pmb{A}}$ to be continuous. In this topology, $\mathcal{G}_{\pmb{A}}$ is by definition etale. We define a clopen basis for the topology indexed by $\mu,\nu\in \mathcal{V}_{\pmb{A}}$ as 
$$
C_{\mu,\nu}:=\{(x,n,y): \; n=|\mu|-|\nu|,\ x\in C_\mu, \; y\in C_\nu,\; \sigma_{\pmb{A}}^{|\mu|}(x)=\sigma_{\pmb{A}}^{|\nu|}(y)\}.
$$
The mapping $\mathcal{O}_{\pmb{A}}\to C^*(\mathcal{G}_{\pmb{A}})$ defined by $S_\mu S_\nu^*\mapsto \chi_{C_{\mu,\nu}}$ is a $*$-isomorphism. The image of the $*$-algebra $\A_0$ generated by $S_1,S_2,\ldots, S_N$ coincides with the space of compactly supported locally constant functions on $\mathcal{G}_{\pmb{A}}$ denoted by $C^\infty_c(\mathcal{G}_{\pmb{A}})$.

\subsubsection{The construction of spectral triples for Cuntz-Krieger algebras}
\label{subsub:spec-trip}

Since $\mathcal{G}_{\pmb{A}}$ is an \'{e}tale groupoid over $\Omega_{\pmb{A}}$, there is a conditional expectation $\Phi:\mathcal{O}_{\pmb{A}}\to C(\Omega_{\pmb{A}})$. Explicitly, it is defined from the property $\Phi(S_\mu S_\nu^*):=\delta_{\mu,\nu}S_\mu S_\mu^*$. We let $\Xi_{\pmb{A}}$ denote the completion of $\mathcal{O}_{\pmb{A}}$ as a right $C(\Omega_{\pmb{A}})$-Hilbert $C^*$-module in the inner product defined from $\Phi$. It was shown in \cite{GMR} that $\Xi_{\pmb{A}}$ decomposes as a direct sum of the finitely generated projective $C(\Omega_{\pmb{A}})$-submodules $\Xi_{n,k}:=C(c_{\pmb{A}}^{-1}(n)\cap \kappa_{\pmb{A}}^{-1}(k))\subseteq \Xi_{\pmb{A}}$. Moreover, we define a self-adjoint regular operator $\D_\psi$ densely on $\Xi_{\pmb{A}}$ by for $f\in C_c(\mathcal{G}_{\pmb{A}})$ setting
$$
\D_\psi f(x,n,y):= \psi(n,\kappa_{\pmb{A}}(x,n,y))f(x,n,y), \quad\mbox{where}\quad \psi(n,k):=\begin{cases}
n, \; &k=0,\\
-|n|-k,\; &k>0. \end{cases}.
$$
Note that if $k=0$, then $n\geq 0$ is automatic from 
$c_{\pmb{A}}+\kappa_{\pmb{A}}\geq 0$. Equivalently, letting $p_{n,k}$ denote the orthogonal projection onto $\Xi_{n,k}$, $\D_\psi$ is the closure of the densely defined operator $\sum_{n,k} \psi(n,k)p_{n,k}$. Note that $\D_\psi$ is invertible on the orthogonal complement of $\Xi_{(0,0)}$ since $\psi(n,k)\neq 0$ for all $n,k$ with equality if and only if $n=k=0$.

It was proven in \cite[Section 5]{GM} that $(\A_0,\Xi_{\pmb{A}},\D_{\psi})$ is an unbounded $(\mathcal{O}_{\pmb{A}},C(\Omega_{\pmb{A}}))$-Kasparov module. 

Recall that $\mathcal{A}_0$ is the $*$-algebra generated by the generators $S_1,\ldots, S_N$. We identify $\mathcal{A}_0=C^\infty_c(\mathcal{G}_{\pmb{A}})$ -- the space of compactly supported locally constant functions. 
Let us define the relevant spectral triple on $\mathcal{O}_{\pmb{A}}$. Pick a point $t\in \Omega_{\pmb{A}}$ and define the discrete set $\mathcal{V}_t:=d^{-1}(t)$. We identify 
\begin{equation}
\mathcal{V}_t:=\{(x,n)\in \Omega_{\pmb{A}}\times \Z: \; \exists k\ \mbox{such that}\  \sigma_{\pmb{A}}^{n+k}(x)=\sigma_{\pmb{A}}^k(t)\}.
\label{eq:tee}
\end{equation}
We can consider $\mathcal{V}_t$ as a subset of $\mathcal{V}_{\pmb{A}}\times \Z$ via the embedding $(x,n)\mapsto (x_1\cdots x_{n+\kappa_{\pmb{A}}(x,n,t)}, n)$. If $n+\kappa_{\pmb{A}}(x,n,t)=0$, we interpret $x_1\cdots x_{n+\kappa_{\pmb{A}}(x,n,t)}$ as the empty word. We define the operator $\D_t$ on $\ell^2(\mathcal{V}_t)$ as 
$$
\D_t\delta_{(x,n)}:=\psi(n,\kappa_{\pmb{A}}(x,n,t))\,\delta_{(x,n)}.
$$
It is proved in \cite{GM} that 
the collection $(\mathcal{A}_0,\ell^2(\mathcal{V}_t),\D_t)$ is a spectral triple coinciding with the unbounded Kasparov product of the unbounded $(\mathcal{O}_{\pmb{A}},C(\Omega_{\pmb{A}}))$-Kasparov module $(\A_0,\Xi_{\pmb{A}},\D,\psi)$ with the representation $\pi_t:C(\Omega_{\pmb{A}})\to \C$ given by point evaluation in $t$.

\begin{prop}
\label{propoflione}
Let $\pmb{A}=(A_{ij})_{i,j=1}^N$ be an $N\times N$-matrix of $0$'s and $1$'s with no row or column being $0$. Take $t\in \Omega_{\pmb{A}}$. The odd spectral triple  $(C^\infty_c(\mathcal{G}_{\pmb{A}}),\ell^2(\mathcal{V}_t),\D_t)$ constructed in Subsubsection \ref{subsub:spec-trip} above satisfies the following properties:
\begin{enumerate}
\item $(C^\infty_c(\mathcal{G}_{\pmb{A}}),\ell^2(\mathcal{V}_t),\D_t)$ is $\mathrm{Li}_1$-summable, i.e. $\mathrm{Tr}(\mathrm{e}^{-s|\D_t|})$ is finite for large enough $s>>0$;
\item The phase $F_t:=2\chi_{[0,\infty)}(\D_t)-1$ is such that for any two  Borel functions $f_1,f_2:\R\to \R$ and $a\in C^\infty_c(\mathcal{G}_{\pmb{A}})$ the commutator $[F_t,a]$ preserves $\Dom(f_1(\D_t))$ and the operator 
$$
f_1(\D_t)[F_t,a]f_2(\D_t),
$$
extends to a trace class operator on $\ell^2(\mathcal{V}_t)$;
\item Under the isomorphism $K^1(\mathcal{O}_{\pmb{A}})\cong \Z^N/(1-\pmb{A})\Z^N$ the class of $[(\mathcal{A}_0,\ell^2(\mathcal{V}_t),\D_t)]$ is mapped to the class $\delta_j\mod (1-\pmb{A})\Z^N$, where $j$ is the first letter of $t$ and $\delta_j$ denotes the $j$'th basis vector in $\Z^N$.
\end{enumerate}
\end{prop}

\begin{proof}
We start by proving item 1): $\mathrm{Li}_1$-summability. We compute that 
\begin{align*}
\mathrm{Tr}(\mathrm{e}^{-s|\D_t|})&=\sum_{(x,n)\in \mathcal{V}_t} \mathrm{e}^{-s|\psi(n,\kappa_{\pmb{A}}(x,n,t)|}=\\
&=\sum_{n\in \Z}\sum_{k=\max(0,-n)}^\infty \#\{(x,n)\in \mathcal{V}_t: \kappa_{\pmb{A}}(x,n,t)=k\} \mathrm{e}^{-s|\psi(n,k)|}.
\end{align*}
If $(x,n)\in \mathcal{V}_t$ satisfies $\kappa_{\pmb{A}}(x,n,t)=k$, then $x$ is determined by the first $n+k$ letters of $x$ and $t$. Therefore, we estimate 
$$\#\{(x,n)\in \mathcal{V}_t: \kappa_{\pmb{A}}(x,n,t)=k\} \leq N^{n+k}.$$
We can estimate 
$$\mathrm{Tr}(\mathrm{e}^{-s|\D_t|})\leq \sum_{n\in \Z}\sum_{k=\max(0,-n)}^\infty N^{n+k}\mathrm{e}^{-s(|n|+k+1)},$$
which is finite if $s>\log(N)$. 

Next, we prove item 2). We identify $C_c(\mathcal{V}_t\times \mathcal{V}_t)$ with a space of finite rank operators on $\ell^2(\mathcal{V}_t)$ by $K\delta_{(x,n)}:=\sum_{(x',n')} K((x',n'),(x,n))\delta_{(x',n')}$, for $K\in C_c(\mathcal{V}_t\times \mathcal{V}_t)$. It is clear that for any two Borel functions $f_1,f_2:\R\to \R$ and $K\in C_c(\mathcal{V}_t\times \mathcal{V}_t)$, the operator $f_1(\D_t)Kf_2(\D_t)$ again belongs to $C_c(\mathcal{V}_t\times \mathcal{V}_t)$, e.g. is of finite rank. By \cite[Proof of Theorem 5.2.3]{GM}, it holds that $[F_t,S_j]\in C_c(\mathcal{V}_t\times \mathcal{V}_t)$ for any $j$ and therefore $[F_t,a]\in C_c(\mathcal{V}_t\times \mathcal{V}_t)$ for any $a\in \mathcal{A}_0$. 
Hence item 2) is true. 
Item 3) is proved in \cite[Theorem 5.2.3]{GM}.
\end{proof}

The remainder of this subsection will use the ingredients of the spectral triple
in Proposition \ref{propoflione} to construct twisted spectral triples
satisfying the conditions of Lemma \ref{asstoprovtauzero}.

\begin{lemma}
\label{explem}
Let $\pmb{A}=(A_{ij})_{i,j=1}^N$ be an $N\times N$-matrix of $0$'s and $1$'s with only non-zero rows and columns. Take $t\in \Omega_{\pmb{A}}$ and consider the odd spectral triple  $(\A_0,\ell^2(\mathcal{V}_t),\D_t)$ from Proposition \ref{propoflione}. For any $s\in \C$ and $a\in \A_0$, it holds that $a\Dom(\mathrm{e}^{s|\D_t|})\subseteq \Dom(\mathrm{e}^{s|\D_t|})$ and the operator $\mathrm{e}^{s|\D_t|}a\mathrm{e}^{-s|\D_t|}$ extends to a bounded operator on $\ell^2(\mathcal{V}_t)$.
\end{lemma}

\begin{proof}
It suffices to consider $a=S_j$ for some $j$. A short computation shows that 
$$S_j\delta_{(x,n)}=\begin{cases}
\delta_{(jx,n+1)},\; &\mbox{if $jx$ is admissible},\\
0,  \; &\mbox{if $jx$ is not admissible}. \end{cases}$$
Therefore, 
\begin{align*}
\mathrm{e}^{s|\D_t|}S_j \delta_{(x,n)}
&=\begin{cases}
\mathrm{e}^{s|\psi(n+1,\kappa_{\pmb{A}}(jx,n+1,t))|}\delta_{(jx,n+1)},\; &\mbox{if $jx$ is admissible},\\
0,  \; &\mbox{if $jx$ is not admissible}. 
\end{cases}
\end{align*}
Since 
$$
\Big||\psi(n+1,\kappa_{\pmb{A}}(jx,n+1,t))|-|\psi(n,\kappa_{\pmb{A}}(x,n,t))|\Big|\leq 2,
$$
we conclude that 
$S_j\Dom(\mathrm{e}^{s|\D_t|})\subseteq \Dom(\mathrm{e}^{s|\D_t|})$ 
and $\mathrm{e}^{s|\D_t|}S_j\mathrm{e}^{-s|\D_t|}$ defines a bounded operator.
\end{proof}

The ingredients to construct a twisted 
spectral triple are now all in place. For $t\in \Omega_{\pmb{A}}$
and $F_t$ as defined in Proposition \ref{propoflione} we define the self-adjoint operator 
\begin{equation}
\D_{{\rm af},t}:=F_t\mathrm{e}^{|\D_t|}.
\label{eq:daft}
\end{equation}
The ``af" stands for {\bf a}ctually {\bf f}initely-summable.
Indeed $(i+\D_{{\rm af},t})^{-1}\in \mathcal{L}^p(\ell^2(\mathcal{V}_t))$ for any $p>\log(N)$ by the proof of  Item 1) in Proposition \ref{propoflione} (see page \pageref{propoflione}). Using Lemma \ref{explem}, for each  $s\in \C$ we can define a homomorphism $\alpha_s:\mathcal{A}_0\to \mathbb{B}(\ell^2(\mathcal{V}_t))$ as 
$$\alpha_s(a):=|\D_{{\rm af},t}|^{is}a|\D_{{\rm af},t}|^{-is}=\mathrm{e}^{is|\D_t|}a\mathrm{e}^{-is|\D_t|}.$$
We write $\sigma:=\alpha_{-i}$. Define $\A$ as the saturation of $\A_0$ under $\sigma$, i.e. $\A$ is the algebra generated by $\cup_{k\in \Z}\sigma^k(\A_0)$. Since $\sigma(a)^*=\sigma^{-1}(a^*)$, $\A$ is a $*$-algebra of bounded operators on $\ell^2(\mathcal{V}_t)$. 

\begin{prop}
\label{itemonefromprpf}
Let $\pmb{A}=(A_{ij})_{i,j=1}^N$ be an $N\times N$-matrix of $0$'s and $1$'s, with no row or column being $0$, and take $t\in \Omega_{\pmb{A}}$. The collection $(\A,\ell^2(\mathcal{V}_t),\D_{{\rm af},t},\sigma)$ is a strongly regular finitely summable twisted spectral triple. Moreover, for any two Borel functions $f_1,f_2:\R\to \R$ and $a\in \A$ the twisted commutator $[\D_{{\rm af},t},a]_{\sigma}$ preserves $\Dom(f_1(\D_{{\rm af},t}))$ and the operator $f_1(\D_{{\rm af},t})[\D_{{\rm af},t},a]_\sigma f_2(\D_{{\rm af},t})$ extends to a trace class operator on $\H$.
\end{prop}

\begin{proof}
That $(\A,\ell^2(\mathcal{V}_t),\D_{{\rm af},t},\sigma)$ is a twisted spectral triple follows from noting that for $a\in \mathcal{A}$ 
\begin{equation}
\label{twicomide}
[\D_{{\rm af},t},a]_{\sigma}=|\D_{{\rm af},t}|[F_t,a],
\end{equation}
which is bounded because for any $a\in \mathcal{A}_0$ and $k\in \Z$,
$$[F_t,\sigma^{k}(a)]=|\D_{{\rm af},t}|^{k}[F_t,a]|\D_{{\rm af},t}|^{-k},$$
which in turn belongs to $C_c(\mathcal{V}_{t}\times \mathcal{V}_t)$ because $[F_t,a]\in C_c(\mathcal{V}_{t}\times \mathcal{V}_t)$ by the proof of Proposition \ref{propoflione} (see page \pageref{propoflione}). By Proposition \ref{strongreforsimsi} (see page \pageref{strongreforsimsi}), $(\A,\ell^2(\mathcal{V}_t),\D_{{\rm af},t},\sigma)$ is strongly regular. The twisted spectral triple $(\A,\ell^2(\mathcal{V}_t),\D_{{\rm af},t},\sigma)$ is finitely summable since $(\A,\ell^2(\mathcal{V}_t),\D_{t})$ is $\mathrm{Li}_1$-summable. The last property stated in proposition follows from the identity \eqref{twicomide} and Item 2) in Proposition \ref{propoflione}.
\end{proof}

Thus for any Cuntz-Krieger algebra we can construct a 
strongly regular twisted spectral triple satisfying condition a) of Lemma
\ref{asstoprovtauzero}.
To prove condition b) of Lemma \ref{asstoprovtauzero} for some of these twisted spectral triples, we 
specialise to a particular family of Cuntz-Krieger algebras.

\subsubsection{The counterexample}
\label{sub:def-props}

The case of interest to us is the Cuntz-Krieger algebra coinciding with the action of the free group on $d$ generators acting on its Gromov boundary. We consider the $2d\times 2d$-matrix 
\begin{equation}
\label{aforfree}
\pmb{A}:=\begin{pmatrix} 
1&0&1&1&\cdots &1&1\\
0&1&1&1&\cdots &\vdots&\vdots\\
1&1&1&0&\cdots &\vdots&\vdots\\
1&1&0&1&\cdots &\vdots&\vdots\\
\vdots&\vdots&&&\ddots&\\
\vspace{1.5mm}
1&1&\cdots &&&1&0\\
1&1&1&\cdots&&0&1
\end{pmatrix}.
\end{equation}
In other words, decomposing $\pmb{A}$ into $2\times 2$-blocks we have the unit $2\times 2$-matrix on all diagonal entries and the $2\times 2$-matrix with all entries $1$ in all other positions. In this case, 
$$K_*(\mathcal{O}_{\pmb{A}})\cong K^{*+1}(\mathcal{O}_{\pmb{A}})\cong 
\begin{cases}
\Z^d\oplus \Z/(d-1)\Z, \;&*=0,\\
\Z^d,\;&*=1.
\end{cases}$$
For more details, see \cite[Proposition 3.4.6 and 3.4.7]{GM} and references therein.

We identify the alphabet $\{1,2,\ldots, 2d\}$ with the alphabet $\{a_1,b_1,a_2,b_2,\ldots, a_d,b_d\}$. We think of $\{a_1,b_1,a_2,b_2,\ldots, a_d,b_d\}$ as a symmetric generating set for $\mathbb{F}_d$ with $a_j=b_j^{-1}$. A word $\mu$ on the alphabet $\{a_1,b_1,a_2,b_2,\ldots, a_d,b_d\}$ is admissible for $\pmb{A}$ if and only if $\mu$ thought of as a product of its letters in $\mathbb{F}_d$ is a reduced word in the generating set $\{a_1,b_1,a_2,b_2,\ldots, a_d,b_d\}$. 

Therefore, identifying a finite admissible word $\mu$ with its product in $\mathbb{F}_d$ induces a bijection of sets $\mathcal{V}_{\pmb{A}}\to \mathbb{F}_d$. 
The space $\Omega_{\pmb{A}}$ can be identified with the space $\partial \mathbb{F}_d$ of all infinite paths in $\mathbb{F}_d$ and coincides with its Gromov boundary. 

Let $\lambda_\mu\in C(\partial \mathbb{F}_d)\rtimes \mathbb{F}_d$ denote the unitary corresponding to the group element in $\mathbb{F}_d$ defined by $\mu$ and $\chi_{C_\mu}\in C(\partial \mathbb{F}_d)\rtimes \mathbb{F}_d$ the characteristic function of the cylinder set $C_\mu$.
By \cite[Section 2]{spielbergaren}, the mapping defined from $S_{a_j}\mapsto \lambda_{a_j}(1-\chi_{C_{b_j}})$ and $S_{b_j}\mapsto \lambda_{b_j}(1-\chi_{C_{a_j}})$ gives an isomorphism $\mathcal{O}_{\pmb{A}}\to C(\partial \mathbb{F}_d)\rtimes \mathbb{F}_d$.

We can describe the space $\ell^2(\mathcal{V}_t)$ and the spectral triple $(\mathcal{A}_0,\ell^2(\mathcal{V}_t),\D_t)$ in terms of the free group in this case. We write $C^\infty(\partial \mathbb{F}_d)$ for the space of locally constant functions; it is generated by the cylinder functions $\{\chi_{C_\mu}: \mu\in \mathbb{F}_d\}$. First note that $\mathcal{A}_0$ concides with the algebraic crossed product $C^\infty(\partial \mathbb{F}_d)\rtimes^{\rm alg} \mathbb{F}_d$. For an element $\mu\in \mathbb{F}_d$ and $t\in \partial\mathbb{F}_d$ we define $\ell(\mu,t)$ as the number of cancellations occuring to put the product $\mu t$ in reduced form. A short computation shows that the map 
\begin{equation}
\varphi_t:\mathbb{F}_d\to \mathcal{V}_t, \quad \mu\mapsto (\mu t, |\mu|-2\ell(\mu,t),t),
\label{eq:tee-map}
\end{equation}
is a bijection of sets. Under $\varphi_t$ and the identification $\A_0=C^\infty(\partial \mathbb{F}_d)\rtimes^{\rm alg} \mathbb{F}_d$, an element $a\lambda_\gamma\in \A_0$, where $a\in C^\infty(\partial\mathbb{F}_d)$ and $\gamma\in \mathbb{F}_d$, acts on $\ell^2(\mathbb{F}_d)$ as 
$$a\lambda_\gamma\delta_{\mu}:=a(\gamma \mu t)\delta_{\gamma\mu}.$$
Moreover, 
$$\kappa_{\pmb{A}}(\mu t, |\mu|-2\ell(\mu,t),t)=\ell(\mu,t).$$
Therefore, $(\mathcal{A}_0,\ell^2(\mathcal{V}_t),\D_t)$ is unitarily equivalent to $(C^\infty(\partial \mathbb{F}_d)\rtimes^{\rm alg} \mathbb{F}_d,\ell^2(\mathbb{F}_d),\D_{\mathbb{F}_d,t})$ where $\D_{\mathbb{F}_d,t}:=\varphi_t^{-1}\D_t\varphi_t$ and we compute that for $\mu\in \mathbb{F}_d$ we have
$$
\D_{\mathbb{F}_d,t}\delta_\mu=\psi\big(|\mu|-2\ell(\mu,t),\ell(\mu,t)\big)\,\delta_\mu.
$$
In particular, $|\D_{\mathbb{F}_d,t}|\delta_\mu
=\big(\big||\mu|-2\ell(\mu,t)\big|+\ell(\mu,t)\big)\delta_\mu$.

Define the function 
$$
\Psi_t(\mu):=\mathrm{e}^{\big||\mu|-2\ell(\mu,t)\big|+\ell(\mu,t)}.
$$
In this case, we can identify $\D_{{\rm af},t}$ with the following operator on $\ell^2(\mathbb{F}_d)$:
\begin{equation}
\label{daftfreea}
\D_{{\rm af},t}\delta_\mu=\big(2\chi_{\{0\}}(\ell(\mu,t))-1\big)\,\Psi_t(\mu)\delta_\mu.
\end{equation}
The reader should note that the action of $C(\partial \mathbb{F}_d)\rtimes \mathbb{F}_d$ factors over the action of $C_b(\mathbb{F}_d)\rtimes \mathbb{F}_d$ on $\ell^2(\mathbb{F}_d)$ and the inclusion $i_t:C(\partial \mathbb{F}_d)\rtimes \mathbb{F}_d\hookrightarrow C_b(\mathbb{F}_d)\rtimes \mathbb{F}_d$ defined from the equivariant $*$-monomorphism $i_t:C(\partial \mathbb{F}_d)\to C_b(\mathbb{F}_d)$ given by $i_t(a)(\mu):=a(\mu t)$. It is clear from the definition that for $b\in C_b(\mathbb{F}_d)$ and $\gamma\in \mathbb{F}_d$, 
$$\sigma(b\lambda_\gamma)\delta_\mu=\Psi_t(\gamma\mu)\Psi_t(\mu)^{-1}b(\gamma \mu )\lambda_\gamma \delta_\mu.$$
We conclude that $\sigma$ is a partially defined homomorphism $C_b(\mathbb{F}_d)\rtimes \mathbb{F}_d\to C_b(\mathbb{F}_d)\rtimes \mathbb{F}_d$. Moreover, the $C^*$-closure $A$ of $\A$ is an intermediate $C^*$-algebra
$$C(\partial \mathbb{F}_d)\rtimes \mathbb{F}_d=\mathcal{O}_{\pmb{A}}\subseteq A\subseteq C_b(\mathbb{F}_d)\rtimes \mathbb{F}_d.$$

The proof that the twisted spectral triple 
$(C^\infty(\partial\mathbb{F}_d)\rtimes^{{\rm alg}}\mathbb{F}_d,
\ell^2(\mathbb{F}_d),D_{{\rm af},t})$
constructed as in
Proposition \ref{itemonefromprpf} satisfies condition b) of 
Lemma \ref{asstoprovtauzero}
is presented below in the appendix.

The verification of the detailed holomorphy statements found in the appendix uses brute force calculation.
We see no direct way of proving condition b) of 
Lemma \ref{asstoprovtauzero} for a general Cuntz-Krieger algebra. It does however seem quite 
likely, to us, that condition b) of 
Lemma \ref{asstoprovtauzero} holds for a more general class of 
Cuntz-Krieger algebras.

\begin{thm}
\label{maintwo}
Take a fixed point $t\in \partial \mathbb{F}_d$. Let $(\A,\ell^2(\mathbb{F}_d),\D_{{\rm af},t},\sigma)$ denote the regular finitely summable twisted
spectral triple with finite discrete dimension spectrum obtained from saturating the weakly twisted spectral triple $(C^\infty(\partial \mathbb{F}_d)\rtimes^{\rm alg}\mathbb{F}_d,\ell^2(\mathbb{F}_d),\D_{{\rm af},t},\sigma)$ with $\D_{{\rm af},t}$ as in Equation \eqref{daftfreea} and $\sigma(a):= |\D_{{\rm af},t}|a|\D_{{\rm af},t}|^{-1}$. The twisted spectral triple $(\A,\ell^2(\mathbb{F}_d),\D_{{\rm af},t},\sigma)$ pairs non-trivially with $K_1(\A)$ but
the cochain $(\phi_m)$ provided by Moscovici's ansatz is the zero cochain.
Hence the equality \eqref{dodycpair} does not hold.
\end{thm}

\begin{proof}
Set $A:=\overline{\A}\subseteq \mathbb{B}(\ell^2(\mathbb{F}_d))$. We need to verify that $(\A,\ell^2(\mathbb{F}_d),\D_{{\rm af},t},\sigma)$ satisfies the assumptions of Lemma \ref{asstoprovtauzero} (see page \pageref{asstoprovtauzero}) and that for some $x\in K_1(A)$,
$$
\langle [(\A,\ell^2(\mathbb{F}_d),\D_{{\rm af},t},\sigma)],x\rangle\neq 0.
$$
The twisted spectral triple $(\A,\ell^2(\mathbb{F}_d),\D_{{\rm af},t},\sigma)$ satisfies assumption a) of Lemma \ref{asstoprovtauzero} by Proposition \ref{itemonefromprpf} (see page \pageref{itemonefromprpf}).
As mentioned we leave the proof of assumption b) of Lemma \ref{asstoprovtauzero} to the appendix, but note that
it uses that $t$ is a fixed point.

As for $\langle [(\A,\ell^2(\mathbb{F}_d),\D_{{\rm af},t},\sigma)],x\rangle\neq 0$,
we take $x:=[\lambda_\gamma]$ where 
$\gamma\in\{a_1,b_1,\ldots,a_d,b_d\}$ is one of the 
generators of $\mathbb{F}_d$. By the definition of the index pairing, 
$\langle [(\A,\ell^2(\mathbb{F}_d),\D_{{\rm af},t},\sigma)],x\rangle$ 
coincides with the index pairing 
$\langle [(\A,\ell^2(\mathbb{F}_d),F_t)],x\rangle$ 
between the $K$-homology class 
$[(\A,\ell^2(\mathbb{F}_d),F_t)]\in K^1(A)$ and $x\in K_1(A)$. 
Again, going to the definition, the index pairing 
$\langle [(\A,\ell^2(\mathbb{F}_d),F_t)],x\rangle$ is by 
definition the index of the operator 
$P\lambda_\gamma P: P\ell^2(\mathbb{F}_d)\to P\ell^2(\mathbb{F}_d)$. 
Computing, we see that for $x=[\lambda_\gamma]$
$$
\langle [(\A,\ell^2(\mathbb{F}_d),\D_{{\rm af},t},\sigma)],x\rangle
=\mathrm{ind}\big(P\lambda_\gamma P:\ell^2(\mathcal{V}_{\pmb{A},t_1})\to \ell^2(\mathcal{V}_{\pmb{A},t_1})\big).
$$
Let us compute this index. We write $\pmb{A}=(A_{i,j})_{i,j=1}^{2d}$ for the matrix in Equation \eqref{aforfree} (i.e. $O_{\pmb{A}}\cong C(\partial\mathbb{F}_d)\rtimes\mathbb{F}_d$). Note that $\mathcal{V}_{\pmb{A}}=\mathbb{F}_d$. For $\mu \in \mathcal{V}_{\pmb{A},t_1}$ 
we have that 
$$
P\lambda_\gamma P\delta_\mu=\begin{cases}
\delta_{\gamma\mu}, \;& |\mu|>0,\\
A_{\gamma,t_1}\delta_\gamma, \;& |\mu|=0. 
\end{cases}
$$
In particular, $\dim\ker P\lambda_\gamma P=1-A_{\gamma,t_1}$. Therefore, 
$$
\mathrm{ind}(P\lambda_\gamma P)
=\dim\ker P\lambda_\gamma P-\dim\ker P\lambda_{\gamma^{-1}} P
=A_{\gamma^{-1},t_1}-A_{\gamma,t_1}
=\delta_{\gamma^{-1},t_1}-\delta_{\gamma,t_1}.
$$
We see that for $\gamma=t_1^{\pm 1}$, 
$\langle [(\A,\ell^2(\mathbb{F}_d),\D_{{\rm af},t},\sigma)],x\rangle\neq 0$. 
\end{proof}

\begin{appendix}

\section{Meromorphic extension of heat traces for the free group}

\subsection{Setting up the proof of holomorphy}
\label{sub:holo}

We are now close to showing that the twisted spectral triple from 
Subsection \ref{sub:def-props} 
satisfies all the assumptions of Theorem \ref{counterexthm} (see page \pageref{counterexthm}) and so
provides a counterexample to the equality \eqref{dodycpair}. So far, Proposition \ref{propoflione} allows us to control the index pairing and Proposition \ref{itemonefromprpf} allows for proving all the assumptions of Lemma \ref{asstoprovtauzero} except for condition b). 

We shall prove this last remaining piece in the special case of 
the free group;
this result is based on the 
following proposition. We use the notation 
$P:=(F_t+1)/2=\chi_{[0,\infty)}(\D_{{\rm af},t})=\chi_{[0,\infty)}(\D_t)$. 
We introduce the notation
$$
\Z^m_0:=\Big\{\mathbbm{k}=(k_1,\ldots,k_m)\in \Z^m: \sum_{j=1}^m k_j=0\Big\}.
$$

\begin{prop}
\label{singprop}
The twisted spectral triple 
$(\A,\ell^2(\mathcal{V}_t),\D_{{\rm af},t},\sigma)$ constructed 
above satisfies all the assumptions of Lemma \ref{asstoprovtauzero} 
if there is a finite set $\mathcal{P}_0\subseteq \R$ satisfying the two conditions:
\begin{enumerate}
\item For any $m\in \N$ and $\mu_1,\nu_1,\mu_2,\nu_2\ldots,\mu_m,\nu_m\in \mathcal{V}_{\pmb{A}}$, the function 
\begin{equation}
(s_1,s_2,\ldots, s_m)\mapsto \mathrm{Tr}(S_{\mu_1}S_{\nu_1}^*\mathrm{e}^{-s_1|\D_t|}\cdots S_{\mu_m}S_{\nu_m}^*\mathrm{e}^{-s_m|\D_t|}),
\label{eq:the-function}
\end{equation}
extends meromorphically from 
$\mathrm{Re}(s_1+\cdots+s_m)>>0$ to $\C^m$,
and is holomorphic outside a subvariety $V_m(\mu_1,\nu_1,,\ldots,\mu_m,\nu_m)\subseteq \C^m$ where it has an order $2$ singularity and 
$$
\{s\in \C: \exists \mathbbm{k}\in \Z^{m}_0: \; \mathbbm{k}+(0_{m-1},s)\in V_m(\mu_1,\nu_1,\ldots,\mu_m,\nu_m)\}\subseteq \mathcal{P}_0+\pi i\Z;
$$
\item For any $m\in \N$ and $\mu_1,\nu_1,\mu_2,\nu_2\ldots,\mu_m,\nu_m\in \mathcal{V}_{\pmb{A}}$, the function 
$$
(s_1,s_2,\ldots, s_m)\mapsto \mathrm{Tr}(PS_{\mu_1}S_{\nu_1}^*P\mathrm{e}^{-s_1|\D_t|}\cdots PS_{\mu_m}S_{\nu_m}^*P\mathrm{e}^{-s_m|\D_t|}),
$$
extends meromorphically from $\mathrm{Re}(s_1+\cdots+s_m)>>0$ to $\C^m$ being holomorphic outside a subvariety $V_m^P(\mu_1,\nu_1,\ldots,\mu_m,\nu_m)\subseteq \C^m$ where it has an order $1$ singularity and 
$$
\{s\in \C: \exists \mathbbm{k}\in \Z^{m}_0: \; \mathbbm{k}+(0_{m-1},s)\in V_m^P(\mu_1,\nu_1,,\ldots,\mu_m,\nu_m)\}\subseteq \mathcal{P}_0+\pi i\Z.
$$
\end{enumerate}
\end{prop}

\begin{proof}
As observed in the paragraph preceding the proposition, it remains to prove condition b) of Lemma \ref{asstoprovtauzero}. In other words, we must show that there is a discrete set $\mathcal{P}\subseteq \C$ such that for all $m\in \N$, $\mu_1,\nu_1,\mu_2,\nu_2\ldots,\mu_m,\nu_m\in \mathcal{V}_{\pmb{A}}$ and $\mathbbm{l}=(l_1,l_2,\ldots,l_m)\in \Z^m$ the functions
\begin{align}
\label{funofsos}
s\mapsto \mathrm{Tr}(\sigma^{l_1}(S_{\mu_1}S_{\nu_1}^*)\sigma^{l_2}(S_{\mu_2}S_{\nu_2}^*)\cdots \sigma^{l_m}(S_{\mu_m}S_{\nu_m}^*)\mathrm{e}^{-s|\D_t|}),\\
\label{funofsosII}
s\mapsto \mathrm{Tr}(F\sigma^{l_1}(S_{\mu_1}S_{\nu_1}^*)\sigma^{l_2}(S_{\mu_2}S_{\nu_2}^*)\cdots \sigma^{l_m}(S_{\mu_m}S_{\nu_m}^*)\mathrm{e}^{-s|\D_t|})
\end{align}
admits a meromorphic extension to $\C$, holomorphic outside 
$\mathcal{P}$ with at most simple poles in $\mathcal{P}$. 
We can write 
\begin{align*}
&\mathrm{Tr}(\sigma^{l_1}(S_{\mu_1}S_{\nu_1}^*)\sigma^{l_2}(S_{\mu_1}S_{\nu_1}^*)\cdots \sigma^{l_m}(S_{\mu_m}S_{\nu_m}^*)\mathrm{e}^{-s|\D_t|})\\
&=\mathrm{Tr}(S_{\mu_1}S_{\nu_1}^*\mathrm{e}^{-k_1|\D_t|}S_{\mu_2}S_{\nu_2}^*\mathrm{e}^{-k_2|\D_t|}\cdots \mathrm{e}^{-k_{m-1}|\D_t|}S_{\mu_m}S_{\nu_m}^*\mathrm{e}^{-(s+k_m)|\D_t|}),
\end{align*}
where $k_j=l_j-l_{j+1}$ and $k_m:=l_m-l_1$. It follows
from 1.  that the function \eqref{funofsos} is meromorphic in $\C$, holomorphic outside and with order $2$ poles in the set 
$$
\mathcal{P}'=\cup_m \big\{s\in \C: \exists \mathbbm{k}\in \Z^{m}: \; \mathbbm{k}+(0_{m-1},s)\in V_m(\mu_1,\nu_1,,\ldots,\mu_m,\nu_m)\big\},
$$
which is discrete as it is contained in $\mathcal{P}_0+\pi i\Z$. 

As for the function in \eqref{funofsosII}, we note that the argument in the previous paragraph reduces the problem to finding a discrete 
subset $\mathcal{P}''\subseteq \C$ such that the function 
$$
s\mapsto \mathrm{Tr}(P\sigma^{l_1}(S_{\mu_1}S_{\nu_1}^*)\sigma^{l_2}(S_{\mu_2}S_{\nu_2}^*)\cdots \sigma^{l_m}(S_{\mu_m}S_{\nu_m}^*)\mathrm{e}^{-s|\D_t|}),
$$
extends meromorphically to $\C$, holomorphic outside 
$\mathcal{P}''$ and with at most simple poles in $\mathcal{P}''$. 
An induction argument using Proposition \ref{itemonefromprpf} 
shows that the function 
\begin{align*}
s\mapsto \mathrm{Tr}&(P\sigma^{l_1}(S_{\mu_1}S_{\nu_1}^*)\sigma^{l_2}(S_{\mu_2}S_{\nu_2}^*)\cdots \sigma^{l_m}(S_{\mu_m}S_{\nu_m}^*)\mathrm{e}^{-s|\D_t|})\\
&-\mathrm{Tr}(PS_{\mu_1}S_{\nu_1}^*P\mathrm{e}^{-k_1|\D_t|}PS_{\mu_2}S_{\nu_2}^*P\mathrm{e}^{-k_2|\D_t|}\cdots \mathrm{e}^{-k_{m-1}|\D_t|}PS_{\mu_m}S_{\nu_m}^*P\mathrm{e}^{-(s+k_m)|\D_t|}),
\end{align*}
is an entire function, here $k_j=l_j-l_{j+1}$ and $k_m:=l_m-l_1$ as above. Therefore, we can take 
$$
\mathcal{P}''
=\bigcup_m \big\{s\in \C: \exists \mathbbm{k}\in \Z^{m}: \; \mathbbm{k}+(0,s)\in V_m^P(\mu_1,\nu_1,,\ldots,\mu_m,\nu_m)\big\},
$$
which is discrete as it is contained in $\mathcal{P}_0+\pi i\Z$.  
We conclude that the functions in \eqref{funofsos} and \eqref{funofsosII} are meromorphic, holomorphic outside 
$\mathcal{P}:=\mathcal{P}'\cup \mathcal{P}''$ with at most order 
$2$ poles in $\mathcal{P}$.
\end{proof}

We note that since the spectrum of $\D_t$ coincides with the integers, the set $\mathcal{P}$ in Proposition \ref{singprop} must be invariant under translation by $2\pi i \Z$.

We see no reason for the assumptions 
of Proposition \ref{singprop} to fail for a general Cuntz-Krieger 
algebra. However, due to the lack of a better approach than brute force calculation, 
we see no way of proving these assumptions in general. 

For the free group and a fixed point $t$, we prove Assumption 1) of Proposition \ref{singprop} below in Subsection \ref{sub:heat} and Assumption 2) of Proposition \ref{singprop} below in Subsection \ref{toeplsuzise}. This proves the omitted step in the proof of Theorem \ref{maintwo}.

\subsection{Heat trace computations on $\A$ for the free group}
\label{sub:heat}

In this subsubsection we will prove that Assumption 1) of Proposition \ref{singprop} is fulfilled for the action of the free group on its Gromov boundary, as described in Subsection \ref{sub:def-props}.

\begin{prop}
\label{kappacompprop}
Let $\mu\in \mathcal{V}_{\pmb{A}}$, $t\in \Omega_{\pmb{A}}$ and $(x,n)\in \mathcal{V}_t$ be an element with $A_{\mu_{|\mu|},x_1}=1$. Define $\mu\wedge t,\,\sigma\wedge t\in \mathcal{V}_t$ as the longest finite words 
for which $\mu=\mu_0(\mu\wedge t)$ for some $\mu_0\in \mathcal{V}_{\pmb{A}}$ and $t=(\sigma\wedge t)t_0$ for some $t_0\in \Omega_{\pmb{A}}$. Then 
$$\kappa(\mu x,n+|\mu|,t)=\begin{cases}
\kappa(x,n,t), \; &\mbox{if}\;\;\kappa(x,n,t)>-n,\\
\kappa(x,n,t)-|\mu\wedge t|, \; &\mbox{if}\;\;\kappa(x,n,t)=-n.\end{cases}$$
\end{prop}
\begin{proof}
The proof is elementary combinatorics, 
and is presaged in \cite[Proposition 4.3]{GM2}.
Let $k\geq 0$ be an integer with $k+n+|\mu|\geq 0$ that 
satisfies $\sigma_{\pmb{A}}^{n+|\mu|+k}(\mu x)= \sigma_{\pmb{A}}^k(t)$. 
If $k+n\geq 0$, then $\sigma_{\pmb{A}}^{n+|\mu|+k}(\mu x)= \sigma_{\pmb{A}}^k(t)$ is equivalent to $\sigma_{\pmb{A}}^{n+k}( x)=\sigma_{\pmb{A}}^k(t)$. In particular, if $\kappa_{\pmb{A}}(x,n,t)>-n$ then $k:=\kappa(x,n,t)$ is the minimal $k$ for which $\sigma_{\pmb{A}}^{n+|\mu|+k}(\mu x)= \sigma_{\pmb{A}}^k(t)$. If $\kappa_{\pmb{A}} (x,n,t)=-n$ then $n$ must be negative and $x=\sigma_{\pmb{A}}^{-n}(t)$ so $\kappa_{\pmb{A}} (\mu x, n+|\mu|,t)$ is the minimal $k$ for which $\sigma_{\pmb{A}}^{n+|\mu|+k}(\mu \sigma_{\pmb{A}}^{-n}(t))= \sigma_{\pmb{A}}^k(t)$ which by the definition of $|\mu\wedge t|$ is given $k=-n-|\mu\wedge t|$. 
\end{proof}

To proceed with our argument, we fix finite words 
$\mu_1,\nu_1,\mu_2,\nu_2\ldots,\mu_m,\nu_m\in \mathcal{V}_{\pmb{A}}$. 
We will first focus on the expression $\mathrm{Tr}(S_{\mu_1}S_{\nu_1}^*\mathrm{e}^{-s_1|\D_t|}\cdots S_{\mu_m}S_{\nu_m}^*\mathrm{e}^{-s_m|\D_t|})$, for $s_1,s_2,\ldots,s_m>>0$. For a finite word $\nu\in \mathcal{V}_{\pmb{A}}$ and $x=\nu x_0\in C_\nu$, we write 
$$\nu^*x:=x_0.$$
Similarly, if $\rho=\nu\rho_0\in \mathcal{V}_{\pmb{A}}$, we write $\nu^*\rho:=\rho_0$. For a finite word $\nu=\nu_1\nu_2\cdots\nu_k$, we write $\bar{\nu}:=\nu_k\nu_{k-1}\cdots \nu_2\nu_1$. We compute that 
\begin{align*}
\mathrm{Tr}(S_{\mu_1}&S_{\nu_1}^*\mathrm{e}^{-s_1|\D_t|}\cdots S_{\mu_m}S_{\nu_m}^*\mathrm{e}^{-s_m|\D_t|})
= \sum_{(x,n)\in \mathcal{V}_t}\langle\delta_{(x,n)}|S_{\mu_1}S_{\nu_1}^*\cdots S_{\mu_m}S_{\nu_m}^*\delta_{(x,n)}\rangle\times \\
&\times\prod_{j=1}^m\exp\Big({-s_j\big(\big|n+\sum_{i=j+1}^m |\mu_i|-|\nu_i|\big|+\kappa(\mu_{j+1}\bar{\nu}_{j+1}^*\cdots \mu_m\bar{\nu}_m^*x,n+\sum_{i=j+1}^m |\mu_i|-|\nu_i|,t)\big)}\Big).
\end{align*}
An empty sum is interpreted as $0$ and $\mu_{j+1}\bar{\nu}_{j+1}^*\cdots \mu_m\bar{\nu}_m^*x$ is interpreted as $x$ for $j=m$. Note that whenever $\langle\delta_{(x,n)}|S_{\mu_1}S_{\nu_1}^*\cdots S_{\mu_m}S_{\nu_m}^*\delta_{(x,n)}\rangle$ is non-zero, the infinite word $\mu_{j+1}\bar{\nu}_{j+1}^*\cdots \mu_m\bar{\nu}_m^*x$ is well-defined for $j=1,\ldots, m$. In fact, we can reduce the possible words that can appear further
using the next result.

\begin{prop}
\label{dickotomy}
Let $\mu_1,\nu_1,\mu_2,\nu_2\ldots,\mu_m,\nu_m\in \mathcal{V}_{\pmb{A}}$ be finite words and $t\in \Omega_{\pmb{A}}$. Exactly one of the following two statements holds:
\begin{enumerate}
\item There are finite words $\rho_1,\ldots, \rho_q$ and coefficients $c_1,\ldots, c_q\in \Z$ such that 
$$
S_{\mu_1}S_{\nu_1}^*\cdots S_{\mu_m}S_{\nu_m}^*=\sum_{l=1}^q c_lS_{\rho_l}S_{\rho_l}^*=\sum_{l=1}^q c_l\chi_{C_{\rho_l}}.
$$
\item $\langle\delta_{(x,n)}|S_{\mu_1}S_{\nu_1}^*\cdots S_{\mu_m}S_{\nu_m}^*\delta_{(x,n)}\rangle=0$ for all $(x,n)\in \mathcal{V}_t$.
\end{enumerate}
\end{prop}

The proposition follows from the Cuntz-Krieger relations 
on the generators, and can be deduced from \cite[Lemma 1.1]{kpr}
or the more general setting described in \cite{Rae}. 
We henceforth assume that 
$\langle\delta_{(x,n)}|S_{\mu_1}S_{\nu_1}^*\cdots S_{\mu_m}S_{\nu_m}^*\delta_{(x,n)}\rangle\neq 0$ for some $(x,n)$ 
and pick finite words $\rho_1,\ldots, \rho_q$ and coefficients 
$c_1,\ldots, c_q\in \Z$ as in Proposition \ref{dickotomy}. 
Observe that we can always 
choose the $\rho_l$ longer if needed. We introduce the notation 
$$
\sigma_{j,l}:=\mu_{j+1}\bar{\nu}_{j+1}^*\cdots \mu_m\bar{\nu}_m^*\rho_l,
$$
upon assuming that $\rho_l$ is chosen long enough for this definition to make sense. Note that 
$$
\sum_{i=j+1}^m |\mu_i|-|\nu_i|=|\sigma_{j,l}|-|\rho_l|,
$$
for any $l$. We write 
\begin{align}
\nonumber
\mathrm{Tr}(S_{\mu_1}S_{\nu_1}^*&\mathrm{e}^{-s_1|\D_t|}\cdots S_{\mu_m}S_{\nu_m}^*\mathrm{e}^{-s_m|\D_t|})\\
\nonumber
=& \sum_{l=1}^q\sum_{(\rho_lx,n)\in \mathcal{V}_t}c_l\prod_{j=1}^m\mathrm{e}^{-s_j(|n+|\sigma_{j,l}|-|\rho_l||+\kappa(\sigma_{j,l}x,n+|\sigma_{j,l}|-|\rho_l|,t))}\\
\nonumber
=& \sum_{l=1}^q\sum_{n\in \Z}\sum_{k=\max(0,-n)}^\infty \sum_{\kappa(\rho_lx,n)=k}c_l\prod_{j=1}^m\mathrm{e}^{-s_j(|n+|\sigma_{j,l}|-|\rho_l||+\kappa(\sigma_{j,l}x,n+|\sigma_{j,l}|-|\rho_l|,t))}\\
\nonumber
=&\sum_{l=1}^q\sum_{n\in \Z}\sum_{k=\max(1,-n+1)}^\infty \sum_{\kappa(\rho_lx,n,t)=k}c_l\prod_{j=1}^m\mathrm{e}^{-s_j(|n+|\sigma_{j,l}|-|\rho_l||+\kappa(\sigma_{j,l}x,n+|\sigma_{j,l}|-|\rho_l|,t))}\\
\nonumber
&+\sum_{l=1}^q\sum_{n=0}^\infty\sum_{\kappa(\rho_lx,n,t)=0}c_l\prod_{j=1}^m\mathrm{e}^{-s_j(|n+|\sigma_{j,l}|-|\rho_l||+\kappa(\sigma_{j,l}x,n+|\sigma_{j,l}|-|\rho_l|,t))}\\
\nonumber
&+\sum_{l=1}^q\sum_{n=-\infty}^{-1}\sum_{\kappa(\rho_lx,n,t)=-n}c_l\prod_{j=1}^m\mathrm{e}^{-s_j(|n+|\sigma_{j,l}|-|\rho_l||+\kappa(\sigma_{j,l}x,n+|\sigma_{j,l}|-|\rho_l|,t))}\\
\label{firstterm}
&=\sum_{l=1}^q\sum_{n\in \Z}\sum_{k=\max(1,-n+1)}^\infty \sum_{\kappa(x,n,t)=k,A_{\rho_{l,|\rho_l|},x_1}=1}c_l\prod_{j=1}^m\mathrm{e}^{-s_j(|n+|\sigma_{j,l}|-|\rho_l||+k)}\\
\label{secondterm}
&+\sum_{l=1}^q\sum_{n=0}^\infty\sum_{\rho_l x\in \sigma_{\pmb{A}}^{-n}(\{t\})}c_l\prod_{j=1}^m\mathrm{e}^{-s_j(|n+|\sigma_{j,l}|-|\rho_l||)}\\
\label{thirdterm}
&+\sum_{l=1}^q\sum_{n=-\infty}^{-1}\sum_{\kappa(\rho_lx,n,t)=-n}c_l\prod_{j=1}^m\mathrm{e}^{-s_j(|n+|\sigma_{j,l}|-|\rho_l||+\kappa(\sigma_{j,l}x,n+|\sigma_{j,l}|-|\rho_l|,t))}.
\end{align}
In the last step, we used Proposition \ref{kappacompprop}.
The three terms \eqref{firstterm}, \eqref{secondterm}, \eqref{thirdterm} will
now be examined separately.

\subsubsection{The expression in \eqref{firstterm}} To study 
\eqref{firstterm}, we  make use of the following elementary counting argument.

\begin{prop}
Let $\pmb{A}$ be as in Subsection \ref{sub:def-props}. 
For $k\geq \max(1,-n+1)$, 
$\rho\in \mathcal{V}_{\pmb{A}}=\mathbb{F}_d$ and $t\in \Omega_{\pmb{A}}=\partial \mathbb{F}_d$,
\begin{align*}
&\#\{x\in \Omega_{\pmb{A}}:\; \sigma_{\pmb{A}}^{n+k}(x)=\sigma_{\pmb{A}}^k(t), \,x_{n+k}\neq t_k, \, A_{\rho_{|\rho|},x_1}=1\}\\
&\qquad=\begin{cases}
(2d-2)^2(2d-1)^{n+k-2}, \;&n+k>1\\
2d-2,\; &n+k=1 \end{cases}.
\end{align*}
\end{prop}
\begin{proof}
The proposition follows from noting that 
\begin{align*}
\#\{x\in \Omega_{\pmb{A}}:\; &\sigma_{\pmb{A}}^{n+k}(x)=\sigma_{\pmb{A}}^k(t), \,x_{n+k}\neq t_k, \, A_{\rho_{|\rho|},x_1}=1\}\\
&=\#\{\gamma\in\mathbb{F}_d :\; |\gamma|=n+k, \,\gamma_{n+k}\neq t_k, \, A_{\rho_{|\rho|},\gamma_1}=1\},
\end{align*}
and a word counting in the free group. 
\end{proof}
We proceed with computing the expression in \eqref{firstterm} for a fixed $l\in \{1,\ldots, q\}$. 
\begin{align}
\nonumber
\sum_{n\in \Z}&\sum_{k=\max(1,-n+1)}^\infty \sum_{\kappa(x,n,t)=k,A_{\rho_{l,|\rho_l|},x_1}=1}\prod_{j=1}^m\mathrm{e}^{-s_j(|n+|\sigma_{j,l}|-|\rho_l||+k)}\\
\nonumber
=&\sum_{n\in \Z}\sum_{k=\max(1,-n+1)}^\infty \!\!\!\!\!\!\!\!\!\!\#\{x\in \Omega_{\pmb{A}}:\, \sigma_{\pmb{A}}^{n+k}(x)=\sigma_{\pmb{A}}^k(t), \,x_{n+k}\neq t_k, \, A_{\rho_{|\rho|},x_1}\!\!=1\}\prod_{j=1}^m\mathrm{e}^{-s_j(|n+|\sigma_{j,l}|-|\rho_l||+k)}\\
\nonumber
=&\sum_{n=1}^\infty\sum_{k=1}^\infty (2d-2)^2(2d-1)^{n+k-2}\prod_{j=1}^m\mathrm{e}^{-s_j(|n+|\sigma_{j,l}|-|\rho_l||+k)}+(2d-2)\prod_{j=1}^m\mathrm{e}^{-s_j(||\sigma_{j,l}|-|\rho_l||+1)}\\
\nonumber
&+\sum_{k=2}^\infty (2d-2)^2(2d-1)^{k-2}\prod_{j=1}^m\mathrm{e}^{-s_j(||\sigma_{j,l}|-|\rho_l||+k)}+\sum_{n=1}^\infty (2d-2)\prod_{j=1}^m\mathrm{e}^{-s_j(|n-|\sigma_{j,l}|+|\rho_l||+n+1)}\\
\label{longcomp}
&+\sum_{n=1}^\infty\sum_{k=n+2}^\infty (2d-2)^2(2d-1)^{-n+k-2}\prod_{j=1}^m\mathrm{e}^{-s_j(|n-|\sigma_{j,l}|+|\rho_l||+k)}
\end{align}

Expressions of this type can be computed quite easily by means of geometric series. We summarize the relevant computations
in the next lemma, leaving the details to the reader. 

\begin{lemma}
\label{computationofsums}
Let $a,b,c,d_1,\ldots, d_m\in \Z$, $\lambda>0$ be parameters and $s_1,\ldots, s_m\in \C$ complex variables with $\mathrm{Re}(s_1+s_2+\cdots +s_m)$ sufficiently large. Set $d:=\min_j(d_j)$. Then the following computations hold.
\begin{align}
\sum_{n=a}^\infty\sum_{k=b}^\infty \lambda^{cn+k} \mathrm{e}^{-\sum_{j=1}^m s_j(|n+d_j|+k)}=&\frac{\lambda^b\mathrm{e}^{-b\sum_{j=1}^ms_j}}{1-\lambda\mathrm{e}^{-\sum_{j=1}^ms_j} }\sum_{n=a}^{-d}\lambda^{cn} \mathrm{e}^{-\sum_{j=1}^m s_j|n+d_j|}\\
\nonumber
&+\frac{\lambda^{b-cd}\mathrm{e}^{-\sum_{j=1}^ms_j(b+d_j-d)}}{(1-\lambda\mathrm{e}^{-\sum_{j=1}^ms_j})(1-\lambda^c\mathrm{e}^{-\sum_{j=1}^ms_j}) }\\
\sum_{n=a}^\infty\sum_{k=n+b}^\infty \lambda^{cn+k}\mathrm{e}^{-\sum_{j=1}^m s_j(|n+d_j|+k)}=&\frac{\lambda^b}{1-\lambda\mathrm{e}^{-\sum_{j=1}^ms_j} }\sum_{n=a}^{-d}\lambda^{(c+1)n} \mathrm{e}^{-\sum_{j=1}^m s_j(|n+d_j|+n+b)}\\
\nonumber
&+\frac{\lambda^{b-(c+1)d}\mathrm{e}^{-\sum_{j=1}^ms_j(b+d_j-2d)}}{(1-\lambda\mathrm{e}^{-\sum_{j=1}^ms_j})(1-\lambda^{c+1}\mathrm{e}^{-2\sum_{j=1}^ms_j}) }\\
\sum_{n=a}^\infty \lambda^{cn}\mathrm{e}^{-\sum_{j=1}^m s_j|n+d_j|}=&\sum_{n=a}^{-d}\lambda^{cn} \mathrm{e}^{-\sum_{j=1}^m s_j|n+d_j|}+\frac{\lambda^{-cd}\mathrm{e}^{-\sum_{j=1}^ms_j(d_j-d)}}{1-\lambda^c\mathrm{e}^{-\sum_{j=1}^ms_j} }
\\
\sum_{k=b}^\infty \lambda^{k}\mathrm{e}^{-\sum_{j=1}^m s_jk}=&\frac{\lambda^{b}\mathrm{e}^{-b\sum_{j=1}^ms_j}}{1-\lambda\mathrm{e}^{-\sum_{j=1}^ms_j} }
\end{align}
\end{lemma}

We define
$$d_l^\pm:=\min_j \pm (|\sigma_{j,l}|-|\rho_l|).$$
Using Lemma \ref{computationofsums}, we compute the 
five expressions in \eqref{longcomp}, obtaining
\begin{align*}
&\;\;\frac{(2d-2)^2}{1-(2d-1)\mathrm{e}^{-\sum_{j=1}^m s_j}}\sum_{n=1}^{-d_l^+} (2d-1)^{n-1} \mathrm{e}^{-\sum_{j=1}^m s_j(|n+|\sigma_{j,l}|-|\rho_l||+1)}\\
&+\frac{(2d-2)^2}{2d-1}\frac{\mathrm{e}^{-\sum_{j=1}^m s_j(1+|\sigma_{j,l}|-|\rho_l|-d^+_l)}}{(1-(2d-1)\mathrm{e}^{-\sum_{j=1}^m s_j})^2}+(2d-2)\mathrm{e}^{-\sum_{j=1}^m s_j(1+|\sigma_{j,l}|-|\rho_l|)}\\
&+\frac{(2d-2)\mathrm{e}^{-\sum_{j=1}^m s_j(2+|\sigma_{j,l}|-|\rho_l|)}}{1-(2d-1)\mathrm{e}^{-\sum_{j=1}^m s_j}}\\
&+(2d-2)\sum_{n=1}^{-d_l^-} \mathrm{e}^{-\sum_{j=1}^m s_j(|n-|\sigma_{j,l}|+|\rho_l||+n+1)}\\
&+\frac{(2d-2)\mathrm{e}^{-\sum_{j=1}^m s_j(1+|\sigma_{j,l}|-|\rho_l|-2d^-_l)}}{1-\mathrm{e}^{-2\sum_{j=1}^m s_j}}+\frac{(2d-2)^4}{(2d-1)^2}\sum_{n=1}^{-d_l^-} \mathrm{e}^{-\sum_{j=1}^m s_j(|n-|\sigma_{j,l}|+|\rho_l||+n+2)}\\
&+\frac{(2d-2)^2\mathrm{e}^{-\sum_{j=1}^m s_j(2+|\sigma_{j,l}|-|\rho_l|-2d^-_l)}}{(1-\mathrm{e}^{-2\sum_{j=1}^m s_j})(1-(2d-1)\mathrm{e}^{-\sum_{j=1}^m s_j})}.
\end{align*}
Using the partial fraction decompositions 
$(1-w^2)^{-1}=\frac{1}{2}(1-w)^{-1}-\frac{1}{2}(1+w)^{-1}$ and 
$$
\frac{1}{(1-w^2)(1-(2d-1)w)}
=\frac{(2d-1)^2}{((2d-1)^2-1)}\frac{1}{(1-(2d-1)w)}-\frac{1}{(4d-2)}\frac{1}{(1-w)}
+\frac{1}{4d}\frac{1}{(1+w)},
$$
we simplify this expression to
\begin{align*}
&\frac{1}{(1-(2d-1)\mathrm{e}^{-\sum_{j=1}^m s_j})^2}\frac{(2d-2)^2}{2d-1}\mathrm{e}^{-\sum_{j=1}^m s_j(1+|\sigma_{j,l}|-|\rho_l|-d^+_l)}\\
&+\frac{1}{1-(2d-1)\mathrm{e}^{-\sum_{j=1}^m s_j}}\bigg((2d-2)^2\sum_{n=1}^{-d_l^+} (2d-1)^{n-1} \mathrm{e}^{-\sum_{j=1}^m s_j(|n+|\sigma_{j,l}|-|\rho_l||+1)}\\
&\qquad\qquad\qquad\qquad\qquad\qquad\qquad+(2d-2)\mathrm{e}^{-\sum_{j=1}^m s_j(2+|\sigma_{j,l}|-|\rho_l|)}\\
&\qquad\qquad\qquad\qquad\qquad\qquad\qquad+\frac{(2d-2)^2(2d-1)^2}{(2d-1)^2-1}\mathrm{e}^{-\sum_{j=1}^m s_j(2+|\sigma_{j,l}|-|\rho_l|-2d^-_l)}\bigg)\\
&+\frac{1}{1-\mathrm{e}^{-\sum_{j=1}^m s_j}}\bigg((d-1)\mathrm{e}^{-\sum_{j=1}^m s_j(1+|\sigma_{j,l}|-|\rho_l|-2d^-_l)}-\frac{(2d-2)^2}{4d-2}\mathrm{e}^{-\sum_{j=1}^m s_j(2+|\sigma_{j,l}|-|\rho_l|-2d^-_l)}\bigg)\\
&+\frac{1}{1+\mathrm{e}^{-\sum_{j=1}^m s_j}}\bigg((d-1)\mathrm{e}^{-\sum_{j=1}^m s_j(1+|\sigma_{j,l}|-|\rho_l|-2d^-_l)}+\frac{(2d-2)^2}{4d-2}\mathrm{e}^{-\sum_{j=1}^m s_j(2+|\sigma_{j,l}|-|\rho_l|-2d^-_l)}\bigg)\\
&+(2d-2)\mathrm{e}^{-\sum_{j=1}^m s_j(1+|\sigma_{j,l}|-|\rho_l|)}+(2d-2)\sum_{n=1}^{-d_l^-} \mathrm{e}^{-\sum_{j=1}^m s_j(|n-|\sigma_{j,l}|+|\rho_l||+n+1)}\\
&+\frac{(2d-2)^4}{(2d-1)^2}\sum_{n=1}^{-d_l^-} \mathrm{e}^{-\sum_{j=1}^m s_j(|n-|\sigma_{j,l}|+|\rho_l||+n+2)}.
\end{align*}
The reader should note that the expressions in the last two lines are entire functions. In the first four lines, only the pre-factors are not entire functions.

\begin{prop}
\label{firsttermholo}
Let $\pmb{A}$ be as in Subsection \ref{sub:def-props} 
and $t\in \Omega_{\pmb{A}}$. The function 
$$
s=(s_1,\dots,s_m)\mapsto
\sum_{k=\max(1,-n+1)}^\infty \sum_{\kappa(x,n,t)
=k,A_{\rho_{l,|\rho_l|},x_1}=1}
\prod_{j=1}^m\mathrm{e}^{-s_j\big(\big|n+|\sigma_{j,l}|-|\rho_l|\big|+k\big)}
$$
appearing in \eqref{firstterm} is holomorphic outside the 
set $\{s\in \C^m: \sum_{j=1}^m s_j\in \mathcal{P}_0+\pi i \Z\}$ 
where $\mathcal{P}_0:=\{0,\log(2d-1)\}$. The order of singularities on the 
subvariety 
$\{s\in \C^m: \sum_{j=1}^m s_j\in \mathcal{P}_0+\pi i \Z\}$ is at most $2$.
\end{prop}

\subsubsection{The expression in \eqref{secondterm}} 
We now turn to computing \eqref{secondterm}. For this term
we begin with the computation
\begin{align}
\nonumber
\sum_{n=0}^\infty&\sum_{\rho_lx\in \sigma_{\pmb{A}}^{-n}(\{t\})}\mathrm{e}^{-\sum_{j=1}^ms_j(|n+|\sigma_{j,l}|-|\rho_l||)}\\
\nonumber
=&\sum_{n=0}^{|\rho_l|-1}\sum_{\rho_lx\in \sigma_{\pmb{A}}^{-n}(\{t\})}\mathrm{e}^{-\sum_{j=1}^ms_j(|n+|\sigma_{j,l}|-|\rho_l||)}+\sum_{n=|\rho_l|}^\infty\sum_{\rho_lx\in \sigma_{\pmb{A}}^{-n}(\{t\})}\mathrm{e}^{-\sum_{j=1}^ms_j(|n+|\sigma_{j,l}|-|\rho_l||)}\\
\label{simplsec}
=&\sum_{n=0}^{|\rho_l|-1}\delta_{\rho_{l,n+1}\cdots \rho_{l,|\rho_l|},t_{n+1}\cdots t_{|\rho_l|}}\mathrm{e}^{-\sum_{j=1}^ms_j(|n+|\sigma_{j,l}|-|\rho_l||)}+A_{\rho_{|\rho_l|},t_1}\mathrm{e}^{-\sum_{j=1}^m s_j|\sigma_{j,l}|}\\
\nonumber
&+\sum_{n=1}^\infty\#\{\mu\in \mathcal{V}_{\pmb{A}}: |\mu|=n, \; A_{\rho_{|\rho_l|},\mu_1}=A_{\mu_n,t_1}=1\}\mathrm{e}^{-\sum_{j=1}^ms_j(n+|\sigma_{j,l}|)}.
\end{align}
In the last step, we  used the bijection $\mu\mapsto \rho_l\mu \sigma^n(t)$
for fixed $\rho_l,\,n$ to identify the two sets $\{\mu\in \mathcal{V}_{\pmb{A}}: |\mu|=n, \, A_{\rho_{|\rho_l|},\mu_1}=A_{\mu_n,t_1}=1\}$ and $\{x: \rho_lx\in \sigma_{\pmb{A}}^{-(n+|\rho_l|)}(\{t\})\}$. 
A short word counting in $\mathbb{F}_d$ shows that 
$$
\#\{\mu\in \mathcal{V}_{\pmb{A}}: |\mu|=n, \; A_{\rho_{|\rho_l|},\mu_1}=A_{\mu_n,t_1}=1\}=
\begin{cases}
(2d-2)^2(2d-1)^{n-2}, \; &n>1,\\
2d-2+\delta_{\rho_{|\rho_l|},t_1}, \; &n=1.
\end{cases}
$$
Using the geometric series computations of 
Lemma \ref{computationofsums}, we compute \eqref{simplsec} to be
\begin{align}
\nonumber
=&\sum_{n=0}^{|\rho_l|-1}\delta_{\rho_{l,n+1}\cdots \rho_{l,|\rho_l|},t_{n+1}\cdots t_{|\rho_l|}}\mathrm{e}^{-\sum_{j=1}^ms_j(|n+|\sigma_{j,l}|-|\rho_l||)}+A_{\rho_{|\rho_l|},t_1}\mathrm{e}^{-\sum_{j=1}^m s_j|\sigma_{j,l}|}\\
\nonumber
&+(2d-2+\delta_{\rho_{|\rho_l|},t_1})\mathrm{e}^{-\sum_{j=1}^m s_j(1+|\sigma_{j,l}|)}+(2d-2)^2\sum_{n=2}^\infty(2d-1)^{n-2}\mathrm{e}^{-\sum_{j=1}^ms_j(n+|\sigma_{j,l}|)}\\
\nonumber
=&\sum_{n=0}^{|\rho_l|-1}\delta_{\rho_{l,n+1}\cdots \rho_{l,|\rho_l|},t_{n+1}\cdots t_{|\rho_l|}}\mathrm{e}^{-\sum_{j=1}^ms_j(|n+|\sigma_{j,l}|-|\rho_l||)}+A_{\rho_{|\rho_l|},t_1}\mathrm{e}^{-\sum_{j=1}^m s_j|\sigma_{j,l}|}\\
\nonumber
&+(2d-2+\delta_{\rho_{|\rho_l|},t_1})\mathrm{e}^{-\sum_{j=1}^m s_j(1+|\sigma_{j,l}|)}+\frac{(2d-2)^2}{2d-1}\frac{\mathrm{e}^{-\sum_{j=1}^ms_j(2+|\sigma_{j,l}|)}}{1-(2d-1)\mathrm{e}^{-\sum_{j=1}^m s_j}}.
\end{align}
The reader should note that only the last term is not an entire function.

\begin{prop}
\label{secondtermholo}
Let $\pmb{A}$ be as in Subsection \ref{sub:def-props} 
and $t\in \Omega_{\pmb{A}}$. The expression 
$$
\sum_{n=0}^\infty\sum_{\rho_l x\in \sigma_{\pmb{A}}^{-n}(\{t\})}\prod_{j=1}^m\mathrm{e}^{-s_j\big(\big|n+|\sigma_{j,l}|-|\rho_l|\big|\big)}
$$
appearing in \eqref{secondterm} is holomorphic outside the set 
$\{s\in \C^m: \sum_{j=1}^ms_j\in \log(2d-1)+2\pi i\Z\}$. 
In particular, it is holomorphic outside $\mathcal{P}_0+\pi i \Z$ where 
$\mathcal{P}_0:=\{0,\log(2d-1)\}$. The order of singularities on the 
subvariety $\{s\in \C^m: \sum_{j=1}^m s_j\in \log(2d-1)+2\pi i \Z\}$ 
is at most $1$.
\end{prop}

\subsubsection{The expression in \eqref{thirdterm}}
\label{thridridmem}

Finally, we study the term \eqref{thirdterm}). This is
the one place in the computation where we require that we have
chosen a fixed point for the shift map. To be clear:

{\bf We assume that $t$ is a fixed point for $\sigma_{\pmb{A}}$}. 
That is, for some  $\gamma_0\in \{a_1,b_1,\ldots, a_d,b_d\}$ we have $t_j=\gamma_0$ for all $j$.

We first make the simplification
\begin{align*}
\sum_{n=-\infty}^{-1}\sum_{\kappa(\rho_lx,n,t)=-n}\prod_{j=1}^m
\mathrm{e}^{-s_j\big(\big|n+|\sigma_{j,l}|-|\rho_l|\big|+\kappa(\sigma_{j,l}x,n+|\sigma_{j,l}|-|\rho_l|,t)\big)}\\
=\sum_{n=1}^{\infty}\sum_{\kappa(\rho_lx,-n,t)=n}
\mathrm{e}^{-\sum_{j=1}^ms_j\big(\big|n-|\sigma_{j,l}|+|\rho_l|\big|+\kappa(\sigma_{j,l}x,-n+|\sigma_{j,l}|-|\rho_l|,t)\big)}.
\end{align*}
We note that the condition $\kappa(\rho_lx,-n,t)=n$ is equivalent to $\rho_lx=\sigma^n(t)$. In particular, the condition $\kappa(\rho_lx,-n,t)=n$ determines $x$ uniquely from $t$. Moreover, if $\kappa(\rho_lx,-n,t)=n$ then 
$$
\kappa(\sigma_{j,l}x,-n+|\sigma_{j,l}|-|\rho_l|,t)=n-|\sigma_{j,l}|+|\rho_l|,
$$
by Proposition \ref{kappacompprop}.

With the assumption that $t$ is a fixed point, we can write 
\begin{align*}
\sum_{n=1}^{\infty}&\sum_{\kappa(\rho_lx,-n,t)=n}
\mathrm{e}^{-\sum_{j=1}^ms_j(|n-|\sigma_{j,l}|+|\rho_l||+n-|\sigma_{j,l}|+|\rho_l|)}\\
=&\sum_{n=1}^\infty \prod_{i=1}^{|\rho_l|}\delta_{\rho_i,\gamma_0}
\mathrm{e}^{-\sum_{j=1}^ms_j(|n-|\sigma_{j,l}|+|\rho_l||+n-|\sigma_{j,l}|+|\rho_l|)}\\
=&\prod_{i=1}^{|\rho_l|}\delta_{\rho_i,\gamma_0} \sum_{n=1}^{-d^-_l} 
\mathrm{e}^{-\sum_{j=1}^ms_j(|n-|\sigma_{j,l}|+|\rho_l||+n-|\sigma_{j,l}|+|\rho_l|)}\\
&+\prod_{i=1}^{|\rho_l|}\delta_{\rho_i,\gamma_0}\frac{\mathrm{e}^{-2\sum_{j=1}^ms_j(1-|\sigma_{j,l}|+|\rho_l|-d^-_l)}}{1-\mathrm{e}^{-2\sum_{j=1}^ms_j}}\\
=&\frac{1}{{1-\mathrm{e}^{-\sum_{j=1}^ms_j}}}\frac{\prod_{i=1}^{|\rho_l|}\delta_{\rho_i,\gamma_0}}{2}\mathrm{e}^{-2\sum_{j=1}^ms_j(1-|\sigma_{j,l}|+|\rho_l|-d^-_l)}\\
&-\frac{1}{{1+\mathrm{e}^{-\sum_{j=1}^ms_j}}}\frac{\prod_{i=1}^{|\rho_l|}\delta_{\rho_i,\gamma_0}}{2}\mathrm{e}^{-2\sum_{j=1}^ms_j(1-|\sigma_{j,l}|+|\rho_l|-d^-_l)}\\
&+\prod_{i=1}^{|\rho_l|}\delta_{\rho_i,\gamma_0} \sum_{n=1}^{-d^-_l} 
\mathrm{e}^{-\sum_{j=1}^ms_j(|n-|\sigma_{j,l}|+|\rho_l||+n-|\sigma_{j,l}|+|\rho_l|)}.
\end{align*}

\begin{prop}
\label{thirdtermholo}
Let $\pmb{A}$ be as in Subsection \ref{sub:def-props}  
and $t\in \Omega_{\pmb{A}}$ a fixed point. The expression 
$$
\sum_{n=-\infty}^{-1}\sum_{\kappa(\rho_lx,n,t)=-n}\prod_{j=1}^m
\mathrm{e}^{-s_j\big(\big|n+|\sigma_{j,l}|-|\rho_l|\big|+\kappa(\sigma_{j,l}x,n+|\sigma_{j,l}|-|\rho_l|,t)\big)}
$$
appearing in \eqref{thirdterm} is holomorphic 
outside the set 
$\{s\in \C^m: \sum_{j=1}^ms_j\in \log(2d-1)+\pi i\Z\}$. In particular, 
it is holomorphic outside $\mathcal{P}_0+\pi i \Z$ where 
$\mathcal{P}_0:=\{0,\log(2d-1)\}$. The order of singularities 
on the subvariety 
$\{s\in \C^m: \sum_{j=1}^m s_j\in \log(2d-1)+\pi i \Z\}$ are at most $1$.
\end{prop}

\subsubsection{Summarizing the computation} 
\label{sub:sanity}
Let us summarize the computations of the last few pages. We pick a fixed point $t\in \Omega_{\pmb{A}}=\partial \mathbb{F}_d$. We have some words words $\mu_1,\nu_1,\mu_2,\nu_2\ldots,\mu_m,\nu_m\in \mathcal{V}_{\pmb{A}}$ fixed and either the relevant heat trace vanishes or we can pick finite words $\rho_1,\ldots, \rho_q$ and coefficients $c_1,\ldots, c_q\in \Z$ as in Proposition \ref{dickotomy}. For simplicity, we write $\vec{\mu}:=(\mu_1,\ldots,\mu_m)$, $\vec{\nu}:=(\nu_1,\ldots, \nu_m)$ and $\vec{\rho}:=(\rho_1,\ldots, \rho_q)$. We also write $\omega_j:=\sum_{i=j+1}^m |\mu_i|-|\nu_i|=|\sigma_{j,l}|-|\rho_l|$. In all else, we follow the notations above.

There is an entire function $f_{\vec{\mu},\vec{\nu},\vec{\rho}}$ on $\C^m$ such that 
\begin{align*}
\mathrm{Tr}(S_{\mu_1}S_{\nu_1}^*&\mathrm{e}^{-s_1|\D_t|}\cdots S_{\mu_m}S_{\nu_m}^*\mathrm{e}^{-s_m|\D_t|})\\
=&\frac{1}{(1-(2d-1)\mathrm{e}^{-\sum_{j=1}^m s_j})^2}\sum_{l=1}^qc_l\frac{(2d-2)^2}{2d-1}\mathrm{e}^{-\sum_{j=1}^m s_j(1+\omega_j-d^+_l)}\\
&+\frac{1}{1-(2d-1)\mathrm{e}^{-\sum_{j=1}^m s_j}}\sum_{l=1}^qc_l\bigg((2d-2)^2\sum_{n=1}^{-d_l^+} (2d-1)^{n-1} \mathrm{e}^{-\sum_{j=1}^m s_j(|n+\omega_j|+1)}\\
&\qquad\qquad\qquad\qquad\qquad\qquad\qquad+(2d-2)\mathrm{e}^{-\sum_{j=1}^m s_j(2+\omega_j)}\\
&\qquad\qquad\qquad\qquad\qquad\qquad\qquad+\frac{(2d-2)^2(2d-1)^2}{(2d-1)^2-1}\mathrm{e}^{-\sum_{j=1}^m s_j(2+\omega_j-2d^-_l)}\\
&\qquad\qquad\qquad\qquad\qquad\qquad\qquad+\frac{(2d-2)^2}{2d-1}\mathrm{e}^{-\sum_{j=1}^ms_j(2+|\sigma_{j,l}|)}\bigg)\\
&+\frac{1}{1-\mathrm{e}^{-\sum_{j=1}^m s_j}}\sum_{l=1}^qc_l\bigg((d-1)\mathrm{e}^{-\sum_{j=1}^m s_j(1+\omega_j-2d^-_l)}\\
&\qquad\qquad\qquad\qquad\qquad\qquad\qquad-\frac{(2d-2)^2}{4d-2}\mathrm{e}^{-\sum_{j=1}^m s_j(2+\omega_j-2d^-_l)}\\
&\qquad\qquad\qquad\qquad\qquad\qquad\qquad+\frac{\prod_{i=1}^{|\rho_l|}\delta_{\rho_i,\gamma_0}}{2}\mathrm{e}^{-2\sum_{j=1}^ms_j(1-\omega_j-d^-_l)}\bigg)\\
&+\frac{1}{1+\mathrm{e}^{-\sum_{j=1}^m s_j}}\sum_{l=1}^qc_l\bigg((d-1)\mathrm{e}^{-\sum_{j=1}^m s_j(1+\omega_j-2d^-_l)}\\
&\qquad\qquad\qquad\qquad\qquad\qquad\qquad+\frac{(2d-2)^2}{4d-2}\mathrm{e}^{-\sum_{j=1}^m s_j(2+\omega_j-2d^-_l)}\\
&\qquad\qquad\qquad\qquad\qquad\qquad\qquad-\frac{\prod_{i=1}^{|\rho_l|}\delta_{\rho_i,\gamma_0}}{2}\mathrm{e}^{-2\sum_{j=1}^ms_j(1-\omega_j-d^-_l)}\bigg)\\
&+f_{\vec{\mu},\vec{\nu},\vec{\rho}}(s_1,s_2,\ldots,s_m).
\end{align*}

In particular, we note that
\begin{align*}
\mathrm{Tr}(\sigma^{l_1}(S_{\mu_1}&S_{\nu_1}^*)\sigma^{l_2}(S_{\mu_1}S_{\nu_1}^*)\cdots \sigma^{l_m}(S_{\mu_m}S_{\nu_m}^*)\mathrm{e}^{-s|\D_t|})\\
=&\mathrm{Tr}(S_{\mu_1}S_{\nu_1}^*\mathrm{e}^{-(l_1-l_2)|\D_t|}S_{\mu_2}S_{\nu_2}^*\mathrm{e}^{-(l_2-l_3)|\D_t|}\cdots \mathrm{e}^{-(l_{m-1}-l_m)|\D_t|}S_{\mu_m}S_{\nu_m}^*\mathrm{e}^{-(s+l_m-l_1)|\D_t|})\\
=&\frac{1}{(1-(2d-1)\mathrm{e}^{-s})^2}\frac{(2d-2)^2}{2d-1}\sum_{l=1}^qc_l\mathrm{e}^{-\sum_{j=1}^{m} (l_j-l_{j+1})(1+\omega_j-d^+_l)-s}\\
&+\frac{1}{1-(2d-1)\mathrm{e}^{-s}}\sum_{l=1}^qc_l\bigg((2d-2)^2\sum_{n=1}^{-d_l^+} (2d-1)^{n-1} \mathrm{e}^{-\sum_{j=1}^{m} (l_j-l_{j+1})(|n+\omega_j|+1)-s(|n|+1)}\\
&\qquad\qquad\qquad\qquad\qquad+(2d-2)\mathrm{e}^{-\sum_{j=1}^{m} (l_j-l_{j+1})(2+o_j)-2s}\\
&\qquad\qquad\qquad\qquad\qquad+\frac{(2d-2)^2(2d-1)^2}{(2d-1)^2-1}\mathrm{e}^{-\sum_{j=1}^{m} (l_j-l_{j+1})(2+\omega_j-2d^-_l)-2s}\\
&\qquad\qquad\qquad\qquad\qquad+\frac{(2d-2)^2}{2d-1}\mathrm{e}^{-\sum_{j=1}^{m} (l_j-l_{j+1})(2+|\sigma_{j,l}|)-s(2+|\rho_l|)}\bigg)\\
&+\frac{1}{1-\mathrm{e}^{-s}}\sum_{l=1}^qc_l\bigg((d-1)\mathrm{e}^{-\sum_{j=1}^{m} (l_j-l_{j+1})(1+\omega_j-2d^-_l)-s}\\
&\qquad\qquad\qquad\qquad-\frac{(2d-2)^2}{4d-2}\mathrm{e}^{-\sum_{j=1}^{m} (l_j-l_{j+1})(2+\omega_j-2d^-_l)-2s}\\
&\qquad\qquad\qquad\qquad+\frac{\prod_{i=1}^{|\rho_l|}\delta_{\rho_i,\gamma_0}}{2}\mathrm{e}^{-2\sum_{j=1}^{m} (l_j-l_{j+1})(1-\omega_j-d^-_l)-2s}\bigg)\\
&+\frac{1}{1+\mathrm{e}^{-s}}\sum_{l=1}^qc_l\bigg((d-1)\mathrm{e}^{-\sum_{j=1}^{m} (l_j-l_{j+1})(1+\omega_j-2d^-_l)-s}\\
&\qquad\qquad\qquad+\frac{(2d-2)^2}{4d-2}\mathrm{e}^{-\sum_{j=1}^{m} (l_j-l_{j+1})(2+\omega_j-2d^-_l)-2s}\\
&\qquad\qquad\qquad\qquad-\frac{\prod_{i=1}^{|\rho_l|}\delta_{\rho_i,\gamma_0}}{2}\mathrm{e}^{-2\sum_{j=1}^{m} (l_j-l_{j+1})(1-\omega_j-d^-_l)-2s}\bigg)\\
&+f_{\vec{\mu},\vec{\nu},\vec{\rho}}(l_1-l_2,l_2-l_3,\ldots,l_{m-1}-l_m,s+l_m-l_1),
\end{align*}
for a entire function $f_{\vec{\mu},\vec{\nu},\vec{\rho}}$ on $\C^m$. This computation shows that 
$$
s\mapsto \mathrm{Tr}(\alpha^{l_1}(S_{\mu_1}S_{\nu_1}^*)\alpha^{l_2}(S_{\mu_1}S_{\nu_1}^*)\cdots \alpha^{l_m}(S_{\mu_m}S_{\nu_m}^*)\mathrm{e}^{-s|\D_t|}),
$$ 
has a meromorphic extension to $\C$ with poles 
of order at most $2$ in $\{0,\log(2d-1)\}+\pi i \Z$.

\subsection{Heat trace computations of Toeplitz operators for the free group}
\label{toeplsuzise}

In this subsection we will prove that Assumption 2) of Proposition \ref{singprop} is fulfilled for the action of the free group on its Gromov boundary, see Subsection \ref{sub:def-props}. 

We fix finite words $\mu_1,\nu_1,\mu_2,\nu_2\ldots,\mu_m,\nu_m\in \mathcal{V}_{\pmb{A}}$. Let us simplify the heat trace of Toeplitz operators $\mathrm{Tr}(PS_{\mu_1}S_{\nu_1}^*P\mathrm{e}^{-s_1|\D_t|}\cdots PS_{\mu_m}S_{\nu_m}^*P\mathrm{e}^{-s_m|\D_t|})$. The image of $P$ is spanned by the 
orthonormal basis
$$\{\delta_{(x,n)}: (x,n)\in \mathcal{V}_t, \; \kappa(x,n,t)=0\}.$$
In other words, an orthonormal 
basis of the image of $P$ is given by 
$(\delta_{(\mu t,|\mu|)})_{\mu \in \mathcal{V}_{\pmb{A},t_1}}$, where 
$$
\mathcal{V}_{\pmb{A},t_1}:=\{\mu\in \mathcal{V}_{\pmb{A}}:|\mu|=0\mbox{   or   } A_{\mu_{|\mu|},t_1}=1\}.
$$
Define the isometry $W_t:\ell^2( \mathcal{V}_{\pmb{A},t_1})\to \ell^2(\mathcal{V}_t)$ by $W_t\delta_\mu:=\delta_{(\mu t,|\mu|)}$. By construction, $P$ coincides with the range projection of $W_t$. 

For $j\in \{1,\ldots, N\}$, we define the operators $T_j:\ell^2( \mathcal{V}_{\pmb{A},t_1})\to \ell^2( \mathcal{V}_{\pmb{A},t_1})$ as $T_j:=W_t^*S_jW_t$. In terms of the orthonormal basis, 
$$
T_j\delta_{\mu}=\begin{cases} 
\delta_{j\mu}, \; &\mbox{if  }j\mu \in \mathcal{V}_{\pmb{A},t_1},\\
0,\; &\mbox{otherwise}.\end{cases}
$$
For a finite word $\mu=\mu_1\cdots \mu_k$ we write $T_\mu:=T_{\mu_1}\cdots T_{\mu_k}$. A short computation shows that for $\mu,\nu\in \mathcal{V}_{\pmb{A}}$, the operator
$$
T_\mu T_\nu^*-W_t^*S_\mu S_\nu^* W_t\in C_c( \mathcal{V}_{\pmb{A},t_1}\times  \mathcal{V}_{\pmb{A},t_1})
$$ 
and so is a finite rank operator in our orthonormal
basis $(\delta_\mu)_{\mu \in  \mathcal{V}_{\pmb{A},t_1}}$. 

One more ingredient is needed in this soup. Define the number operator $\mathfrak{N}$ densely on $\ell^2( \mathcal{V}_{\pmb{A},t_1})$ as the self-adjoint operator satisfying 
$$\mathfrak{N}\delta_\mu:=|\mu|\delta_\mu.$$
It is immediate that $\mathfrak{N}=W_t^*\D_tW_t$.

Combining the facts of the two previous paragraphs together, we see that the difference 
$$\mathrm{Tr}(PS_{\mu_1}S_{\nu_1}^*P\mathrm{e}^{-s_1|\D_t|}\cdots PS_{\mu_m}S_{\nu_m}^*P\mathrm{e}^{-s_m|\D_t|})-\mathrm{Tr}(T_{\mu_1}T_{\nu_1}^*\mathrm{e}^{-s_1\mathfrak{N}}\cdots T_{\mu_m}T_{\nu_m}^*\mathrm{e}^{-s_m\mathfrak{N}}),$$
extends to an entire function on $\C^m$. Compute $\mathrm{Tr}(T_{\mu_1}T_{\nu_1}^*\mathrm{e}^{-s_1\mathfrak{N}}\cdots T_{\mu_m}T_{\nu_m}^*\mathrm{e}^{-s_m\mathfrak{N}})$ yields
\begin{align*}
\mathrm{Tr}(T_{\mu_1}T_{\nu_1}^*&\mathrm{e}^{-s_1\mathfrak{N}}\cdots T_{\mu_m}T_{\nu_m}^*\mathrm{e}^{-s_m\mathfrak{N}})\\
&= \sum_{\mu \in  \mathcal{V}_{\pmb{A},t_1}}\langle\delta_{\mu}|T_{\mu_1}T_{\nu_1}^*\cdots T_{\mu_m}T_{\nu_m}^*\delta_{\mu}\rangle\mathrm{e}^{-\sum_{j=1}^m s_j(|\mu|+\sum_{i=j+1}^m |\mu_i|-|\nu_i|)}.
\end{align*}
Note that whenever $\langle\delta_{\mu}|T_{\mu_1}T_{\nu_1}^*\cdots T_{\mu_m}T_{\nu_m}^*\delta_{\mu}\rangle$ is non-zero, $|\mu|+\sum_{i=j+1}^m |\mu_i|-|\nu_i|\geq 0$. In fact, we can reduce the structure of the words $\mu_1,\nu_1,\mu_2,\nu_2\ldots,\mu_m,\nu_m\in \mathcal{V}_{\pmb{A}}$ in a similar fashion as in Proposition \ref{dickotomy}.

\begin{prop}
\label{dickotomytoep}
Let $\mu_1,\nu_1,\mu_2,\nu_2\ldots,\mu_m,\nu_m\in \mathcal{V}_{\pmb{A}}$ be finite words and $t\in \Omega_{\pmb{A}}$. Exactly one of the following two statements holds:
\begin{enumerate}
\item There is a finite rank operator $K\in C_c(\mathcal{V}_{\pmb{A},t_1}\times \mathcal{V}_{\pmb{A},t_1})$, finite words $\rho_1,\ldots, \rho_q$ and coefficients $c_1,\ldots, c_q\in \Z$ such that 
$$T_{\mu_1}T_{\nu_1}^*\cdots T_{\mu_m}T_{\nu_m}^*\delta_\mu=\sum_{l=1}^q c_l\chi_{C_{\rho_l}}(\mu t)+K\delta_\mu,\quad\forall\mu\in \mathcal{V}_{\pmb{A},t_1}.$$
\item The sequence $\big(\langle\delta_{\mu}|T_{\mu_1}T_{\nu_1}^*\cdots T_{\mu_m}T_{\nu_m}^*\delta_{\mu}\rangle\big)_{\mu\in \mathcal{V}_{\pmb{A},t_1}}$ is finitely supported.
\end{enumerate}
\end{prop}

We see that either $\mathrm{Tr}(T_{\mu_1}T_{\nu_1}^*\mathrm{e}^{-s_1\mathfrak{N}}\cdots T_{\mu_m}T_{\nu_m}^*\mathrm{e}^{-s_m\mathfrak{N}})$ is an entire function, or there are finite words $\rho_1,\ldots, \rho_q$ and coefficients $c_1,\ldots, c_q\in \Z$ such that it differs from the expression 
$$\sum_{l=1}^q \sum_{n=0}^\infty\sum_{\mu\in \mathcal{V}_{\pmb{A},t_1},|\mu|=n}  c_l\chi_{C_{\rho_l}}(\mu t)\mathrm{e}^{-\sum_{j=1}^m s_j(n+\sum_{i=j+1}^m |\mu_i|-|\nu_i|)},$$
by an entire function. The last expression can be computed as
\begin{align*}
\sum_{l=1}^q &\sum_{n=0}^\infty\sum_{\mu\in \mathcal{V}_{\pmb{A},t_1},|\mu|=n}  c_l\chi_{C_{\rho_l}}(\mu t)\mathrm{e}^{-\sum_{j=1}^m s_j(n+\sum_{i=j+1}^m |\mu_i|-|\nu_i|)}\\
=&\sum_{l=1}^q \sum_{n=0}^{|\rho_l|+1}\sum_{\mu\in \mathcal{V}_{\pmb{A},t_1},|\mu|=n}  c_l\chi_{C_{\rho_l}}(\mu t)\mathrm{e}^{-\sum_{j=1}^m s_j(n+\sum_{i=j+1}^m |\mu_i|-|\nu_i|)}\\
&+\sum_{l=1}^q \sum_{n=|\rho_l|+2}^\infty  c_l\#\{|\mu|=n:\; \mu_1\cdots \mu_{|\rho_l|}=\rho_l, \; A_{\mu_n,t_1}=1\}\mathrm{e}^{-\sum_{j=1}^m s_j(n+\sum_{i=j+1}^m |\mu_i|-|\nu_i|)}\\
=&\sum_{l=1}^q \sum_{n=0}^{|\rho_l|+1}\sum_{\mu\in \mathcal{V}_{\pmb{A},t_1},|\mu|=n}  c_l\chi_{C_{\rho_l}}(\mu t)\mathrm{e}^{-\sum_{j=1}^m s_j(n+\sum_{i=j+1}^m |\mu_i|-|\nu_i|)}\\
&+\sum_{l=1}^q \sum_{n=|\rho_l|+2}^\infty  c_l(2d-2)^2(2d-1)^{n-|\rho_l|-2}\mathrm{e}^{-\sum_{j=1}^m s_j(n+\sum_{i=j+1}^m |\mu_i|-|\nu_i|)}\\
=&\sum_{l=1}^q \sum_{n=0}^{|\rho_l|+1}\sum_{\mu\in \mathcal{V}_{\pmb{A},t_1},|\mu|=n}  c_l\chi_{C_{\rho_l}}(\mu t)\mathrm{e}^{-\sum_{j=1}^m s_j(n+\sum_{i=j+1}^m |\mu_i|-|\nu_i|)}\\
&+\sum_{l=1}^q \sum_{n=0}^\infty  c_l(2d-2)^2(2d-1)^{n}\mathrm{e}^{-\sum_{j=1}^m s_j(n+|\rho_l|+2+\sum_{i=j+1}^m |\mu_i|-|\nu_i|)}\\
=&\sum_{l=1}^q \sum_{n=0}^{|\rho_l|+1}\sum_{\mu\in \mathcal{V}_{\pmb{A},t_1},|\mu|=n}  c_l\chi_{C_{\rho_l}}(\mu t)\mathrm{e}^{-\sum_{j=1}^m s_j(|\mu|+\sum_{i=j+1}^m |\mu_i|-|\nu_i|)}\\
&+(2d-2)^2\sum_{l=1}^q c_l\frac{(2d-1)\mathrm{e}^{-\sum_{j=1}^m s_j(|\rho_l|+2+\sum_{i=j+1}^m |\mu_i|-|\nu_i|)}}{1-(2d-1)\mathrm{e}^{-\sum_{j=1}^m s_j}}.\\
\end{align*}
The first term is entire.

\begin{prop}
\label{toeplitztermholo}
Let $\pmb{A}$ be as in Subsection \ref{sub:def-props} and $t\in \Omega_{\pmb{A}}$ a fixed point. The function 
$$
(s_1,s_2,\ldots, s_m)\mapsto \mathrm{Tr}(PS_{\mu_1}S_{\nu_1}^*P\mathrm{e}^{-s_1|\D_t|}\cdots PS_{\mu_m}S_{\nu_m}^*P\mathrm{e}^{-s_m|\D_t|})
$$
extends to a meromorphic function on $\C^m$ 
which is holomorphic outside the set 
$\{s\in \C^m: \sum_{j=1}^ms_j\in \log(2d-1)+2\pi i\Z\}$. 
In particular, it is holomorphic outside $\mathcal{P}_0+\pi i \Z$ where 
$\mathcal{P}_0:=\{0,\log(2d-1)\}$. The order of singularities 
on the subvariety $\{s\in \C^m: \sum_{j=1}^m s_j\in \log(2d-1)+\pi i \Z\}$ 
are at most $1$.
\end{prop}

\end{appendix}

\end{document}